\documentclass[a4paper, 11pt]{amsart}

\usepackage{amsfonts}
\usepackage{amsthm}
\usepackage{amsmath, amssymb}
\usepackage{mathtools} 
\usepackage{enumerate}
\usepackage[all]{xy}
\usepackage{graphicx}
\usepackage{color}
\usepackage[percent]{overpic}
\usepackage{relsize} 
\usepackage[colorlinks,linkcolor=blue,citecolor=blue]{hyperref}
\usepackage{thmtools, thm-restate}
\usepackage[percent]{overpic}

\usepackage[top=2.5cm, bottom=2.5cm, left=2.5cm, right=2.5cm]{geometry}
\numberwithin{equation}{section}
\newtheorem{thm}{Theorem}[section]
\newtheorem{mydef}[thm]{Definition}
\newtheorem{prop}[thm]{Proposition}
\newtheorem{corol}[thm]{Corollary}
\newtheorem{lemma}[thm]{Lemma}
\newtheorem{eg}[thm]{Example}

\newtheorem{remark}[thm]{Remark}
\newtheorem*{claim}{Claim}
\newtheorem{conjecture}[thm]{Conjecture}

\newcommand{\Z}{\mathbb{Z}}

\newcommand{\Q}{\mathbb{Q}}

\newcommand{\mubar}{\overline{\mu}}

\DeclareMathOperator{\lcm}{lcm}
\DeclareMathOperator{\im}{im}
\DeclareMathOperator{\Hom}{Hom}
\DeclareMathOperator{\inter}{int}
\DeclareMathOperator{\tor}{tor}
\DeclareMathOperator{\coker}{coker}

\newcommand{\spin}{\ifmmode{\rm Spin}\else{${\rm spin}$\ }\fi}
\newcommand{\spinc}{\ifmmode{{\rm Spin}^c}\else{${\rm spin}^c$}\fi}

\newcommand{\spincs}{\mathfrak s}

\newcommand{\norm}[1]{\lVert #1 \rVert^2}
\newcommand{\wt}[1]{\widetilde{#1}}

\begin{document}

\title{Smoothly embedding Seifert fibered spaces in $S^4$}

\author{Ahmad Issa}
\address{Department of Mathematics \\
         The University of Texas at Austin \\
         Austin, TX, 78712, USA}
\email{aissa@math.utexas.edu}

\author{Duncan McCoy}
\address{Department of Mathematics \\
         The University of Texas at Austin \\
         Austin, TX, 78712, USA}
\email{d.mccoy@math.utexas.edu}

\begin{abstract} 
  Using an obstruction based on Donaldson's theorem, we derive strong restrictions on when a Seifert fibered space $Y = F(e; \frac{p_1}{q_1}, \ldots, \frac{p_k}{q_k})$ over an orientable base surface $F$ can smoothly embed in $S^4$. 
This allows us to classify precisely when $Y$ smoothly embeds provided $e > k/2$, where $e$ is the normalized central weight and $k$ is the number of singular fibers. Based on these results and an analysis of the Neumann-Siebenmann invariant $\overline{\mu}$, we make some conjectures concerning Seifert fibered spaces which embed in $S^4$. Finally, we also provide some applications to doubly slice Montesinos links, including a classification of the smoothly doubly slice odd pretzel knots up to mutation.
\end{abstract}

\maketitle

\section{Introduction}
It is known that every closed $3$-manifold smoothly embeds in $S^5$ \cite{MR0184246, MR0175139,MR0133136}. However, the question of which closed $3$-manifolds embed in $S^4$ is far more subtle. Not every 3-manifold embeds in $S^4$ and, in fact, the existence of embeddings often depends on whether one is working in the smooth or topological category. 
The question of which closed orientable 3-manifolds embed in $S^4$ appears as Problem~3.20 on Kirby's list. Over the years many different techniques and obstructions have been developed to address the question. For example, Hantzsche \cite{MR1545714} proved that if $Y$ embeds in $S^4$ then the torsion part of $H_1(Y)$ must split as a direct double, that is, $\tor H_1(Y) \cong G \oplus G$ for some abelian group $G$. There have also been applications of topological obstructions based on linking forms \cite{MR2538589}, Casson-Gordon signatures \cite{MR710099} and the $G$-index theorem \cite{MR1620508}, as well as smooth obstructions based on Rokhlin's theorem, the Neumann-Siebenman invariant, Furuta's $10/8$ theorem, Donaldson's theorem and the Ozsv\'{a}th-Szab\'{o} $d$-invariants, see e.g. \cite{08102346} and \cite{MR3271270}. For a nice introduction to the subject of embedding $3$-manifolds in $S^4$ see \cite{08102346}.

In this paper we study the question of which Seifert fibered spaces over an orientable base surface smoothly embed in $S^4$. 
We use $Y = F(e; \frac{p_1}{q_1}, \ldots, \frac{p_k}{q_k})$ to denote the Seifert fibered space over orientable base surface $F$ which is obtained by surgery as in Figure~\ref{fig:sfs_as_surgery}. After possibly changing orientation, $Y$ may be assumed to be in \emph{standard form}, where $\frac{p_i}{q_i} > 1$ for all $i$ and with non-negative generalized Euler invariant $\varepsilon(Y) :=  e - \sum_{i=1}^k \frac{q_i}{p_i} \ge 0$.\footnote{With these conventions the Poincar{\'e} homology sphere oriented to bound the positive definite $E_8$ plumbing is $S^2(2; 2, \frac{3}{2}, \frac{5}{4})$.} 

By using an obstruction based on Donaldson's theorem \cite{MR910015}, we show that if $Y$ embeds smoothly in $S^4$ then $e\leq \frac{k+1}{2}$ and classify precisely which embed when $e = \frac{k+1}{2}$.

\begin{restatable}{thm}{sphereembthmhalfbound}
  \label{thm:classification_e_ge_(k+1)/2} Let $Y=F(e; \frac{p_1}{q_1}, \ldots, \frac{p_k}{q_k})$ be a Seifert fibered space over orientable base surface $F$ with $\varepsilon(Y) > 0$ and $\frac{p_i}{q_i} > 1$ for all $i$. If $Y$ embeds smoothly in $S^4$, then $e\leq \frac{k+1}{2}$. Moreover, if $e = \frac{k+1}{2}$ then $Y$ smoothly embeds in $S^4$ if and only if $Y$ takes the form
  $$Y = F \left( e; \frac{a}{a-1}, \left\lbrace a, \frac{a}{a-1}\right\rbrace^{\times (e-1)}\right) = F \left(\frac{k+1}{2}; \frac{a}{a-1}, a, \frac{a}{a-1}, a, \ldots, \frac{a}{a-1}\right)$$
  where $e\geq 1$ and $a\geq 2$ is an integer.
  \end{restatable}

This upper bound is one example of the difference between smooth and topological embeddings. The optimal upper bound for topological embeddings is $e\leq k-1$ (see Proposition~\ref{prop:top_bound}).

Classifying which Seifert fibered spaces embed smoothly in $S^4$ becomes increasingly difficult as $e$ decreases relative to $k$. For $e=\frac{k}{2}$, we are able to obtain a partial classification. 

\begin{restatable}{thm}{sphereembthm}
  \label{thm:classification_e_ge_k2} Let $Y=F(\frac{k}{2}; \frac{p_1}{q_1}, \ldots, \frac{p_k}{q_k})$ be a Seifert fibered space over orientable base surface $F$ with $\frac{p_i}{q_i} > 1$ for all $i$, $k$ even and $\varepsilon(Y) > 0$. If $Y$ smoothly embeds in $S^4$ then there exist positive integers $p,q,r,s$ with $\frac{p}{q}, \frac{r}{s} > 1$, $(p,q) = (r,s) = 1$ and $\frac{s}{r} + \frac{q}{p} = 1 - \frac{1}{pr}$ such that the following holds.
  \begin{enumerate}
  \item\label{classif:case2}
    $$Y = F\left(\frac{k}{2}; \frac{p}{q}, \frac{r}{s}, \left\lbrace\frac{p}{p-q}, \frac{p}{q}\right\rbrace^{\ge 0}, \left\lbrace\frac{r}{r-s}, \frac{r}{s}\right\rbrace^{\ge 0}\right),\mbox{ or}$$
  \item\label{classif:case3}
    $$Y = F\left(\frac{k}{2}; \frac{p}{q}, \frac{r}{s}, \left\lbrace pr, \frac{pr}{pr-1}\right\rbrace^{\ge 1}, \left\lbrace\frac{p}{q}, \frac{p}{p-q}\right\rbrace^{\ge 0}, \left\lbrace \frac{r}{s}, \frac{r}{r-s} \right\rbrace^{\ge 0}\right).$$
  \end{enumerate}
  Moreover, in case \eqref{classif:case2} $Y$ embeds smoothly in $S^4$. Here the notation $\{\frac{a}{b}, \frac{a}{a-b}\}^{\ge m}$ means that there are at least $m$ pairs of fractions of this form.
\end{restatable}

Both Theorem~\ref{thm:classification_e_ge_(k+1)/2} and Theorem~\ref{thm:classification_e_ge_k2} are derived from a more general result, Theorem~\ref{thm:partitions} below, stating that a Seifert fibered space which smoothly embeds in $S^4$ must satisfy a strong condition which we call \emph{partitionable}.

\begin{mydef}\label{def:partitionable} Let $Y = F(e; \frac{p_1}{q_1}, \ldots, \frac{p_k}{q_k})$ be a Seifert fibered space over orientable base surface $F$ with $\varepsilon(Y) > 0$ and $\frac{p_i}{q_i} > 1$ for all $i$. We say that $Y$ is \emph{partitionable} if $\tor H_1(Y) \cong G \oplus G$ for some finite abelian group $G$, and there exist partitions $P_1$ and $P_2$ of $\{1,\ldots,k\}$, each into precisely $e$ classes, such that the following hold. For each partition $P \in \{P_1, P_2\}$: 
  \begin{enumerate}[(a)]
  \item\label{cond:noncomp} There exists a unique class $C \in P$ such that $\mathlarger{\sum_{j \in C} \frac{q_j}{p_j} = 1 - \frac{1}{\lcm(p_1,\ldots,p_k)}}$.
    \item\label{cond:ineq} For each other class $C' \in P$, $\mathlarger{\sum_{j \in C'} \frac{q_j}{p_j} = 1}$. 
    \item\label{cond:comp} No non-empty union of a proper subset of classes in $P_1$ is equal to a union of classes in $P_2$.
  \end{enumerate}
The classes satisfying condition \eqref{cond:ineq} are said to be \emph{complementary}.
\end{mydef}
With this definition in place, the general obstruction we derive from Donaldson's theorem can be stated as the following.
\begin{restatable}{thm}{thmcomplementarypartitions}
  \label{thm:partitions} Let $Y = F(e; \frac{p_1}{q_1}, \ldots, \frac{p_k}{q_k})$ with $F$ an orientable surface, $\frac{p_i}{q_i} > 1$ for all $i$, and $\varepsilon(Y) > 0$. If $Y$ smoothly embeds in $S^4$ then $Y$ is partitionable.
\end{restatable}

In this paper we focus on Seifert fibered spaces over orientable base surface and with non-zero generalized Euler invariant, i.e. $\varepsilon \neq 0$. We point out that there are already relatively strong results known when $\varepsilon=0$ or the base surface is non-orientable.
For orientable base surface and $\varepsilon=0$, Donald showed that $Y$ can smoothly embed in $S^4$ only if it can be written in the form $Y=F(0;\frac{p_1}{q_1}, -\frac{p_1}{q_1}, \dots, \frac{p_k}{q_k}, -\frac{p_k}{q_k})$ \cite[Theorem~1.3]{MR3271270}, see also \cite{MR2538589}. Donald also obtained similar results when the base surface is non-orientable \cite[Theorem~1.2]{MR3271270} and further results in the non-orientable case can be found in \cite{MR1620508}.

In the course of applying Theorem~\ref{thm:partitions} it becomes natural to define an operation on Seifert fibered spaces, which we call {\em expansion}.

\begin{mydef}\label{def:expansion} 
Let $Y=F(e;\frac{p_1}{q_1}, \dots, \frac{p_k}{q_k})$ be a Seifert fibered space with $k\geq 1$. The Seifert fibered space $Y'$ is obtained from $Y$ by {\em expansion} if it takes the form
  \[Y'=F\left(e;\frac{p_1}{q_1}, \dots, \frac{p_k}{q_k},-\frac{p_j}{q_j},\frac{p_j}{q_j}\right) = F\left(e+1;\frac{p_1}{q_1}, \dots, \frac{p_k}{q_k},\frac{p_j}{p_j - q_j},\frac{p_j}{q_j}\right),\]
 for some $j$ in the range $1\leq j\leq k$.
\end{mydef}
With this definition, notice that the spaces in Theorem~\ref{thm:classification_e_ge_(k+1)/2} and Theorem~\ref{thm:classification_e_ge_k2} are precisely those obtained by a sequence of expansions from spaces of the form $F(1;\frac{a}{a-1})$, $F(1;\frac{p}{q},\frac{r}{s})$, or $F(2; \frac{p}{q},\frac{r}{s}, pr, \frac{pr}{pr-1})$. In fact, we prove these results by showing that whenever $e$ is large relative to $k$, any space which is partitionable is obtained by expansion from some other space which is also partitionable.

In the opposite direction, the notion of expansion also proves to be useful for constructing embeddings into $S^4$.
\begin{restatable}{lemma}{lemaddfibers}
  \label{lem:add_fibers}
  If $Y'$ is obtained by expansion from $Y$, then $Y'$ smoothly embeds in $Y\times [0,1]$. In particular, if $Y$ embeds smoothly in $S^4$, then so does $Y'$.
\end{restatable}
This easily shows that the Seifert fibered spaces in Theorem~\ref{thm:classification_e_ge_(k+1)/2} and Theorem~\ref{thm:classification_e_ge_k2}\eqref{classif:case2} smoothly embed in $S^4$. Since Seifert fibered spaces of the form $S^2(1;\frac{a}{a-1})$ and $S^2(1;\frac{p}{q}, \frac{r}{s})$, where $\frac{s}{r} + \frac{q}{p} = 1 - \frac{1}{pr}$ and $a > 1$ is an integer, are homeomorphic to $S^3$, they embed in $S^4$. When the base surface is $S^2$ the spaces we wish to embed are precisely those obtained by expansion from these descriptions of $S^3$, so their embeddings can be constructed via Lemma~\ref{lem:add_fibers}. The higher genus base surface case follows from this case by a result of Crisp-Hillman \cite[Lemma 3.2]{MR1620508}, see Proposition~\ref{prop:genus_bump}.

The family of Seifert fibered spaces in Theorem~\ref{thm:classification_e_ge_k2}\eqref{classif:case3} which we are unable to completely resolve arises when $Y$ is partitionable with a partition containing a complementary class of size three. When the base surface is $F = S^2$, we have further tools at our disposal, namely the Neumann-Siebenmann invariant $\mubar$. An analysis of this invariant gives further restrictions.

\begin{prop}\label{prop:k2_even_cond} In Theorem~ \ref{thm:classification_e_ge_k2}\eqref{classif:case3} with $F = S^2$, the space $Y$ can smoothly embed in $S^4$ only if $p$ and $r$ are both odd.
\end{prop}

We conjecture that for $Y$ to smoothly embed, not only must $Y$ be partitionable as in Theorem~\ref{thm:partitions}, but that each complementary class in the partitions must have size two. This would rule out the spaces in Theorem~\ref{thm:classification_e_ge_k2}\eqref{classif:case3} from embedding. It would also imply that if $e > 1$ and $Y$ smoothly embeds in $S^4$, then $Y$ is necessarily an expansion of a partitionable space (see Lemma~\ref{lemma:2on5expansion}\eqref{enum:comppairexp}). This suggests the following conjecture.

\begin{conjecture}\label{conj:expansion} A Seifert fibered space $Y$ over $S^2$ with $\varepsilon(Y) > 0$ smoothly embeds in $S^4$ if and only if it is obtained by a (possibly empty) sequence of expansions from some $Y'$ of the form $Y' = S^2(1; \frac{p_1}{q_1}, \ldots, \frac{p_l}{q_l})$ with $\frac{p_i}{q_i}>1$ for all $i$ which also smoothly embeds in $S^4$.
\end{conjecture}
Notice that the ``if'' direction of this conjecture is provided by Lemma~\ref{lem:add_fibers}. Since expansion preserves the generalized Euler invariant, the space $Y'$ in this conjecture necessarily satisfies $\varepsilon(Y')=\varepsilon(Y)>0$.

As well as the behaviour in the case $e\geq \frac{k}{2}$ discussed above, we have further evidence for the ``only if'' direction. We find that expansions naturally arise from the partitionable condition. For example, when $e\geq \frac{2k+3}{5}$ a partitionable space is obtained by expansion from some smaller partitionable space (see Lemma~\ref{lemma:2on5expansion}). We also consider the case of $Y$ with all exceptional fibers of even multiplicity. For such spaces the $\mubar$ invariant is particularly effective and shows that if $Y$ smoothly embeds in $S^4$, then in the induced partitions there can only be one complementary class of size larger than two and this class has size three (see Proposition~\ref{prop:even_conditions}). It may be possible that further analysis can rule the existence a complementary class of size three.

If true, Conjecture~\ref{conj:expansion} would reduce the problem of which $Y$ (over base surface $S^2$) smoothly embed in $S^4$ to the case when $e = 1$, which we now briefly discuss. When $k = 3$ and $Y$ is an integer homology sphere several infinite families of examples, as well as some sporadic examples, are known to bound Mazur manifolds and thus to embed in $S^4$ \cite{MR544597, MR634760, MR607896, MR774711}. Donald showed that the rational homology sphere $S^2(1; 4, 4, \frac{12}{5})$ smoothly embeds in $S^4$ \cite[Example~2.14]{MR3271270}. However a conjectural picture of which Seifert fibered spaces embed in the $k=3$ case remains unclear. It is an interesting open question whether there exist any examples which embed with $k \ge 4$. There appears to be some evidence towards a negative answer to this question. In \cite[Conjecture 20]{MR2400877}, Koll\'{a}r conjectures that every Seifert fibered integer homology sphere with $k \ge 4$ does not smoothly bound an integer homology ball, and thus does not smoothly embed in $S^4$. These considerations along with the upper bound from Theorem~\ref{thm:classification_e_ge_(k+1)/2} lead us to a further conjecture, which in particular, implies a negative answer to the aforementioned question. 
\begin{conjecture}\label{conj:3_lines}
If $Y=S^2(e;\frac{p_1}{q_1}, \dots, \frac{p_k}{q_k})$ smoothly embeds in $S^4$, where $\frac{p_i}{q_i}>1$ for all $i$ and $\varepsilon(Y)>0$, then $e\in \{\frac{k+1}{2},\frac{k}{2}, \frac{k-1}{2}\}$.
\end{conjecture}
We prove this conjecture in the special case where every exceptional fiber has even multiplicity. More generally, using the Neumann-Siebenmann invariant we prove a lower bound on $e$, which complements the upper bound given in Theorem~\ref{thm:classification_e_ge_(k+1)/2}.
\begin{restatable}{thm}{thmlowerbound}\label{thm:lower_bound}
Let $Y=S^2(e; \frac{p_1}{q_1}, \dots,\frac{p_k}{q_k})$ be a Seifert fibered space with $\varepsilon(Y)>0$ and $\frac{p_i}{q_i}>1$ for all $i$. Then $Y$ can embed smoothly in $S^4$ only if $\dim H^1(Y;\Z_2)\leq 2e$.
\end{restatable}
If $p_i$ is even for precisely $N \ge 1$ different values of $i\in\{1,\ldots,k\}$, then $\dim H^1(Y;\Z_2) = N-1$ (see Lemma~\ref{lem:order_mod2_cohom}). So when $p_i$ is even for all $i$, Theorem~\ref{thm:lower_bound} provides the lower bound $e\geq \frac{k-1}{2}$ as stipulated by Conjecture~\ref{conj:3_lines}.

It is worth noting that the obstructions considered in this paper use only the fact that $S^4$ is an integer homology sphere. So all of our results could be be restated in terms of Seifert fibered spaces embedding in integer homology $4$-spheres. It is an interesting open question whether there is a 3-manifold which does not embed in $S^4$, but does embed in some other integer homology sphere.

In another direction, we also show that the Seifert fibered spaces over $S^2$ in Theorem~\ref{thm:classification_e_ge_(k+1)/2} and Theorem~\ref{thm:classification_e_ge_k2}\eqref{classif:case2} are double branched covers of doubly slice Montesinos links. A link in $S^3$ is (smoothly) doubly slice if it arises as the cross-section of an unknotted smoothly embedded $2$-sphere in $S^4$. It is an easy consequence of this definition that the double branched cover of a doubly slice link smoothly embeds in $S^4$. 
Note, however, that not every embedding of Seifert fibered spaces can arise in this manner. The integer homology sphere $S^2(1; 3, \frac{5}{2}, \frac{34}{9})$ bounds a Mazur manifold and therefore smoothly embeds in $S^4$ \cite{MR774711}. However, it is the double branched cover of precisely one Montesinos knot, and this knot is not doubly slice (in fact, it is not even slice as it fails the Fox-Milnor condition).

As a consequence of Theorem~\ref{thm:classification_e_ge_(k+1)/2} and these constructions of doubly slice links, we obtain a classification of the smoothly doubly slice odd pretzel knots up to mutation. An odd pretzel knot is one of the form $P(c_1, \ldots, c_n)$, where the $c_i$ are odd integers, see Figure~\ref{fig:pretzel}.

\begin{restatable}{thm}{thmoddpretzels}\label{thm:oddpretzels}
If $K$ is an odd pretzel knot, then the following are equivalent:
\begin{enumerate}[(i)]
\item\label{it:pretzel_embed} $\Sigma(K)$ embeds smoothly in $S^4$,
\item\label{it:pretzel_mutant} $K$ is a mutant of a smoothly doubly slice odd pretzel knot,
\item\label{it:pretzel_list} and $K$ is a mutant of $P(a,-a, \dots, a)$ for some odd $a$ with $|a|\geq 3$.
\end{enumerate}
\end{restatable}

In the special case where the odd pretzel knot has $3$ or $4$ strands, Theorem~\ref{thm:oddpretzels} follows from earlier work of Donald \cite[Theorem 1.5]{MR3271270}.

We also note one further easy application of our results to doubly slice Montesinos links. Although we were unable to find it stated in the literature, it seems possible that the following result was already known in the alternating case.

\begin{restatable}{prop}{qaltMontesinos}\label{prop:qa_doubly_slice}
A  quasi-alternating Montesinos link is never topologically doubly slice. 
\end{restatable}


\subsection*{Structure}
The first three sections of this paper are primarily background material. Section~\ref{sec:background} discusses background material on Seifert fibered spaces and the plumbings they bound.
Section~\ref{sec:alg_top} recounts some homological consequences of embedding 3-manifolds into $S^4$. Section~\ref{sec:SF_homology} is devoted to calculating various homological properties of Seifert fibered spaces.
The analysis of the obstruction based on Donaldson's theorem is given in Section~\ref{sec:obstruction}, where we prove Theorem~\ref{thm:partitions}.
In Section~\ref{sec:application}, we study partitionable spaces and show that under various circumstances partitionable spaces can be obtained by expansion from smaller partitionable spaces. This allows us to prove the obstruction part of Theorem~\ref{thm:classification_e_ge_(k+1)/2} and Theorem~\ref{thm:classification_e_ge_k2}. The proofs of Theorem~\ref{thm:classification_e_ge_(k+1)/2} and Theorem~\ref{thm:classification_e_ge_k2} are completed in Section~\ref{sec:embeddings} by providing embeddings of the necessary spaces. The proof of Lemma~\ref{lem:add_fibers} is contained in this section, as well as some observations about the $\varepsilon=0$ case. In Section~\ref{sec:mubar} our attention turns to the $\mubar$ invariant, allowing us to prove Theorem~\ref{thm:lower_bound}, as well as give various restrictions in the presence of exceptional fibers of even multiplicity. Finally, Section~\ref{sec:dbly_slice} contains the results relating to doubly slice links.
 
\subsection*{Conventions and notation} Throughout this paper $F$ will always denote an orientable surface. We will sometimes use $\Z_n$ to denote the cyclic group $\Z_n = \Z/n\Z$. Unless explicitly stated otherwise all homology and cohomology groups are assumed to have integer coefficients.

\subsection*{Acknowledgements}
The first author would like to thank Cameron Gordon for his guidance, encouragement and support, Kyle Larson for many discussions on the general question of embedding $3$-manifolds in $4$-manifolds and Ryan Budney for his interest in this work and some interesting conversations. The second author would like to thank Brendan Owens for some helpful discussions.

\section{Seifert fibered spaces and plumbings}\label{sec:background}
In this section we briefly recall some standard facts on Seifert fibered spaces and definite manifolds which they bound, as well as establish notation and conventions. See \cite{MR518415} for a more in depth treatment on Seifert fibered spaces and plumbings.

Given a rational number $r > 1$, there is a unique (negative) continued fraction expansion
$$r = [a_1, \ldots, a_{n}]^- := a_1 - \cfrac{1}{a_2 - \cfrac{1}{\begin{aligned}\ddots \,\,\, & \\[-3ex] & a_{n-1} - \cfrac{1}{a_{n}} \end{aligned}}},$$
where $n \ge 1$ and $a_i \ge 2$ are integers for all $i \in \{1, \ldots, n\}$. We associate to $r$ the weighted linear graph (or linear chain) given in Figure~\ref{fig:linearchain}. We call the vertex with weight labelled by $a_i$ the $i$th vertex of the linear chain associated to $r$, so that the vertex labelled with weight $a_1$ is the first, or starting vertex of the linear chain.

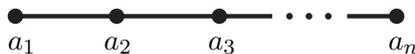
\begin{figure}[h]
  \begin{overpic}[width=150pt]{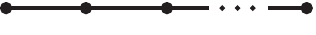}
    \put (0, 0) {$a_1$}
    \put (24, 0) {$a_2$}
    \put (50, 0) {$a_3$}
    \put (95, 0) {$a_n$}
  \end{overpic}
  \caption{Weighted linear chain representing $r = [a_1,\ldots,a_n]^-$.}
  \label{fig:linearchain}
\end{figure}

We denote by $Y_g = F(e; \frac{p_1}{q_1}, \ldots, \frac{p_k}{q_k})$ the Seifert fibered space over the closed orientable genus $g$ surface $F$ given in Figure~\ref{fig:sfs_as_surgery}, where $e\in\Z$, and $\frac{p_i}{q_i} \in \Q$ is non-zero for all $i\in\{1,\ldots,k\}$. When $g = 0$, this is the usual surgery presentation for a Seifert fibered space over $S^2$. In general, each of the $g$ pairs of $0$-framed components increases the genus of the base space by one, see \cite[Appendix]{MR1620508}. 

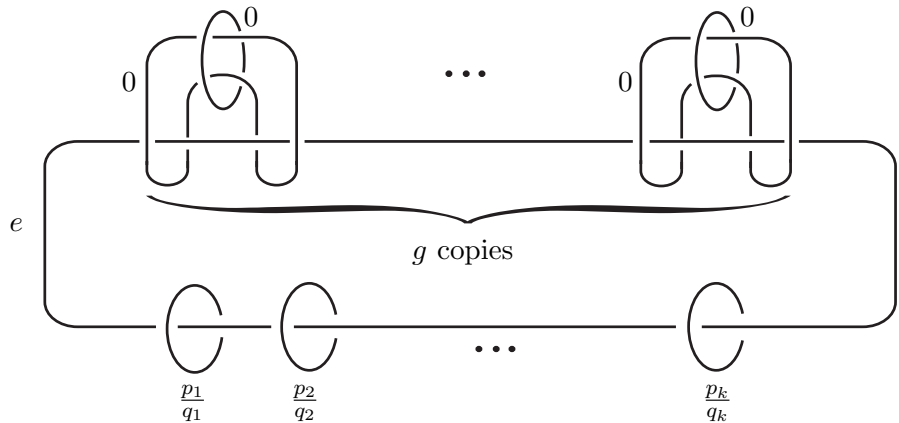
\begin{figure}[!ht]
  \begin{overpic}[height=160pt]{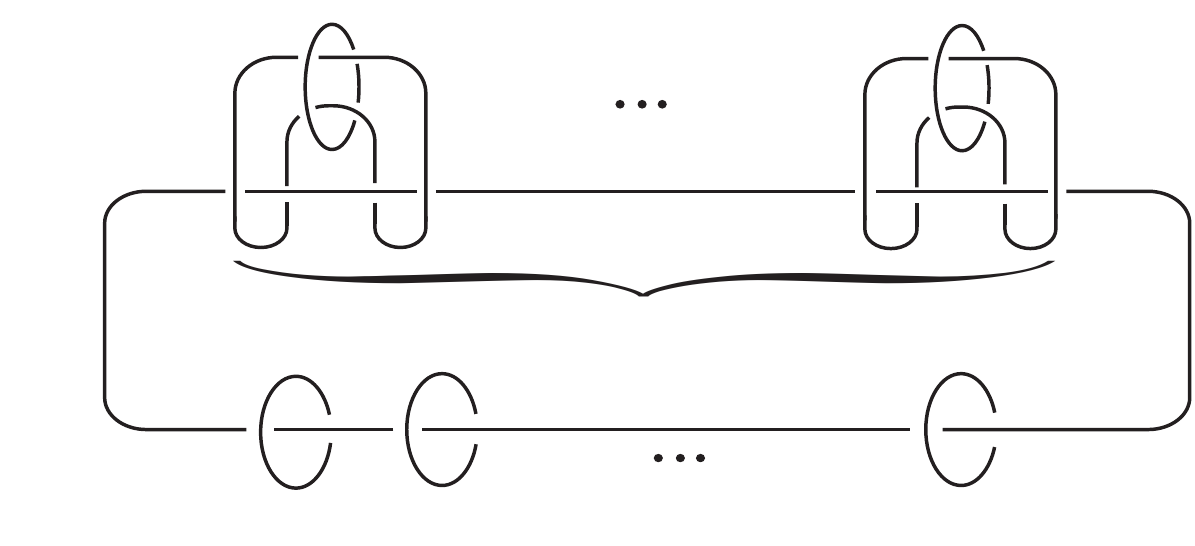}
    \put (5, 20) {$e$}
    \put (23, 1) {$\frac{p_1}{q_1}$}
    \put (35, 1) {$\frac{p_2}{q_2}$}
    \put (79, 1) {$\frac{p_k}{q_k}$}
    \put (48, 17) {$g$ copies}
    \put (17, 35) {$0$}
    \put (30, 42) {$0$}
    \put (70, 35) {$0$}
    \put (83, 42) {$0$}
  \end{overpic}
  \caption{Surgery presentation of the Seifert fibered space $F(e; \frac{p_1}{q_1}, \ldots, \frac{p_k}{q_k})$, where $F$ is an orientable genus $g$ surface.}
  \label{fig:sfs_as_surgery}
\end{figure}
The generalised Euler invariant of $Y_g$ is given by $\varepsilon(Y) = e - \sum_{i=1}^k \frac{q_i}{p_i}$. Every Seifert fibered space $Y_g$ is (possibly orientation reversing) homeomorphic to one in standard form, i.e. such that $\varepsilon(Y_g) \ge 0$ and $\frac{p_i}{q_i} > 1$ for all $i \in \{1,\ldots,k\}$. When in standard form, we call $e$ the normalized central weight of $Y_g$.

We henceforth assume that $Y_g$ is in standard form. Then $Y_g$ bounds a positive semi-definite $4$-manifold which we now describe. We first describe the case $Y_0$ where the base surface is $S^2$. If $\varepsilon(Y_0) \neq 0$ then $Y_0$ is a rational homology sphere with $|H_1(Y_0)| = |(p_1\cdots p_k) \varepsilon(Y_0)|$, and if $\varepsilon(Y_0) = 0$ then $Y_0$ is a rational homology $S^1\times S^2$.

For each $i\in\{1,\ldots, k\}$, we have the unique continued fraction expansion $\frac{p_i}{q_i} = [a_1^i, \ldots, a_{h_i}^i]^-$ where $h_i \ge 1$ and $a_j^i \ge 2$ are integers for all $j \in \{1,\ldots,h_i\}$. We associate to $Y_0 = S^2(e; \frac{p_1}{q_1}, \ldots, \frac{p_k}{q_k})$ the weighted star-shaped graph in Figure~\ref{fig:plumbing}. The $i$th leg (sometimes also called the $i$th arm) of the star-shaped graph is the weighted linear subgraph for $p_i/q_i$ generated by the vertices labelled with weights $a_1^i, \ldots, a_{h_i}^i$. The degree $k$ vertex labelled with weight $e$ is called the central vertex.

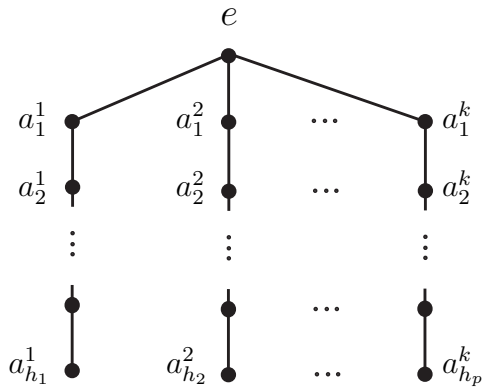
\begin{figure}[h]
  \begin{overpic}[height=150pt]{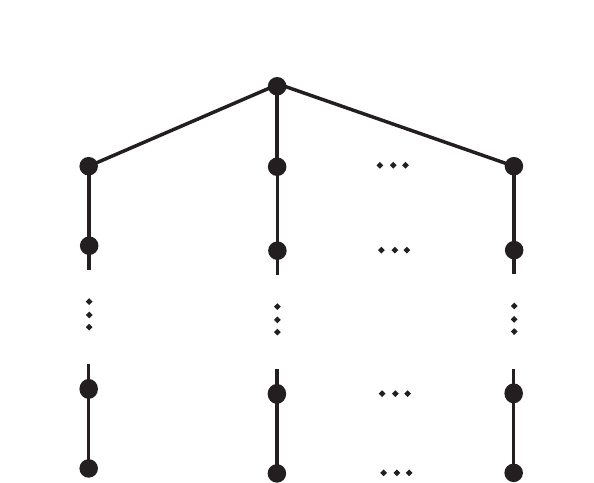}
    \put (45, 73) {\Large $e$}
    \put (4, 52) {\large $a_1^1$}
    \put (4, 38) {\large $a_2^1$}
    \put (2, 2) {\large $a_{h_1}^1$}
    
    \put (36, 52) {\large $a_1^2$}
    \put (36, 38) {\large $a_2^2$}
    \put (34, 2) {\large $a_{h_2}^2$}
    
    \put (90, 52) {\large $a_1^k$}
    \put (90, 38) {\large $a_2^k$}
    \put (90, 2) {\large $a_{h_p}^k$}    


  \end{overpic}
  \caption{The weighted star-shaped plumbing graph $\Gamma$.}
  \label{fig:plumbing}
\end{figure}

Let $\Gamma$ be either the weighted star-shaped graph for $Y_0$, or a disjoint union of weighted linear graphs. There is an oriented smooth $4$-manifold $X_\Gamma$ given by plumbing $D^2$-bundles over $S^2$ according to the weighted graph $\Gamma$. We denote by $|\Gamma|$ the number of vertices in $\Gamma$. Let $m = |\Gamma|$ and denote the vertices of $\Gamma$ by $v_1, v_2, \ldots, v_m$. The zero-sections of the $D^2$-bundles over $S^2$ corresponding to each of $v_1, \ldots, v_m$ in the plumbing together form a natural spherical basis for $H_2(X_\Gamma)$. With respect to this basis, which we call the vertex basis, the intersection form of $X_\Gamma$ is given by the weighted adjacency matrix $Q_\Gamma$ with entries $Q_{ij}$, $1 \le i,j \le m$ given by

$$Q_{ij} = \begin{cases} 
      \text{w}(v_i), & \mbox{if }i = j \\
      -1, & \mbox{if }v_i\mbox{ and }v_j\mbox{ are connected by an edge} \\
      0, & \mbox{otherwise} 
   \end{cases},
$$
where $\text{w}(v_i)$ is the weight of vertex $v_i$. Denoting by $Q_X$ the intersection form of $X$, we call $(H_2(X), Q_X) \cong (\Z^m, Q_\Gamma)$ the intersection lattice of $X_{\Gamma}$ (or of $\Gamma$). We denote the intersection pairing of two elements $x,y\in\Z^m$ by $x\cdot y = x^T\, Q_\Gamma\, y$ and the norm $x \cdot x$ by $\norm{x}$. Now assume that $\Gamma$ is the star-shaped plumbing for $Y_0$. If $\varepsilon(Y) > 0$ then $X_\Gamma$ is a positive definite $4$-manifold and $\Gamma$ is the standard positive definite plumbing graph for $Y_0$. If $\varepsilon(Y_0) = 0$, then $X_\Gamma$ is a positive semi-definite manifold. 

\begin{figure}[!ht]
  \begin{overpic}[height=200pt]{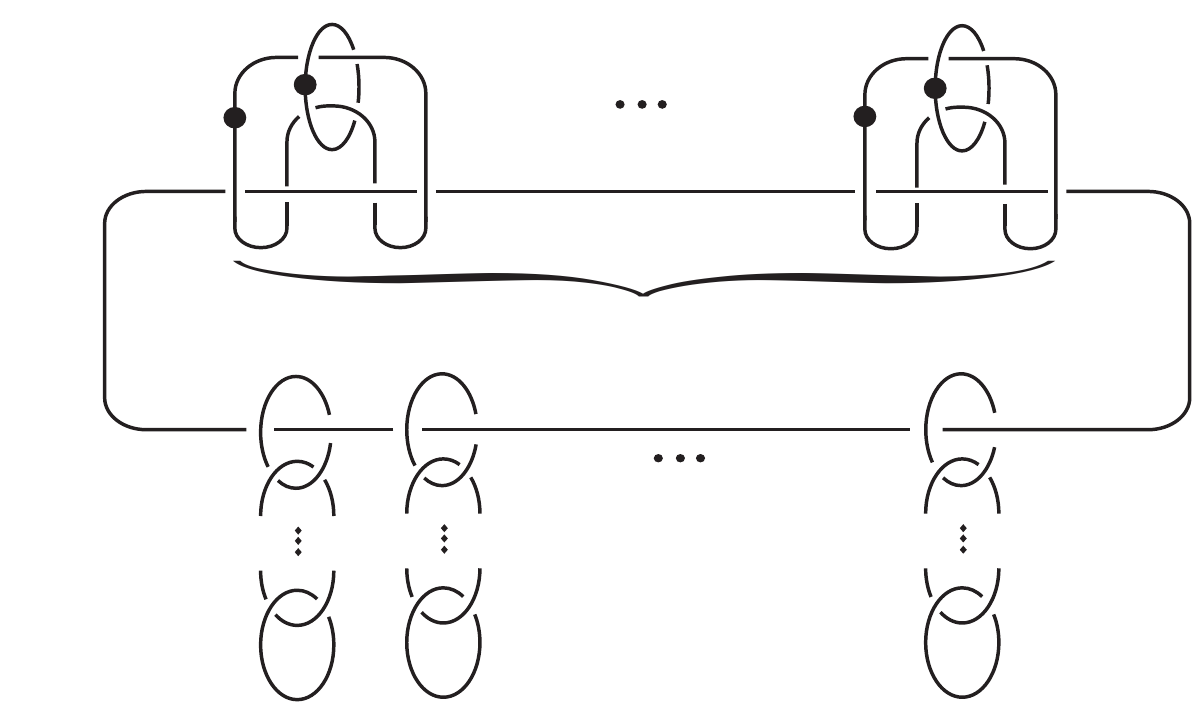}
    \put (5, 32) {$e$}
    \put (48, 30) {$g$ copies}
    \put (23, 29) {$a^1_1$}
    \put (18, 16) {$a^1_2$}
    \put (17, 4) {$a^1_{h_1}$}
    \put (35.5, 29) {$a^2_1$}
    \put (41, 16) {$a^2_2$}
    \put (41, 4) {$a^2_{h_2}$}
    \put (79, 29) {$a^k_1$}
    \put (84.5, 16) {$a^k_2$}
    \put (84.5, 4) {$a^k_{h_k}$}
  \end{overpic}
  \caption{Kirby diagram for the positive semi-definite $4$-manifold $X_{\Gamma,g}$ with boundary the Seifert fibered space $Y_g = F(e; \frac{p_1}{q_1}, \ldots, \frac{p_k}{q_k})$ over the orientable genus $g$ surface $F$. Recall that $\frac{p_i}{q_i} = [a_1^i, \ldots, a_{h_i}^i]^-$ for all $i\in\{1,\ldots,k\}$. The intersection form of $X_{\Gamma,g}$ is isomorphic to $(\Z^m, Q_\Gamma)$ where $\Gamma$ is the weighted star-shaped graph in Figure~\ref{fig:plumbing}.}
  \label{fig:definite_mfd_bounding_sfs}
\end{figure}
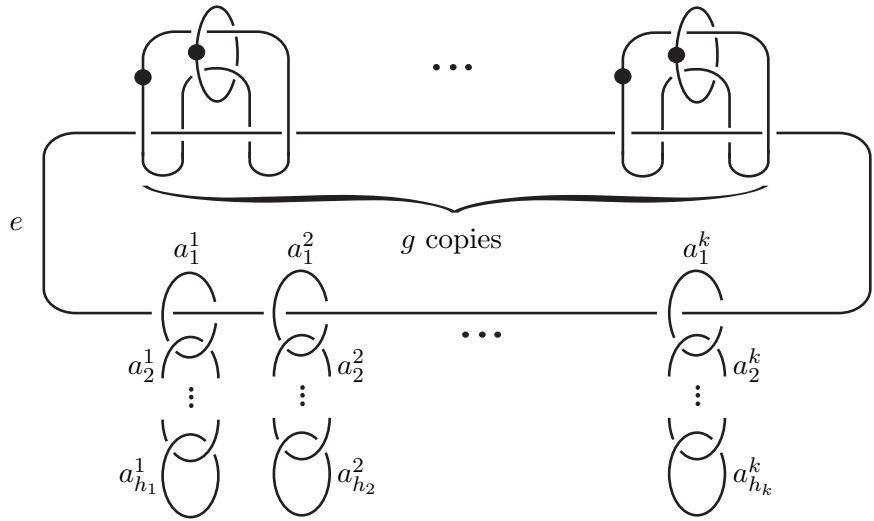

Generalising the case above for $Y_g$ over an orientable genus $g$ surface, we have that $Y_g$ is the boundary of the $4$-manifold $X_{\Gamma,g}$ in Figure~\ref{fig:definite_mfd_bounding_sfs}. Since the $2$-handles do not homologically link any $1$-handles, the intersection form of $X_{\Gamma,g}$ is independent of $g$. In particular, $(H_2(X_{\Gamma,g}), Q_{X_{\Gamma,g}}) \cong (\Z^m, Q_\Gamma)$ where $\Gamma$ is the weighted star-shaped graph in Figure~\ref{fig:plumbing}.

Let $\iota : (\Z^m, Q_\Gamma) \rightarrow (\Z^r, \mbox{Id})$, $r > 0$, be a map of lattices, i.e. a $\Z$-linear map preserving pairings, where $(\Z^r, \mbox{Id})$ is the standard positive diagonal lattice. We denote the orthonormal coordinate vectors of $(\Z^r, \mbox{Id})$ by $e_1,\ldots,e_r$. We call $\iota$ a lattice embedding if it is injective. We adopt the following standard abuse of notation. First, for each $i\in\{1,\ldots,m\}$, we identify the vertex $v_i$ with the corresponding $i$th basis element of $(\Z^m, Q_\Gamma)$. Moreover, we shall identify an element $v \in (\Z^m, Q_\Gamma)$ with its image $\iota(v) \in (\Z^r, \mbox{Id})$.

\section{Homological consequences of embedding in $S^4$}\label{sec:alg_top}
In this section we recall some homological results concerning gluing two $4$-manifolds with boundary and embedding $3$-manifolds in $S^4$.

In \cite[Proposition 7]{latticeineq} the following condition for two 4-manifolds to glue to give a closed definite manifold was proven.
\begin{prop}\label{thm:def_gluing_thm}
Let $U_1$ and $U_2$ be 4-manifolds with $\partial U_1 = -\partial U_2 =Y$. Then the closed 4-manifold $X=U_1 \cup_Y U_2$ is positive definite if and only if
\begin{enumerate}[(a)]
\item\label{enum:injcond} the inclusion-induced map $(i_1)_* \oplus (i_2)_* \colon H_1(Y; \Q) \rightarrow H_1(U_1; \Q)\oplus H_1(U_2; \Q)$ is injective and
\item\label{enum:sigeq} for $i=1,2$, $U_i$ has the maximal possible signature, that is, \[\sigma(U_i) = b_2(U_i) + b_1(U_i) - b_3(U_i) - b_2(Y).\]
\end{enumerate}
\end{prop}

We now give some well-known consequences of a 3-manifold embedding into $S^4$.

\begin{prop}\label{prop:S4_splitting}
 Let $Y$ be a closed orientable $3$-manifold topologically locally flatly embedded in $S^4$. Then $S^4$ can be decomposed as $S^4 = U_1 \cup_Y -U_2$, where $\partial U_1 = \partial U_2 = -Y$, such that
\begin{enumerate}
\item\label{item:h2splits} the restriction map $H^2(U_1;\Z) \oplus H^2(U_2;\Z)\rightarrow H^2(Y;\Z)$ is an isomorphism,
\item\label{item:h3=0} $H^3(U_1;\Z)\cong H^3(U_2;\Z)\cong 0$,
\item\label{item:alexduality} $\tor H^2(U_1;\Z)\cong \tor H^2(U_2;\Z)$, and
\item\label{item:sig=0} $\sigma(U_i) = b_2(U_i) + b_1(U_i) - b_3(U_i) - b_2(Y)=0$.
\end{enumerate} 
\end{prop}
\begin{proof} 
Since $S^4$ has trivial first homology, any embedded connected 3-manifold must separate $S^4$ into two submanifolds which we call $U_1$ and $U_2$. Consider the Mayer-Vietoris sequence for $S^4 = U_1 \cup_Y -U_2$. This contains within it the exact sequence,
\[
0\rightarrow H^2(U_1;\Z) \oplus H^2(U_2;\Z)\rightarrow H^2(Y;\Z)\rightarrow 0,
\]
which proves the restriction map in \eqref{item:h2splits} is an isomorphism. It also contains the exact sequence,
\[
0\rightarrow H^3(U_1;\Z) \oplus H^3(U_2;\Z) \rightarrow H^3(Y;\Z) \rightarrow H^4(S^4;\Z) \rightarrow 0.
\]
Since the map $H^3(Y) \rightarrow H^4(S^4)$ is surjective from $\Z$ to $\Z$ it is an isomorphism, implying \eqref{item:h3=0}.

As $U_1, U_2$ are subsets of $S^4$, Alexander duality shows that $H_1(U_1;\Z)\cong H^2(U_2;\Z)$. However by the universal coefficient theorems $H_1(U_1;\Z)$ and $H^2(U_1;\Z)$ have isomorphic torsion subgroups. We have $\sigma(U_i)=0$ since both $U_i$ and $-U_i$ can be glued to form a positive-definite manifold, but the required value for $\sigma(U_i)$ given by Proposition~\ref{thm:def_gluing_thm} is invariant under change of orientations. 
\end{proof}

The following corollary, first due to Hantzsche \cite{MR1545714}, follows immediately from \eqref{item:h2splits} and \eqref{item:alexduality} of Proposition~\ref{prop:S4_splitting}.

\begin{corol}\label{cor:direct_double} If a $3$-manifold $Y$ embeds topologically locally flatly in $S^4$, then $\tor H_1(Y;\Z)$ splits as a direct double, that is, $\tor H_1(Y;\Z) \cong G \oplus G$ for some finite abelian group $G$.
\end{corol}
In Section~\ref{sec:mubar}, the following well-known variant of Proposition~\ref{prop:S4_splitting} will also be useful.
\begin{lemma}\label{lem:spin_splitting}
Let $Y$ be a rational homology sphere smoothly embedded in $S^4$ which decomposes $S^4$ as $S^4=U_1 \cup_Y -U_2$ with $U_1$ and $U_2$ as in Proposition~\ref{prop:S4_splitting}. Then
\begin{enumerate}
\item $|\spin(Y)|=d^2$ for some integer $d\geq 1$, and
\item for $i=1,2$, the manifold $U_i$ is a spin rational homology ball with $|\spin(U_i)|=d$ and the restriction map $\spin(U_i)\rightarrow \spin(Y)$ is injective.
\end{enumerate}
\end{lemma}
\begin{proof}
First notice that $U_1$ and $U_2$ are spin since they are submanifolds of $S^4$. As $Y$ is a rational homology sphere, it follows immediately from the relevant Mayer-Vietoris sequence that $U_1$ and $U_2$ are rational homology balls. 

By Proposition~\ref{prop:S4_splitting}\eqref{item:alexduality} we see that $H^2(U_1;\Z)\cong H^2(U_2;\Z)$ and hence that $H_1(U_1;\Z)\cong H_1(U_2;\Z)$. Applying the universal coefficient theorem shows that $H^1(U_1;\Z_2)\cong H^1(U_2;\Z_2)$. The Mayer-Vietoris sequence with $\Z_2$ coefficients shows that the restriction maps yield an isomorphism
\[
H^1(U_1;\Z_2) \oplus H^1(U_2;\Z_2) \rightarrow H^1(Y;\Z_2).
\]
Since spin structures on a spin manifold $M$ form a torsor over the group $H^1(M;\Z_2)$, this shows that $|\spin(Y)|=d^2$ where $d=|\spin(U_1)|=|\spin(U_2)|$. The restriction map $\spin(U_i)\rightarrow \spin(Y)$ is injective since the restriction map $H^1(U_i;\Z_2) \rightarrow H^1(Y;\Z_2)$ is injective.
\end{proof}

\section{Homology of Seifert fibered spaces}\label{sec:SF_homology}
In this section we prove several useful statements about the homology of Seifert fibered spaces over orientable base surfaces.
\begin{lemma}\label{lemma:sfs_homology}
The Seifert fibered space $Y=F(e; \frac{p_1}{q_1}, \dots, \frac{p_k}{q_k})$, where $F$ is an orientable genus $g$ surface, has homology
\[H_1(Y;\Z)\cong \Z^{2g} \oplus \bigoplus_{i=1}^{k} \frac{\Z}{D_i\Z},\]
where for $i\in\{1,\ldots,k\}$, $D_i=d_{i+1}/d_i$ where
\[
d_j=
\begin{cases}
1 &\text{if $j=1,2$}\\
\gcd\{ p_{\sigma(1)}p_{\sigma(2)}\cdots p_{\sigma(j-2)} \mid \sigma\in S_{k} \} &\text{if $3\leq j\leq k$}\\
(p_{1} \cdots p_k)\varepsilon(Y)&\text{if $j=k+1$.}
\end{cases}
\]
\end{lemma}
\begin{proof}
  From the surgery description of $Y$ in Figure~\ref{fig:sfs_as_surgery}, we see that $H_1(Y)$ has a presentation matrix given by a diagonal block matrix with $g$ blocks of the form $\left(\begin{smallmatrix} 0 & 0 \\ 0 & 0\end{smallmatrix}\right)$, and a block of the form
\[
A :=
\begin{pmatrix}
e      & 1   & \dots & 1  \\
q_1    & p_1 &       & 0  \\
\vdots &     & \ddots&    \\
q_k    & 0   &       & p_k
\end{pmatrix}.
\]
This shows that $H_1(Y;\Z)\cong \Z^{2g} \oplus \coker A$. For each $i\in\{1,\ldots,k+1\}$, let $d_i$ be the $i$th determinantal divisor of $A$, that is, the greatest common divisor of all $i \times i$ minors of $A$. Then it is a standard algebraic fact that $\coker A \cong \bigoplus_{i=1}^{k} \frac{\Z}{D_i\Z}$, where $D_i=d_{i+1}/d_i$ for all $1\leq i\leq k$. 
We will compute $d_1,\ldots, d_k$ for our particular $A$. Since $A$ contains an entry equal to one, we have $d_1=1$. Since $A$ has a $2\times 2$ minor with determinant one, we have $d_2=1$. 

Let $i \in \{3,4,\ldots,k\}$. The $i \times i$ submatrices of $A$
\[
\begin{pmatrix}
1      & 1 & \dots &1 & 1  \\
p_1    & 0  &    &0 & 0  \\
0      & p_2 &   & \vdots & \vdots \\
\vdots & \vdots  & \ddots & 0 & 0 \\
0      & 0   &   & p_{i-1}  & 0
\end{pmatrix}\mbox{ and } \begin{pmatrix}
e      & 1 & \dots &1 & 1  \\
q_1    & p_1  &  0  &0 & 0  \\
q_2      & 0 & \ddots  & \vdots & \vdots \\
\vdots & \vdots  &   & p_{i-2} & 0 \\
q_{i-1}      & 0   & \cdots  & 0  & 0
\end{pmatrix}
\]
show that (up to signs) $p_1 p_2 \cdots p_{i-1}$ and $p_1 p_2 \cdots p_{i-2} q_{i-1}$ appear as $i \times i$ minors of $A$, and so $d_i$ divides their greatest common divisor, which is $p_1 p_2 \cdots p_{i-2}$. Similarly, one can get that $d_i$ divides $p_{\sigma(1)} \cdots p_{\sigma(i-2)}$ for any permutation $\sigma \in S_k$. However, notice that in any $i \times i$ submatrix $A'$ of $A$, a non-zero product of $i$ entries of $A'$, one from each column and row, must necessarily be a multiple of a product of $i-2$ of $p_1,\ldots,p_k$. Hence, $\det(A')$ is a multiple of $\gcd\{ p_{\sigma(1)}p_{\sigma(2)}\dots p_{\sigma(j-2)} \mid \sigma\in S_{k} \}$. Thus, $d_i = \gcd\{ p_{\sigma(1)}p_{\sigma(2)}\dots p_{\sigma(j-2)} \mid \sigma\in S_{k} \}$.

The final statement in the lemma follows by observing that $d_{k+1} = p_1 \cdots p_k \varepsilon(Y)$, and so $D_k = d_{k+1}/d_k$ is non-zero for $\varepsilon(Y) \neq 0$.
\end{proof}

For a positive prime $p$ we use $V_p(\alpha)$ to denote the $p$-adic valuation of $\alpha$.\footnote{That is $V_p(\alpha)=n$ if $\alpha$ can be written in the form $\alpha=p^n\frac{a}{b}$ with $a,b\in \Z$ both coprime to $p$.} Recall that any finitely generated abelian group can be decomposed as a direct sum
\[
G\cong \Z^m\oplus \bigoplus_{p\, \text{prime}} G_p,
\]
where $G_p$ is the $p$-primary part of $G$. For a cyclic group $\Z/n\Z$ the $p$-primary part is cyclic of order $p^{V_p(n)}$. It will be useful to consider such a decomposition for the homology of Seifert fibered spaces.

\begin{lemma}\label{lemma:p_adic_decomp}
Let $Y=F(e; \frac{p_1}{q_1}, \dots, \frac{p_k}{q_k})$ with $F$ an orientable surface and $\varepsilon(Y)\neq 0$. For a prime $p$, let $v_1\leq \dots \leq v_k$ denote the $p$-adic valuations $V_p(p_1), \dots, V_p(p_k)$ ordered so as to be increasing. Then the $p$-primary part of $H_1(Y;\Z)$ is isomorphic to
\[
\frac{\Z}{p^{v}\Z} \oplus \bigoplus_{i=1}^{k-2} \frac{\Z}{p^{v_i}\Z},
\]
where $v=v_{k}+v_{k-1}+V_p(\varepsilon(Y))$. Moreover we have $v\geq v_{k-1}$ and if $v_k>v_{k-1}$, then $v=v_{k-1}$.
\end{lemma}
\begin{proof}
We can write
\[H_1(Y;\Z)= \Z^{2g} \oplus \bigoplus_{i=1}^{k} \frac{\Z}{D_i\Z},\]
with the $D_i=d_{i+1}/d_i$ as defined in Lemma~\ref{lemma:sfs_homology}. By definition the $d_j$ are such that
\[
V_p(d_j)=
\begin{cases}
0 &\text{if $j=1$ or $2$,}\\
v_1+\dots +v_{j-2} &\text{if $3\leq j\leq k$,}\\
v_1 + \dots + v_k + V_p(\varepsilon(Y)) &\text{if $j=k+1$}.
\end{cases}
\] 
Therefore we have that
\[
V_p(D_j)=
\begin{cases}
0 &\text{if $j=1$,}\\
v_{j-1} &\text{if $1<j<k$,}\\
v_{k-1} + v_k + V_p(\varepsilon(Y)) &\text{if $j=k$.}
\end{cases}
\]
The statements about the $p$-primary part is immediate from these $p$-adic valuation computations. Notice that $\varepsilon(Y)\neq 0$ can be expressed as a fraction with denominator $\lcm(p_1,\dots, p_k)$. Since $V_p(\lcm(p_1,\dots, p_k))=v_k$, this shows that 
$V_p(\varepsilon(Y))\geq -v_k$, which shows that $v=V_p(D_k)\geq v_{k-1}$. Finally suppose that $v_k>v_{k-1}$. In this case when we write each summand of $\varepsilon(Y)=e-\sum \frac{q_i}{p_i}$ as a fraction over the common denominator $\lcm(p_1,\dots, p_k)$, the numerators will all be divisible by $p$ with the exception of the numerator of corresponding to the unique summand $\frac{q_j}{p_j}$ where $V_p(p_j)=v_k$, which will not be divisible by $p$. Thus when we write $\varepsilon(Y)$ as a fraction over $\lcm(p_1,\dots,p_k)$, the numerator will not be divisible by $p$ and hence $V_p(\varepsilon(Y))=-V_p(\lcm(p_1,\dots, p_k))=-v_k$. So $v=v_{k-1}$ as required, in this case.
\end{proof}

We use this to determine the effect of expansion on homology. Although we only deal with the $\varepsilon(Y)\neq 0$ case, it is not hard to see that  a similar result holds when $\varepsilon(Y)=0$.

\begin{lemma}\label{lem:expansion_homo}
Let $Y=F(e;\frac{p_1}{q_1},\dots, \frac{p_k}{q_k})$ be a Seifert fibered space over orientable base surface $F$ with $\varepsilon(Y)\neq 0$. If $Y'=F(e+1;\frac{p_1}{q_1},\dots, \frac{p_k}{q_k},\frac{p_k}{p_k-q_k},\frac{p_k}{q_k})$ is obtained from $Y$ by expansion, then
\[
H_1(Y';\Z) \cong H_1(Y;\Z) \oplus \frac{\Z}{p_k \Z} \oplus \frac{\Z}{p_k\Z}.
\]
In particular $\tor H_1(Y;\Z)$ is a direct double if and only if $\tor H_1(Y';\Z)$ is a direct double. 
\end{lemma}
\begin{proof}
Since expansion preserves the generalized Euler invariant, we have $\varepsilon(Y)=\varepsilon(Y')$.
For a fixed prime $p$, let $v_1\leq \dots \leq v_k$ denote the $p$-adic valuations of $p_1,\dots, p_k$ ordered to be increasing. By Lemma~\ref{lemma:p_adic_decomp} the $p$-primary part of $H_1(Y;\Z)$ is  
\[
\frac{\Z}{p^{v_1}\Z}\oplus \dots \oplus \frac{\Z}{p^{v_{k-2}}\Z} \oplus \frac{\Z}{p^{v}\Z},
\]
where $v=v_{k}+v_{k-1}+ V_p(\varepsilon(Y))$.
Now let $w_1\leq \dots \leq w_{k+2}$ be the $p$-adic valuations of $p_1,\dots, p_k,p_k,p_k$ in increasing order. Notice that this sequence is obtained from the $v_i$ by inserting two extra copies of $V_p(p_k)$ at the appropriate point. First suppose that $V_p(p_k)=v_j$ for some $j\leq k-1$. Calculating the $p$-primary part of $H_1(Y';\Z)$ using Lemma~\ref{lemma:p_adic_decomp} we obtain
\[
\frac{\Z}{p^{v_1}\Z}\oplus \dots \oplus \frac{\Z}{p^{v_{k-2}}\Z} \oplus \frac{\Z}{p^{v}\Z}\oplus \frac{\Z}{p^{v_j}\Z}\oplus \frac{\Z}{p^{v_j}\Z},
\]
since $w_{k+2}=v_k$, $w_{k+1}=v_{k-1}$ and $\varepsilon(Y)=\varepsilon(Y')$. Thus consider the case that $v_k=V_p(p_k)>v_{k-1}$. In this case, we showed in Lemma~\ref{lemma:p_adic_decomp} that $v=v_{k-1}$ and $V_p(\varepsilon(Y))=-v_k$. Thus calculating the $p$-primary part of $H_1(Y';\Z)$ yields
\[
\frac{\Z}{p^{v_1}\Z}\oplus \dots \oplus \frac{\Z}{p^{v_{k-1}}\Z} \oplus \frac{\Z}{p^{v_k}\Z}\oplus \frac{\Z}{p^{v_k}\Z},
\]
since $v=v_{k-1}=w_{k-1}$, $w_k=v_k$ and $w_{k+2}+w_{k+1}+V_p(\varepsilon(Y'))=v_k$.
In either case, the $p$-primary part of $H_1(Y';\Z)$ is obtained from the $p$-primary part of $H_1(Y;\Z)$ by adding a $\frac{\Z}{p^{V_{p}(p_k)}\Z}\oplus \frac{\Z}{p^{V_{p}(p_k)}\Z}$ summand. Since this is true for all primes we see that
\[
H_1(Y';\Z) \cong H_1(Y;\Z) \oplus \frac{\Z}{p_k \Z} \oplus \frac{\Z}{p_k\Z},
\]
as required.
\end{proof}
The following is a key ingredient in the proof of Theorem~\ref{thm:partitions}.
\begin{lemma}\label{lemma:homology_is_double2}
  Let $Y=F(e; \frac{p_1}{q_1}, \dots, \frac{p_k}{q_k})$ be a Seifert fibered space over orientable base surface $F$ with $\varepsilon(Y) > 0$ and $p_i/q_i>1$ for all $i$. Suppose that $\tor H_1(Y) \cong G \oplus G$ for some finite abelian group $G$. If $\mathcal{P} = \{C_1, \ldots, C_n\}$ is a partition of $\{1,\ldots,k\}$ into $n \le e$ classes such that
    \begin{equation}\label{eq:posineq_double}
      \sum_{j \in C_i} \frac{q_j}{p_j} \le 1
    \end{equation}
    for all $i\in\{1,\dots,n\}$, then $n = e$ and there is precisely one value $i \in \{1,\dots,n\}$ for which the inequality in \eqref{eq:posineq_double} is strict and this satisfies
    \[
    1-\sum_{j \in C_i} \frac{q_i}{p_i}=\frac{1}{\lcm (p_1,\dots, p_k)}.
  \]
  Moreover, if $k$ is even then $\gcd(p_1,\dots, p_k) = 1$.
\end{lemma}

\begin{proof}
Since $\tor H_1(Y;\Z)$ is a direct double, then for each prime $p$ the $p$-primary part of $\tor H_1(Y;\Z)$ must also be a direct double. Let $v_1\leq\dots\leq v_k$ be the $p$-adic valuations of the $p_i$ in increasing order. By Lemma~\ref{lemma:p_adic_decomp} the relevant $p$-primary part is isomorphic to
\begin{equation}\label{eq:p_prim_double}
\frac{\Z}{p^{v_1}\Z}\oplus \dots \oplus \frac{\Z}{p^{v_{k-2}}\Z} \oplus \frac{\Z}{p^{v}\Z},
\end{equation}
where $v=v_k+v_{k-1}+ V_p(\varepsilon(Y))\geq v_{k-1}$. Since the $v_i$ are increasing, this can be a direct double only if $v=v_{k-1}=v_{k-2}$. This implies that 
\[V_p(\varepsilon(Y))=-v_k=-V_p(\lcm(p_1,\dots, p_k)).\]
Notice also that we must have an even number of non-trivial summands in \eqref{eq:p_prim_double}. Thus if $k$ is even, we necessarily have $v_1=V_p(\gcd(p_1,\dots, p_k))=0$. Since our choice of prime was arbitrary, it follows that
\[
\varepsilon(Y)=\frac{1}{\lcm(p_1,\dots, p_k)}
\]
and, if $k$ is even, that
\[\gcd(p_1,\dots, p_k)=1.\]
Now suppose that we have a partition $\mathcal{P}$ as in the statement of the lemma. We may split $\varepsilon(Y)$ up as follows:

\[
\varepsilon(Y)=e-n + \sum_{k=1}^{n} \left(1-\sum_{i\in C_k} \frac{q_i}{p_i}\right)=\frac{1}{\lcm(p_1,\dots, p_k)},
\]
where $1-\sum_{i\in C_k} \frac{q_i}{p_i}\geq 0$ for all $k$. Thus we see immediately that $e=n$. However notice that if $1-\sum_{i\in C_k} \frac{q_i}{p_i}> 0$, then
\[
1-\sum_{i\in C_k} \frac{q_i}{p_i}\geq \frac{1}{\lcm(p_1, \dots, p_k)}.
\]
Consequently we must have $1-\sum_{i\in C_k} \frac{q_i}{p_i}=0$ for all but one $k$ for which we have the required equality.
\end{proof}
The following will be useful in Section~\ref{sec:mubar}.
\begin{lemma}\label{lem:order_mod2_cohom}
Let $Y=S^2(e; \frac{p_1}{q_1}, \dots, \frac{p_k}{q_k})$ be a Seifert fibered space with $\varepsilon(Y)\neq 0$ and $N$ exceptional fibers of even multiplicity. If $N\geq 1$, then
\[\dim H^1(Y;\Z_2)= N-1.\]
If $N=0$, then
\[\dim H^1(Y;\Z_2)\leq 1.\]
\end{lemma}
\begin{proof}
Since $\Z_2$ is a field $\dim H_1(Y;\Z_2)=\dim H^1(Y;\Z_2)$. Thus we will compute $\dim H_1(Y;\Z_2)$. By the universal coefficient theorem, $\dim H_1(Y;\Z_2)$ is equal to the number of summands in the 2-primary part of $H_1(Y;\Z)$. Let $0\leq v_1\leq \dots \leq v_k$ be the 2-adic valuations of the $p_i$ ordered to be increasing. By Lemma~\ref{lemma:p_adic_decomp} this 2-primary part can be written as
\[
\frac{\Z}{2^{v_1}\Z}\oplus \dots \oplus \frac{\Z}{2^{v_{k-2}}\Z} \oplus \frac{\Z}{2^{v}\Z},
\]
where $v\geq v_{k-1}$.
By assumption precisely $N$ of $v_1,\ldots,v_k$ are non-zero. So if $N\geq 2$, then $N-1$ values of $v_1,\dots, v_{k-1}$ are non-zero, giving the desired number of summands. If $N\leq 1$, then only $v$ can be non-zero, so $\dim H^1(Y;\Z_2)\leq 1$ in this case. However, if $N=1$, then $v_k>v_{k-1}=0$, so Lemma~\ref{lemma:p_adic_decomp} also shows that $v=v_{k-1}=0$ in this case, as required.
\end{proof}

We end the section with one easy topological application of Proposition~\ref{lemma:homology_is_double2}, which is the topological analogue of the upper bound in Theorem~\ref{thm:classification_e_ge_(k+1)/2}.
\begin{prop}\label{prop:topbound}
  \label{prop:top_bound} Let $Y=F(e; \frac{p_1}{q_1}, \ldots, \frac{p_k}{q_k})$ with $\frac{p_i}{q_i} > 1$ for all $i$ and $\varepsilon(Y)> 0$. If $Y$ embeds topologically in $S^4$, then $e\leq k-1$. 
  \end{prop}
\begin{proof}
If $Y$ embeds into $S^4$, then Proposition~\ref{prop:S4_splitting} shows that $\tor H^2(Y)\cong \tor H_1(Y)$ is a direct double. This implies that $e\leq k-1$. For if $e \geq k$, the partition $\{\{1\},\{2\},\ldots,\{k\}\}$ would violate the conditions of Lemma~\ref{lemma:homology_is_double2} since there would be $k > 1$ classes for which the inequality \eqref{eq:posineq_double} of Lemma~\ref{lemma:homology_is_double2} is strict.
\end{proof}
\begin{remark}
The bound in Proposition~\ref{prop:top_bound} is sharp. It follows from the work of Freedman that every integer homology sphere embeds topologically locally flatly in $S^4$ \cite{MR679066}. For a given $k\geq 3$, there exist Seifert fibered integer homology spheres for any value of $e$ in the range $1\leq e \leq k-1$.
\end{remark}

\section{Obstruction to smoothly embedding Seifert fibered spaces}\label{sec:obstruction}

In this section we analyse an obstruction to smoothly embedding a Seifert fibered space $Y$ over an orientable base surface in $S^4$ coming from Donaldson's theorem, culminating in a proof of Theorem~\ref{thm:partitions}.

For the duration of this section we will use the following notation. Let
\[Y=F \left(e; \frac{p_1}{q_1}, \dots, \frac{p_k}{q_k}\right)\]
be a Seifert fibered space over an orientable base surface of genus $g$ with $\varepsilon(Y) > 0$ and $\frac{p_i}{q_i}>1$ for all $i$. As in Figure~\ref{fig:definite_mfd_bounding_sfs}, there is a positive-definite $X$ with boundary $Y$ and intersection form $(\Z^n, Q_{\Gamma})$, where $\Gamma$ is a weighted star-shaped graph as in Figure~\ref{fig:plumbing} and $n = |\Gamma|$ is the number of vertices in $\Gamma$.

Before embarking on the proof, we summarise the idea behind the obstruction based on Donaldson's theorem as follows. A smooth embedding of $Y$ into $S^4$ splits $S^4$ into two $4$-manifolds $U_1$ and $U_2$ with boundary $-Y$. The smooth manifold $W_i = X \cup U_i$ is positive definite, so Donaldson's theorem implies that it has standard diagonal intersection form. The inclusion map $X \hookrightarrow{} W_i$ induces maps of intersection lattices $\iota_i : (H_2(X), Q_X) \rightarrow (H_2(W_i), \mbox{Id})$, which we can write as the transpose of an integer matrix $A_i$. Following Greene-Jabuka \cite{MR2808326}, Donald proved that the image of the restriction map $H^2(W_i) \rightarrow H^2(Y)$ is isomorphic to $\frac{\im A_i}{\im Q_X}$ \cite[Theorem 3.6]{MR3271270}. Combining this with the fact that the restriction-induced map $H^2(U_1) \oplus H^2(U_2) \rightarrow H^2(Y)$ is an isomorphism, he showed that $\frac{\im A_1}{\im Q_X} \oplus \frac{\im A_2}{\im Q_X} = \coker Q_X$. This condition implies that the augmented matrix $(A_1 | A_2)$ is surjective over the integers, see Theorem~\ref{thm:embsurjectivity}.

Using the fact that $H_1(Y)$ must split as a direct double, we are able to prove some structural results about the form any lattice embedding $(H_2(X), Q_X) \rightarrow (\Z^{b_2(X)}, \mbox{Id})$ must take. An important ingredient in this proof is the lattice inequality from \cite[Theorem 6]{latticeineq}, stated below. It is this result which makes an analysis of the obstruction based on Donaldson's theorem feasible.
\begin{thm}[Theorem 6 of \cite{latticeineq}] \label{thm:emb_ineq_eq_case}
Let $\iota : (\Z^{|\Gamma'|}, Q_{\Gamma'}) \rightarrow (\Z^m, \mbox{Id})$ be a lattice embedding, where $m > 0$ and $\Gamma'$ is a disjoint union of weighted linear chains representing fractions $\frac{p_1}{q_1}, \ldots, \frac{p_n}{q_n} \in \Q_{>1}$. Suppose that there is a unit vector $w \in (\Z^m, \mbox{Id})$ which pairs non-trivially with (the image of) the starting vertex of each linear chain. Then
  \[\sum_{i=1}^n \frac{q_i}{p_i} \leq 1.\]
 Moreover, if we have equality then $w$ has pairing $\pm 1$ with the starting vertex of each linear chain.
\end{thm}

The following theorem is the key obstruction to smoothly embedding Seifert fibered spaces in $S^4$ derived from Donaldson's theorem. It is a slight variation of \cite[Corollary 3.9]{MR3271270}.
\begin{thm}\label{thm:embsurjectivity}
If $Y$ embeds smoothly into $S^4$, then there exist lattice embeddings 
\[\iota_i : (\Z^{n},Q_{\Gamma})\rightarrow (\Z^{n},\mbox{Id})\]
for $i=1,2$, such that the augmented matrix $(A_1 | A_2)$ is surjective, where $A_i$ is the transpose of the integer matrix representing $\iota_i$ for $i=1,2$.
\end{thm}

\begin{proof}
Unless explicitly stated otherwise, all homology and cohomology groups in this proof are taken with coefficients in $\Z$. If $Y$ embeds smoothly into $S^4$, then Proposition~\ref{prop:S4_splitting} shows that it splits into two $4$-manifolds $U_1$ and $U_2$, with $\partial U_1\cong \partial U_2 \cong -Y$. Let $W_i$ be the closed manifold $W_i=X \cup_Y U_i$. We claim that Proposition~\ref{thm:def_gluing_thm} implies this is positive definite with $b_2(W_i)=b_2(X)$. To see this, note that in Proposition~\ref{thm:def_gluing_thm} injectivity condition \eqref{enum:injcond} is satisfied since the map $H_1(Y;\Q) \rightarrow H_1(X;\Q)$ is injective, and the signature condition \eqref{enum:sigeq} follows from $b_1(X) + b_2(X) - b_3(X) - b_2(Y) = 2g + n - 0 - 2g = n = \sigma(X)$, together with Proposition~\ref{prop:S4_splitting}\eqref{item:sig=0}. Thus, Donaldson's diagonalization theorem implies that the intersection form of $W_i$ is diagonalizable. Hence, the inclusion $H_2(X) \rightarrow H_2(W_i)$ induces an embedding of lattices $\iota_i:(\Z^{b_2},Q_{\Gamma})\rightarrow (\Z^{b_2},I)$, for $i\in\{1,2\}$. 

Now fix $i\in\{1,2\}$. By considering the long exact sequences of pairs and the inclusion $(X, Y) \xhookrightarrow{} (W_i, U_i)$, we have the following commutative diagram with exact rows:
\[\xymatrix{
  H^2(W_i,U_i) \ar[d]^{\cong i_1} \ar[r]^\alpha &H^2(W_i)\ar[d]^{i_2}\ar[r]^\beta & H^2(U_i) \ar[d]^{i_3} \ar[r] & H^3(W_i, U_i) \ar[d]^{\cong i_4}\\
  H^2(X,Y) \ar[r]^\gamma   &H^2(X)\ar[r]^\delta & H^2(Y) \ar[r]^\epsilon & H^3(X,Y).}\]
It follows by excision that the maps $i_1$ and $i_4$ are isomorphisms. 

%
Recall that we have a basis for $H_2(X)$ for which the intersection form of $X$ is given by the matrix $Q_{\Gamma}$. By the universal coefficient theorems, $\tor H^2(X) \cong \tor H_1(X) = 0$, so we may choose the dual basis for $H^2(X) \cong \Hom(H_2(X),\Z)$. We choose the Poincar{\'e} dual basis for $H^2(X, Y)$. With respect to these bases the map $\gamma$ is represented by $Q_{\Gamma}$. By Donaldson's theorem, we can choose a basis for $H^2(W_i)/\tor \cong \Hom(H_2(W_i),\Z)$ for which $Q_W = \mbox{Id}$. The map $H^2(W_i)/\tor \rightarrow H^2(X)$ is dual to the inclusion induced map $H_2(X) \rightarrow H_2(W)/\tor$, and is therefore given by $A_i$ with respect to these choices of bases.

Now notice that $H^3(X,Y)\cong H_1(X)$ is torsion free. Thus $\tor H^2(Y)\subseteq \ker \epsilon$. However since $H^3(X)=0$, the map $\epsilon$ is surjective. Since $H^2(Y)$ and $H^3(X,Y)$ have the same rank we see that $\im \delta = \tor H^2(Y)=\ker \epsilon$. This allows us to identify $\tor H^2(Y)$ with $\coker \gamma$ via $\delta$.
Since $i_1$ is an isomorphism, we see that $\im\gamma \subset \im i_2$. In turn this shows that $\delta$ induces an injective map $\frac{ \im i_2}{\im \gamma} \rightarrow \tor H^2(Y)$. We have that $\im(i_3 \circ \beta) = \im(\delta \circ i_2) \subset \tor H^2(Y)$, which we may identify with $\frac{ \im i_2}{\im \gamma}$ via $\delta$. Since $H^2(X)$ is torsion free, $i_2$ maps finite order elements of $H^2(W)$ to $0$. Thus, in coordinates with respect to the bases given earlier $\frac{ \im i_2}{\im \gamma}$ is given by $\frac{\im A_i}{\im Q_{\Gamma}}$.

We claim that $\im (i_3 \circ \beta) = \im (i_3)$. It suffices to check these finite groups have the same order. By Proposition~\ref{prop:S4_splitting}\eqref{item:h2splits}, the order of $\im (i_3)$ is the square root of $|\tor H_1(Y)| = |\coker Q_{\Gamma}|$, where in this last equality uses the fact that $Q_{\Gamma}$ presents $\tor H_1(Y)$. Using the fact that $Q_{\Gamma} = A_i A_i^T$ we see that this is also the order of $\im (i_3 \circ \beta) \cong \frac{\im A_i}{\im Q_{\Gamma}}$, proving the claim. Thus, we can identify $\tor H^2(Y)$ with $\Z^{b_2}/\im Q_{\Gamma}$, and under this identification the image of $\tor H^2(U_i)\rightarrow \tor H^2(Y)$ is $\im A_i/\im Q_{\Gamma}$.

By Proposition~\ref{prop:S4_splitting}\eqref{item:alexduality} the map $\tor H^2(U_1)\oplus \tor H^2(U_1) \rightarrow \tor H^2(Y)$ is an isomorphism. Thus  
\begin{equation}\label{eq:sumcoker}
\frac{\Z^{b_2}}{\im Q_X}\cong \frac{\im A_1}{\im Q_X} \oplus \frac{\im A_2}{\im Q_X},
\end{equation} 
where the direct sum is an internal direct sum as subspaces of $\coker Q_X$.
 
It suffices to show that \eqref{eq:sumcoker} implies $\im (A_1 \mid A_2) = \Z^{b_2(X)}$. Let $x \in \Z^{b_2(X)}$ and let $q : \Z^{b_2(X)} \rightarrow \coker Q_X$ be the quotient map. By Equation \eqref{eq:sumcoker}, $q(a_1) + q(a_2) = q(x)$ for some $a_1 \in \im(A_1)$ and $a_2 \in \im(A_2)$. Thus, $a_1 + a_2 = x + k$ for some $k \in \im(Q_X)$. Since $Q_X = A_1 A_1^T$, we have $\im Q_X \subset \im (A_1)$. Therefore $(a_1 - k) + a_2 = x$ shows that $x \in \im (A_1 \mid A_2)$, as required.
\end{proof}

With $Y$, $X$ and $\Gamma$ as defined at the beginning of this section, we have the following lemma which in particular shows that from an embedding of lattices we can define a partition.

\begin{lemma}\label{lem:embedding_structure}
 Let $\iota : (\Z^{|\Gamma|}, Q_\Gamma) \rightarrow (\Z^{N}, \mbox{Id})$, where $N > 0$ be a lattice embedding. Let $\{e_1,\dots, e_N\}$ be an orthonormal basis for $(\Z^{N}, \mbox{Id})$. If $\tor H_1(Y;\Z) = G \oplus G$ for some abelian group $G$, then upto an automorphism of $\Z^N$ we may assume the following. The image of the central vertex is $e_1+\dots + e_e$. For each $i\in \{1,\dots, e\}$ let $C_i$ be the subset of $\{1,\dots, k\}$ consisting of $j$ such that the first vertex of the linear chain $p_j/q_j$ pairs non-trivially with $e_i$. Then
 \begin{enumerate}
 \item $\{C_1,\ldots,C_e\}$ is a partition of $\{1,\dots, k\}$ such that
 \[
 \sum_{j\in C_i} \frac{q_i}{p_i}=1
 \]
 for $i=1,\dots, e-1$ and
 \[
 \sum_{j\in C_e} \frac{q_i}{p_i}=1-\frac{1}{\lcm(p_1,\dots, p_k)}
 \]
 \item and for $i\in \{1, \dots ,e \}$ the vertices with which $e_i$ pairs non-trivially are precisely the leading vertices of the arms in $C_i$ and the central vertex. 
 \end{enumerate}
\end{lemma}
\begin{proof}
Let $p_i/q_i=[a_1^i,\dots, a_{l_i}^i]^-$, where $a_j^i\geq 2$. Let $v_j^i$ denote the image of the $j$th vertex in the linear chain corresponding to $p_i/q_i$. So $\norm{v_j^i}=a_j^i$. Let $\nu$ be the image of the central vertex.
By applying an automorphism of $\Z^N$ we may assume that $\nu$ takes the form $\nu=\alpha_1 e_1 + \dots + \alpha_n e_n$ with $\alpha_i>0$ and $n\leq e$. Let $C_1,\dots, C_n$ be the sets defined by
\[
C_i=\{j\in\{1,\ldots,k\}\,|\, e_i\cdot v_1^j \neq 0\}
\]
as in the statement of the lemma. Since $\nu \cdot v_1^j=-1$, each $j$ in the range $1\leq j\leq k$ is contained in at least one $C_i$. {\em A priori} the $C_i$ may not be a partition, since they may not be pairwise disjoint and some $C_i$'s may be empty. However by discarding repetitions, we can obtain $C_i'$ such that $C_i'\subseteq C_i$ and the non-empty $C_i'$ form a genuine partition of $\{1, \dots, k\}$.

By Theorem~\ref{thm:emb_ineq_eq_case} we can conclude that for each $i$ we have 
\begin{equation*}
\sum_{j\in C_i'} \frac{q_j}{p_j}\leq \sum_{j\in C_i} \frac{q_j}{p_j} \leq 1.
\end{equation*}
Thus by Lemma~\ref{lemma:homology_is_double2} the partition consisting of the $C_i'$ has precisely $e$ non-empty classes. It follows that $\nu$ must take the form $\nu=e_1+ \dots + e_e$ as required. Furthermore Lemma~\ref{lemma:homology_is_double2} also implies that after permuting the $e_i$ if necessary, we can assume that
\begin{equation}\label{eq:C_i'equal}
\sum_{j\in C_i'} \frac{q_j}{p_j}= \sum_{j\in C_i}\frac{q_j}{p_j}=1
\end{equation}
for $i=1,\dots, e-1$ and
\begin{equation}\label{eq:C_e_equal}
1-\frac{1}{\lcm(p_1,\dots, p_k)}=\sum_{j\in C_e'} \frac{q_j}{p_j}\leq \sum_{j\in C_e}\frac{q_j}{p_j}\leq 1.
\end{equation}
This shows that $C_i=C_i'$ for $i=1,\dots, e-1$. To show that the $C_i$ form a partition, it remains to verify that $C_e=C_e'$. We will use the following claim to complete the proof.

\begin{claim}
Let $v^j_s$ be a vertex such that $j\not\in C_l'$ but $v^j_s \cdot e_l\neq 0$ for some $l$ in the range $1\leq l \leq e$. Then $s=1$, $l=e$ and $v^j_s \cdot e_e = \pm 1$.
\end{claim}
\begin{proof}
Since $j\not\in C_l'$ the vector $v^j_s$ is orthogonal to all vertices in the linear chains corresponding to elements of $C_l'$. As  we can consider a single vertex as a linear chain in its own right, Theorem~\ref{thm:emb_ineq_eq_case} applies to show that
\begin{equation*}
\frac{1}{\norm{v^j_s}}+ \sum_{i \in C_l'} \frac{q_i}{p_i}\leq 1.
\end{equation*}
By \eqref{eq:C_i'equal} and \eqref{eq:C_e_equal} we see that this is only possible if $l=e$ and $\norm{v^j_s}\geq \lcm(p_1,\dots, p_k)$.
However since $\norm{v^j_s}=a^j_s$ appears in the continued fraction expansion for $p_j/q_j$, we see that $\norm{v^j_s}\leq p_j$ with equality only if $p_j/q_j= \norm{v^j_s}$ is an integer, in which case $s=1$. As $\lcm(p_1,\dots, p_k)\geq p_j$, this implies that $s=1$ and $\norm{v^j_s}=\lcm(p_1,\dots, p_k)$. However, by \eqref{eq:C_e_equal} we have that $\frac{1}{\norm{v^j_s}}+ \sum_{i \in C_l'} \frac{q_i}{p_i}=1$. Thus we can apply the equality case of Theorem~\ref{thm:emb_ineq_eq_case} to conclude that $v^j_s \cdot e_e = \pm 1$. 
\end{proof}

We will now check that $C_e'=C_e$. If not, then there would be a vertex $v^j_1$ for some $j\not\in C_e'$ such that $v^j_1 \cdot e_e\neq 0$. By the claim, such a vertex satisfies $v^j_1 \cdot e_e=\pm 1$. However we have $j\in C_l$ for some unique $1\leq l<e$. By the equality case of Theorem~\ref{thm:emb_ineq_eq_case}, this implies that $v_1^j \cdot e_l =\pm 1$. Thus $\nu \cdot v_1^j = v_1^j \cdot e_l + v_1^j \cdot e_e$ must be even, contradicting $v_1^j \cdot \nu=-1$ . Thus we can conclude that $C_e'=C_e$ completing the proof that $C_1, \dots, C_e$ are a partition.

Finally, we check that the non-leading vertices cannot pair non-trivially with $e_l$ for any $l\in \{1,\dots, e\}$. Since the non-leading vertices have trivial pairing with the central vertex $\nu$, it suffices to check that a non-leading vertex can pair non-trivially with $e_l$ for at most one $l\in \{1,\dots, e\}$. However, this follows easily from the above claim, which shows that for $s>1$ a vertex $v_s^j$ can pair non-trivially with $e_l$ only if $j\in C_l'=C_l$. This completes the proof.
\end{proof}

For the following lemma let $Y$, $X$ and $\Gamma$ as defined at the beginning of this section, and recall that a class $C \subset \{1,\ldots,k\}$ is called complementary if $\sum_{i \in C} \frac{q_i}{p_i} = 1$.
\begin{lemma}\label{thm:compclasscond}
  Suppose that $\tor H_1(Y) \cong G \oplus G$ for some abelian group $G$. For $i=1,2$, let $\iota_i : (\Z^n, Q_{\Gamma}) \rightarrow (\Z^n, \mbox{Id})$, where $n = |\Gamma|$, be a map of lattices, and represent $\iota_i$ as an integer matrix by the transpose of $A_i$. Let $A = (A_1 | A_2)$, and suppose that the column space of $A$ is all of $\Z^n$. For $i\in\{1,2\}$, let $P_i$ be the partition of $\{1,\ldots,k\}$ induced by $\iota_i$ as in Lemma~\ref{lem:embedding_structure}. Then, no non-empty union of complementary classes of $P_1$ is a union of complementary classes of $P_2$.
\end{lemma}

\begin{proof} 
We are assuming that both $\iota_1$ and $\iota_2$ satisfy the conclusions of Lemma~\ref{lem:embedding_structure}. For $i\in\{1,2\}$, let $C_1^i,\ldots, C_{\ell_i}^i$ be a non-empty collection of complementary classes in $P_i$. Suppose for sake of contradiction that $\cup_{i=1}^{\ell_1} C_i^1 = \cup_{i=1}^{\ell_2} C_i^2$ and denote their common union by $H \subset \{1,\ldots,k\}$. Since $\sum_{i \in C} \frac{q_i}{p_i} = 1$ for a complementary class $C$, we have $\ell_1 = \sum_{i\in H} \frac{q_i}{p_i} = \ell_2$ and we denote their common value by $\ell$. Our goal will be to find a non-zero row vector $\overline{x}$ with coprime integer entries and an integer $p>1$ such that $\overline{x} A_i \equiv 0 \bmod p$ for both $i=1$ and $i=2$. Given such a vector we will use that $\overline{x} A \equiv 0 \bmod p$ to show that $A$ is not surjective over $\Z$.

 Let $R$ be the weighted star-shaped graph with central weight $\ell$ and legs given by the legs of $\Gamma$ indexed by elements of $H$. 
For $i=1,2$, there is a map of lattices $q_i : (\Z^{|R|}, Q_R) \rightarrow (\Z^n, \mbox{Id})$ which is the restriction of $\iota_i$ on the non-central vertices of $R$ and maps the central vertex of $R$ to $e_1 + \cdots + e_{\ell}$. That $q_i$ is a map of lattices follows from the structure imposed by Lemma~\ref{lem:embedding_structure}. The classes $C^i_1, \dots, C^i_\ell$ are precisely the ones whose leading vertices pair non-trivially with exactly one of the unit vectors $e_1,\dots, e_\ell$ and this non-trivial pairing is necessarily $-1$ in all cases. Furthermore no non-leading vertex pairs non-trivially with any of $e_1,\dots, e_\ell$.

  Recall that the image of the vertices of $\Gamma$ under $\iota_i$ are given by the rows of $A_i$. By ordering the vertices, we may assume that the first row of $A_i$ corresponds to the central vertex $\nu$, and the next $|R| - 1$ rows correspond to the non-central vertices of $\Gamma$ that appear in $R$. Let $B_i$ be the transpose of the integer matrix representing $q_i$. With the above choice of vertex ordering, $B_i$ is obtained by taking the first $|R|$ rows of $A_i$, and replacing the first row by the vector $(\underbrace{1, 1, \ldots, 1}_{\text{$\ell$ ones}}, 0, \ldots 0)$.

 For both $i=1,2$, we have $B_i B_i^T = Q$, where $Q = Q_R$ is the matrix representing the intersection lattice $(\Z^{|R|}, Q_R)$ with respect to the vertex basis. Since the classes $C_1^i,\ldots,C_\ell^i$ are complementary, the boundary of the plumbing with weighted graph $R$ is a Seifert fibered space $Y'$ with $\varepsilon(Y') = 0$. Thus $\det{Q} = 0$, implying that there exists a non-zero row vector $x=(x_1,\dots, x_{|R|}) \in \Z^{|R|}$ such that $x Q = 0$. Hence, $(x B_i) (x B_i)^T = x Q x^T = 0$, implying $x B_i = 0 \in \Z^n$. Thus we have obtained $x$ such that $xB_1=xB_2=0\in \Z^n$. By dividing out by any common factors we may assume that $\gcd(x_1, \dots, x_{|R|})=1$.
\begin{claim}
The entry $x_1$ is divisible by an integer $p>1$.
\end{claim}
With this claim, the proof concludes as follows. Consider the vector $\bar{x}=(x_1,\dots, x_{|R|},0,\dots, 0)\in \Z^n$. Since $B_i$ is obtained from $A_i$ by taking the first $|R|$ rows and modifying the first row, we see that every entry of $\bar{x} A_i$ is a multiple of $x_1$. This shows that $\bar{x} A = \bar{x}\cdot (A_1 \mid A_2) \equiv 0 \pmod p$, where $p$ is the integer from the claim.

Since $\gcd(x_1,\dots, x_{|R|})=1$, we can write $1$ as an integer combination of the $x_i$. This implies there is a column vector $v\in \Z^n$ such that $\bar{x} v=1$. If $A$ were surjective, then there would be a vector $u$ such that $v=Au$. This would show $0\equiv\bar{x} A u =\bar{x} v= 1 \pmod{p}$,
which is a contradiction. We complete the proof by proving the claim.

\begin{proof}[Proof of Claim] Consider a leg in $R$ with corresponding fraction $p/q$. We will show that $p$ divides $x_1$. Suppose that the continued fraction expansion of $p/q$ is $\frac{p}{q} = [a_1, \ldots, a_{\rho-1}]^-$, where $a_j \ge 2$ are integers for all $j$. By ordering the vertices we may assume that the first $\rho$ rows of $B_i$ correspond to the central vertex followed by the vertices of our chosen leg in $R$. Thus, the top-left $\rho\times \rho$ submatrix of $Q = B_i B_i^T$ is
      $$\begin{pmatrix}
      \ell & -1  & 0 & 0  &\cdots& 0 & 0  \\
      -1   & a_1 & -1&  0 &\cdots& 0 & 0 \\
      0    &-1  &a_2& -1 &\cdots& 0 & 0 \\
      \vdots & &   &    &  \ddots & & \vdots    \\
      0    &0   & 0 & 0 & \cdots &a_{\rho-2} & -1 \\
      0    &0   & 0 & 0 & \cdots & -1 & a_{\rho-1} \\
      \end{pmatrix}.
      $$
Let $Q' = \left(\begin{array}{cccc}
        -1 & 0 & \cdots \\\hline
          & Q_{p/q} &
        \end{array}\right)
        $
be the matrix obtained from the above matrix by removing the first column, where $Q_{p/q}$ is the intersection matrix of the linear chain representing $p/q$. We have $$(x_1, \ldots, x_{\rho+1})\cdot Q' = 0,$$ since $xQ = 0$ and the corresponding columns $2,\ldots,\rho$ of $Q$ are supported in the first $\rho$ rows. This implies that $(x_2, \ldots, x_{\rho})\cdot Q_{p/q} = (x_1, 0, \ldots, 0)$. Thus, we can change the last row of $Q_{p/q}$ to $(x_1, 0, \ldots, 0)$, by first multiplying the last row of $Q_{p/q}$ by $x_{\rho}$, then for each $j\in\{1,\ldots,\rho-1\}$ adding $x_j$ multiples of the $j$th row to the last row. The determinant of this new matrix is $x_{\rho-1} \cdot \det(Q_{p/q}) = x_{\rho-1} \cdot p$. However, by expanding the determinant along the final row we see that
      $$\begin{vmatrix}
      a_1 & -1&  0 & 0 &\cdots& 0 & 0 \\
      -1  &a_1& -1 & 0 &\cdots& 0 & 0 \\
      0   &-1 &a_2 &-1 &\cdots& 0 & 0 \\
      &   &    & & \ddots    \\
      0   & 0 & 0 & 0 & \cdots &a_{\rho-2} & -1 \\
      x_1 & 0 & 0 & 0 & \cdots & 0 & 0 \\
      \end{vmatrix} = x_1.
      $$
  Thus $x_1 = x_{\rho-1} p$ is a multiple of $p$, proving the claim.
\end{proof}

\end{proof}

This allows us to prove our main obstruction to embedding Seifert fibered spaces in $S^4$.

\thmcomplementarypartitions*
\begin{proof} Suppose that $Y$ smoothly embeds in $S^4$. By Corollary~\ref{cor:direct_double}, $\tor H_1(Y)$ splits as a direct double. Theorem~\ref{thm:embsurjectivity} implies that there are lattice embeddings $\iota_i : (\Z^{|\Gamma|}, Q_{\Gamma}) \rightarrow (\Z^{|\Gamma|}, \mbox{Id})$, where $i\in\{1,2\}$ and $\Gamma$ is the weighted star-shaped graph describing the intersection lattice of the standard positive-definite $4$-manifold bounding $Y$. Moreover, $(A_1|A_2)$ is surjective, where $A_i$ is the transpose of the integer matrix representing $\iota_i$ for $i\in\{1,2\}$. As in Lemma~\ref{lem:embedding_structure}, for $i\in\{1,2\}$, there is a partition $P_i$ induced by $\iota_i$ satisfying properties \eqref{cond:noncomp} and \eqref{cond:ineq} of Definition~\ref{def:partitionable}. Lemma~\ref{thm:compclasscond} shows that no non-empty union of any subset of complementary classes of $P_1$ is a union of any subset of complementary classes of $P_2$.

For $i=1,2$, let $\mathcal{C}_i$ be a non-empty proper subset of $P_i$. For sake of contradiction suppose that $\cup_{C \in \mathcal{C}_1} C = \cup_{C \in \mathcal{C}_2} C$, and let $H \subset \{1,\ldots,k\}$ be their common union. Properties \eqref{cond:noncomp} and \eqref{cond:ineq} imply that for $i\in\{1,2\}$, $\mathcal{C}_i$ contains a non-complementary class if and only if $\sum_{j\in H} \frac{q_i}{p_i}$ is not an integer. Thus, $\mathcal{C}_1$ and $\mathcal{C}_2$ either both contain a non-complementary class or both do not. Thus either $P_1$ and $P_2$, or $P_1\backslash \mathcal{C}_1$ and $P_2\backslash \mathcal{C}_2$ contain only complementary classes. This shows that property~\eqref{cond:comp} of Definition~\ref{def:partitionable} holds.
\end{proof}
\section{Applications of Theorem~\ref{thm:partitions}}\label{sec:application}
Now we consider which spaces can pass the obstruction given by Theorem~\ref{thm:partitions} when $e\geq \frac{k}{2}$. We will prove the obstruction halves of Theorem~\ref{thm:classification_e_ge_(k+1)/2} and Theorem~\ref{thm:classification_e_ge_k2}, leaving the construction of the embeddings into $S^4$ to Section~\ref{sec:embeddings}. 

\sphereembthmhalfbound*
\begin{proof}
  Let $P_1$ and $P_2$ of $\{1,\ldots,k\}$ be the partitions from Theorem \ref{thm:partitions}, each into $e$ classes. For each partition, there are $e$ classes and at most one class of size one, since a size one class must be non-complementary. Thus, $k \ge 1 + 2(e-1)$, and so $e \le \frac{k+1}{2}$. Now assume that $e = \frac{k+1}{2}$, in particular $k$ is odd. For each partition all but one class has size $2$, and the remaining class has size $1$. Using that no non-empty proper subset of classes in $P_1$ is a union of classes in $P_2$, we without loss of generality assume that $P_1 = \{\{1\},\{2,3\},\{4,5\},\ldots,\{k-1,k\}\}$ and $P_2 = \{\{1,2\}, \{3,4\}, \ldots, \{k-2,k-1\}, \{k\}\}$. By Lemma~\ref{lemma:homology_is_double2}, $1 - \frac{q_1}{p_1} = \frac{1}{\lcm(p_1,\ldots,p_k)}$, and thus $\frac{p_1}{q_1} = \frac{a}{a - 1}$ where $a = \lcm(p_1,\ldots,p_k)$. For a complementary classes $\{i,j\}$ we have $\frac{q_i}{p_i} + \frac{q_j}{p_j} = 1$. Applying this to the complementary classes in $P_1$ and $P_2$ allows us to write the remaining fractions in terms of $a$, which shows that $M$ is of the required form.

Finally, the fact that the Seifert fibered spaces of this form smoothly embed in $S^4$ follows from Proposition \ref{prop:s4_embeddings2} proved in Section~\ref{sec:embeddings}.
\end{proof}

We now analyse the $e = \frac{k}{2}$ case. The reader may find it helpful to recall definitions of \emph{expansion} (Definition~\ref{def:expansion}) and \emph{partitionable} (Definition~\ref{def:partitionable}) stated in the introduction. We first prove the following lemma.

\begin{lemma}\label{lemma:2on5expansion} Let $Y = F(e; \frac{p_1}{q_1}, \ldots, \frac{p_k}{q_k})$ be a Seifert fibered space over orientable base surface $F$ with $k \ge 3$, $\frac{p_i}{q_i} > 1$ for all $i$, and $\varepsilon(Y) > 0$. Suppose that $Y$ is partitionable with partitions $P_1$ and $P_2$ such that either
  \begin{enumerate}[(i)]
  \item\label{enum:comppairexp} $m_1 + m_2 \ge e$ where $m_i$ is the number of complementary pairs in $P_i$ for $i\in\{1,2\}$, or
  \item\label{enum:singleton} both $P_1$ and $P_2$ contain a class of size one, or
  \item\label{enum:2on5ineq} $e \ge \frac{2k+3}{5}$.
   \end{enumerate}
  Then $Y$ is an expansion of a partitionable Seifert fibered space $Y'$.
\end{lemma}
\begin{proof} By Lemma~\ref{lem:expansion_homo} the property that $\tor H_1$ is a direct double is not changed by expansions. Thus, in order to show that $Y'$ is partitionable it suffices to come up with partitions satisfying the three remaining conditions in Definition~\ref{def:partitionable}.
  
  Suppose first that \eqref{enum:comppairexp} holds. We claim that there are complementary pairs $\{a,b\} \in P_1$ and $\{b,c\} \in P_2$ with $a,b,c$ distinct. Suppose otherwise, then $\sum_{i=1}^k \frac{q_i}{p_i} \ge m_1 + m_2 \ge e$ since each complementary pair contributes one and there are $m_1 + m_2$ disjoint complementary pairs, but this contradicts Definition~\ref{def:partitionable} which implies that $\sum_{i=1}^k \frac{q_i}{p_i} < e$.
  
  By permuting the fractions $\frac{p_1}{q_1},\ldots,\frac{p_k}{q_k}$, we may assume that $b=k$, $a = k-1$, $c=k-2$. Since $\{k-1,k\}$ and $\{k-2,k\}$ are complementary, we have that $\frac{p_{k-2}}{q_{k-2}}=\frac{p_{k-1}}{q_{k-1}} = \frac{p_k}{p_k-q_k}$. Thus $Y$ is an expansion of $Y' = F(e-1; \frac{p_1}{q_1}, \ldots, \frac{p_{k-2}}{q_{k-2}})$. Let $P_1' = P_1 \setminus \{\{k-1,k\}\}$ and let $P_2'$ be obtained from $P_2 \setminus \{\{k-2,k\}\}$ by replacing $k-1$ with $k-2$ in the class $C$ containing $k-1$, and call this new class $C'$.
  We claim that $P_1'$ and $P_2'$ satisfy the conditions in Definition~\ref{def:partitionable}, showing that $Y'$ is partitionable. These conditions follow from the corresponding conditions for $P_1$ and $P_2$. Conditions \eqref{cond:noncomp} and \eqref{cond:ineq} follow immediately. To see condition \eqref{cond:comp} let $S_1 \subsetneq P_1'$ and $S_2 \subset P_2'$ be non-empty with the union of classes in $S_1$ equal to the union of classes in $S_2$. We denote their common union by $H \subset \{1,\ldots,k-2\}$.
 If $k-2 \not\in H$ then this would contradict condition \eqref{cond:comp} for $P_1, P_2$ since $S_1 \subset P_1$ and $S_2 \subset P_2$. Similarly, if $k-2 \in H$ then $S_1 \cup \{\{k-1,k\}\}$ and $(S_2 \cup \{C\}) \setminus \{C'\}$ would contradict condition \eqref{cond:comp} for $P_1, P_2$. This proves the conclusion if \eqref{enum:comppairexp} holds.

  Now suppose that \eqref{enum:singleton} holds. If $k = 3$ then $P_1$ and $P_2$ each contain a complementary class of size two and \eqref{enum:comppairexp} holds. Thus we can assume that $k \ge 4$ and by permuting the fractions we may assume that $\{k\} \in P_1$ and $\{k-2\} \in P_2$. In particular these are the non-complementary classes so $\frac{p_k}{q_k} = \frac{p_{k-2}}{q_{k-2}} = m/(m-1)$, where $m = \lcm(p_1,\ldots,p_k)$. Let $C \in P_2$ be the complementary class containing $k$, and let $i \in C$ with $i \neq k$. Since $C$ is complementary $\frac{m-1}{m} + \frac{q_i}{p_i} \le 1$ with equality only if $C$ has size two. Rearranging this gives $\frac{q_i}{p_i} \le \frac{1}{m}$. However, $\frac{q_i}{p_i} \ge \frac{1}{m}$ since $m = \lcm(p_1,\ldots,p_k) \ge p_i$. Thus we must have equality and so $|C| = 2$. Similarly the complementary class in $P_1$ containing $k-2$ has size two. Since $k > 3$, this implies that we can assume that $P_1$ and $P_2$ take the form
  $$P_1 = \{\ldots,\{\ldots,k-3\}, \{k-2\}, \{k-1,k\}\},$$
  $$P_2 = \{\ldots,\{\ldots,k-1\}, \{k-3,k-2\}, \{k\}\}.$$
  Then $Y$ is an expansion of $Y' = F(e-1; \frac{p_1}{q_1}, \ldots, \frac{p_{k-2}}{q_{k-2}})$ with partitions $P_1' = P_1\setminus \{\{k-1,k\}\}$ and $P_2'$ obtained from $P_2 \setminus \{\{k\}\}$ by replacing the class $C$ containing $k-1$ by $C' := C\setminus \{\{k-1\}\}$. We check the conditions of Definition~\ref{def:partitionable}. First \eqref{cond:noncomp} and \eqref{cond:ineq} are immediate, noting that $C' \in P_2'$ is the non-complementary class. To verify \eqref{cond:comp}, let $S_1 \subsetneq P_1'$ and $S_2 \subset P_2'$ be non-empty with the union of classes in $S_1$ equal to the union of classes in $S_2$. If $S_2$ does not contain $C'$ then $S_1 \subset P_1$, $S_2 \subset P_2$ contradicting \eqref{cond:comp} for $P_1, P_2$. If $S_2$ contains $C'$ then $S_1 \cup \{\{k-1,k\}\}$ and $(S_2 \cup \{\{k\},\{k-3,k-2\},C\})\setminus \{C'\}$ would contradict \eqref{cond:comp} for $P_1, P_2$. This completes the proof if \eqref{enum:singleton} holds.
  
  Now suppose that \eqref{enum:2on5ineq} holds, so $e \ge \frac{2k + 3}{5}$. If \eqref{enum:singleton} holds then we are done, so we may assume that the non-complementary class of $P_2$ has size at least two. We now show that \eqref{enum:comppairexp} holds. Let $m_i$ be the number of complementary pairs in $P_i$ for $i\in\{1,2\}$. Thus there are $e-m_i-1$ complementary classes in $P_i$ of size at least $3$, for $i\in\{1,2\}$. Hence,
   \begin{align*}
       k &\ge 1 + 2m_1 + 3(e-m_1-1),\mbox{ and} \\
       k &\ge 2 + 2m_2 + 3(e-m_2-1).
   \end{align*}
  
  Adding these inequalities give $2k \ge 6e - (m_1+m_2) - 3$. Rearranging gives
  $$m_1 + m_2 \ge 6e - 2k - 3 \ge e + (5e - 2k) - 3 \ge e,$$
  since $e \ge \frac{2k + 3}{5}$. This completes the proof.
\end{proof}
Now we are ready to analyze the $e=\frac{k}{2}$ case.
\sphereembthm*

\begin{proof} We will prove that if $Y$ embeds then it takes the desired form. We leave the proof that the family in \eqref{classif:case2} smoothly embeds to the next section, see Proposition~\ref{prop:s4_embeddings2}.

 Suppose that $Y=F(e;\frac{p_1}{q_1}, \dots, \frac{p_k}{q_k})$ with $k=2e$ is partitionable. Since the property that $k=2e$ is preserved under expansions, we can assume that $Y$ is obtained by a (possibly empty) sequence of expansions from a partitionable space that is minimal in the sense that it is not obtained by expansion from any other partitionable space. Assume that $Y$ is such a minimal space. By Lemma~\ref{lemma:2on5expansion}\eqref{enum:2on5ineq} minimality implies that $e\leq \frac{2k+2}{5} =\frac{4e+2}{5}$. This shows that $e\leq 2$.

If $e=1$, then $Y=F(1;\frac{p}{q}, \frac{r}{s})$ for some $\frac{p}{q}, \frac{r}{s}$ such that $\frac{q}{p}+ \frac{s}{r}=1-\frac{1}{\lcm(p,r)}$. However Lemma~\ref{lemma:homology_is_double2} implies that $p$ and $r$ are coprime so $\lcm(p,r)=pr$.

If $e=2$, then $Y=F(2;\frac{p_1}{q_1}, \dots, \frac{p_4}{q_4})$. We consider the possible partitions, $P_1=\{C_1,C_2\}$ and $P_2=\{D_1,D_2\}$ of such a $Y$. We assume that $C_1$ and $D_1$ are the complementary classes and $C_2$ and $D_2$ are the non-complementary classes. By Lemma~\ref{lemma:2on5expansion} the minimality of $Y$ shows that we cannot have $|C_2|=|D_2|=1$ or $|C_1|=|C_2|=2$. Thus we can assume that $|C_1|=3$, $|C_2|=1$, $|D_1|=2$ and $|D_2|=2$.
Suppose that $C_2=\{1\}$. This implies that $\frac{q_1}{p_1}=1-\frac{1}{\lcm(p_1,\dots, p_4)}$. We may assume that $\{1,2\}$ is a class in $P_2$. Since $\frac{p_2}{q_2}\leq \lcm(p_1,\dots, p_4)$, we have that $\frac{q_1}{p_1}+\frac{q_2}{p_2}\geq 1$. Thus $D_1=\{1,2\}$ is the complementary class and $\frac{p_2}{q_2}=\lcm(p_1,\dots, p_4)$. By Lemma~\ref{lemma:homology_is_double2}, we have $\gcd(p_1,\dots, p_4)=1$. Since $p_1=p_2=\lcm(p_1,\dots, p_4)$, it follows that $p_3$ and $p_4$ must be coprime. Since the complementary class $C_1$ is $C_1=\{2,3,4\}$, it follows that $\frac{q_3}{p_3}+\frac{q_4}{p_4}+\frac{1}{\lcm(p_1,\dots, p_4)}=1$. This implies that $\lcm(p_1,\dots, p_4)=p_3 p_4$. Thus by taking $\frac{p_3}{q_3}=\frac{p}{q}$ and $\frac{p_4}{q_4}=\frac{r}{s}$ we see that $Y$ takes the form $Y=F(2;\frac{p}{q},\frac{r}{s}, \frac{pr}{pr-1}, pr)$, where $\frac{q}{p} + \frac{s}{r}+\frac{1}{pr}=1$.

Thus if $Y$ is partitionable and $e=\frac{k}{2}$, then $Y$ is obtained by a sequence of expansions from either $F(1;\frac{p}{q}, \frac{r}{s})$ or $F(2;\frac{p}{q},\frac{r}{s}, \frac{pr}{pr-1}, pr)$, where $\frac{q}{p} + \frac{s}{r}+\frac{1}{pr}=1$. By Theorem~\ref{thm:partitions}, this shows that if $Y$ smoothly embeds in $S^4$, then it is of the form required by the theorem.
\end{proof}
\begin{remark}\label{rem:size_3_class}
We remark that the family \eqref{classif:case3} in Theorem~\ref{thm:classification_e_ge_k2} arises only when one of the partitions has a complementary class indexing fractions of the form $\frac{p}{q}, \frac{r}{s}, pr$. The above proof shows this when $e=2$, and it follows inductively for larger $e$ from the way the partitions for $Y'$ are obtained from $P_1$ and $P_2$ in the proof of Lemma~\ref{lemma:2on5expansion}.
\end{remark}

\section{Constructing embeddings of Seifert fibered spaces}\label{sec:embeddings}
In this section we construct embeddings of the families of Seifert fibered spaces in Theorem \ref{thm:classification_e_ge_(k+1)/2} and Theorem \ref{thm:classification_e_ge_k2}\eqref{classif:case2}. We also recall what is known in the $\varepsilon(Y) = 0$ case and make some observations which give some new embeddings.

\lemaddfibers*
\begin{proof}
Let $Y=F(e;\frac{p_1}{q_1}, \dots, \frac{p_k}{q_k})$ and $Y'=F(e;\frac{p_1}{q_1}, \dots, \frac{p_k}{q_k}, -\frac{p_k}{q_k},\frac{p_k}{q_k})$ a space obtained by expansion from $Y$.
  We will explicitly find a subset of $Y\times[0,1]$ which is homeomorphic to $Y'$. Let $N_1\subset Y$ be a Seifert fibered neighbourhood of the exceptional fiber corresponding to $p_k/q_k$, that is, a set homeomorphic to $S^1 \times D^2$ whose boundary is a union of regular fibres. Consider the set $M=N_1 \times [\frac{1}{4}, \frac{3}{4}]$. The boundary $\partial M$ is homeomorphic to $S^1 \times S^2$ and it naturally inherits a Seifert fibred structure of the form $\partial M =S^2(0; -\frac{p_k}{q_k}, \frac{p_k}{q_k})$. On $N_1 \times \{\frac{1}{4}\}$ and $N_1 \times \{\frac{3}{4}\}$ this structure is a translate of the one on $N_1$, giving the two exceptional fibres, and is the obvious product structure on $\partial N_1 \times [\frac{1}{4}, \frac{3}{4}]$. Now let $N_2 \subseteq N_1$ be a Seifert fibred neighbourhood of a regular fiber. 
  We take $X$ to be the subset \[ X=(Y\setminus \inter N_2) \times \{0\} \cup \partial N_2 \times [0, \frac{1}{4}] \cup (\partial M \setminus \inter N_2 \times \{\frac{1}{4})\}. \]
  As a manifold, $X$ is obtained by taking $Y$ and $M$, deleting open fibred neighbourhoods of regular fibers in both and gluing the two resulting manifolds along their boundaries so that the boundary fibers match up. From this description $X$ is clearly homeomorphic to $Y'$. Thus by smoothing the corners of $X$ we can obtain a smooth embedding of $Y'$ into $Y\times [0,1]$.
\end{proof}
\begin{remark}
Although all our applications are for Seifert fibered spaces over orientable surfaces, both the definition of expansion and Lemma~\ref{lem:add_fibers} work perfectly well over non-orientable surfaces.
\end{remark}
The following proposition is due to Crisp-Hillman \cite[Lemma 3.2]{MR1620508}.

\begin{prop}\label{prop:genus_bump}Let $Y_g = F_g(e; \frac{p_1}{q_1}, \ldots, \frac{p_k}{q_k})$ where $F_g$ is an orientable genus $g \ge 0$ surface. If $Y_g$ smoothly embeds in $S^4$, then $Y_{g+1}$ smoothly embeds in $S^4$.
\end{prop}

\begin{proof} We follow the approach due to Donald \cite[Lemma 2.23]{DonaldPhDThesis}. We prove that $Y_{g+1}$ smoothly embeds in $Y_g \times [0,1]$ via Kirby calculus. Start with a surgery presentation for $Y_g$ as in Figure~\ref{fig:sfs_as_surgery}. Take a relative handle decomposition of $Y_g \times [0,1]$ by attaching handles around the meridian of the curve representing the central vertex (the $e$ framed curve) as shown in Figure~\ref{fig:genus_bump}. To see the embedding of $Y_{g+1}$ in this manifold observe that the dotted circle and one of the $0$-framed unknots form a Whitehead double, so their boundary along with the surgery presentation for $Y_g$ provide the embedding into $Y_g \times [0,1]$. To see that the Kirby diagram is $Y_g \times [0,1]$, observe that $0$-framed handle in the Whitehead double can be unlinked from the dotted curve by sliding over the meridional $0$-framed unknot. This curve can then be cancelled with the $3$-handle, leaving the $1$-handle and $2$-handle which form a cancelling pair.

  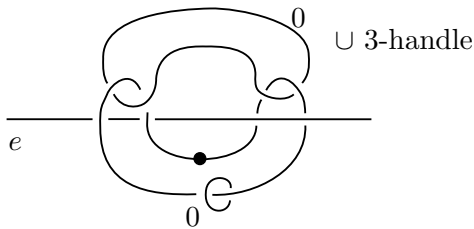
\begin{figure}[h]
  \begin{overpic}[width=0.3\textwidth]{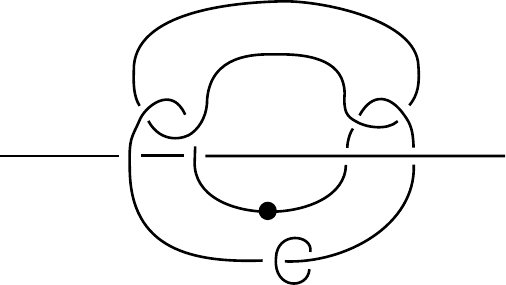}
    \put (78,51) {$0$}
    \put (49,-4) {$0$}
    \put (0.5,17) {$e$}
    \put (90,45) {$\cup$ 3-handle}
  \end{overpic}
  \caption{Increasing the genus}
  \label{fig:genus_bump}
\end{figure}
\end{proof}
Together these allow us to find the embeddings required for Theorem~\ref{thm:classification_e_ge_(k+1)/2} and Theorem~\ref{thm:classification_e_ge_k2}.

\begin{prop}\label{prop:s4_embeddings2} Let $Y$ be a Seifert fibered space over orientable base surface $F$, with $k > 2$ exceptional fibers, in either of the following two families:
  \begin{enumerate}[(a)]
  \item\label{enum:emb_family_1} $F\left(\frac{k+1}{2}; \frac{a}{a-1}, a, \ldots, \frac{a}{a-1} \right) = F(0; -a, a, \ldots, -a),$ where $a > 1$ is an integer, or
  \item\label{enum:emb_family_2} $F\left(\frac{k}{2}; \frac{p}{q}, \frac{p}{p-q}, \cdots, \frac{p}{q}, \frac{r}{s}, \frac{r}{r-s}, \cdots, \frac{r}{s}\right) = F\left(0; \frac{p}{q}, -\frac{p}{q}, \ldots, \frac{p}{q}, \frac{r}{s}, -\frac{r}{s}, \ldots, \frac{r}{s}\right)$ where $\frac{p}{q}, \frac{r}{s} > 1$ and $\frac{q}{p} + \frac{s}{r} = 1 - \frac{1}{pr}.$
  \end{enumerate}
  Then $Y$ smoothly embeds in $S^4$.
\end{prop}
\begin{proof}
Observe that $S^3$ admits Seifert fibered structures of the form $S^2(1;\frac{a}{a-1})$ and $S^2(1;\frac{p}{q}, \frac{r}{s})$, where $\frac{q}{p} + \frac{s}{r} = 1 - \frac{1}{pr}$. Since $S^3$ smoothly embeds in $S^4$ and each of the families is obtained from one of these structures on $S^3$ by a sequence of expansions and possibly increasing the genus of the base surface, Lemma~\ref{lem:add_fibers} and Proposition~\ref{prop:genus_bump} allow us to build the necessary embeddings. 
\end{proof}

\begin{remark}\label{rmk:embs_in_lit}Some of the Seifert fibered spaces in Proposition~\ref{prop:s4_embeddings2} were already known to embed in $S^4$. Crisp-Hillman \cite[Section 3a]{MR1620508} showed that the manifolds in \eqref{enum:emb_family_1} embed in $S^4$. Donald \cite{MR3271270} showed that for $k = 3,4$, the manifolds in family \eqref{enum:emb_family_1} and a subfamily of those in \eqref{enum:emb_family_2} embed in $S^4$ as the double branched cover of doubly slice links.
\end{remark}

We now recall what is known about and make some brief observations on smoothly embedding Seifert fibered spaces $Y$ over an orientable base surface with $\varepsilon(Y) = 0$.

Donald \cite[Theorem~1.3]{MR3271270} used Donaldson's theorem to prove that in order for $Y$ to smoothly embed the Seifert invariants must occur in complementary pairs. More precisely, he shows the following.

\begin{thm}\label{thm:eps_eq_0_obstruction} Let $Y$ be a Seifert fibered space over a closed orientable surface $F$ with $\varepsilon(Y) = 0$. If $Y$ smoothly embeds in $S^4$ then $Y$ is of the form
  $$F\left(0; \frac{p_1}{q_1}, -\frac{p_1}{q_1}, \ldots, \frac{p_k}{q_k}, -\frac{p_k}{q_k}\right) = F\left(k; \frac{p_1}{q_1}, \frac{p_1}{p_1 - q_1}, \ldots, \frac{p_k}{q_k}, \frac{p_k}{p_k - q_k}\right),$$
where $k \ge 0$ and $\frac{p_i}{q_i} > 1$ for all $i \in \{1,\ldots,k\}$.
\end{thm}

We remark that a proof of Theorem~\ref{thm:eps_eq_0_obstruction} also follows from \cite[Proof of Theorem~4]{latticeineq}. It is still not known precisely which Seifert fibered spaces $Y$ of the form given in Theorem~\ref{thm:eps_eq_0_obstruction} smoothly embed in $S^4$. Crisp-Hillman \cite[Remark following Lemma 3.1]{MR1620508} showed that if $p_i$ is odd for all $i\in\{1,\ldots,k\}$ then $Y$ smoothly embeds. Donald \cite{MR3271270} showed that $S^2(0;a,-a,b,-b)$, where $a,b\in \Z$ are non-zero, embeds if $a$ is even and $b$ is odd. If $a$ and $b$ are both even and $a \neq b$, then he used Furuta's $10/8$ theorem to show that the Seifert fibered space does not embed. 
It turns out that embedding Seifert fibered spaces with $\varepsilon=0$ is closely related to embedding Seifert fibered spaces over $D^2$. We will make use of the following easy observation. 
\begin{lemma}\label{lem:disk_subspace}
Let $Y=F(e;\frac{p_1}{q_1}, \dots, \frac{p_k}{q_k})$, then for any subset $\{i_1,\dots, i_l\}\subseteq \{1,\dots, k\}$, $Y$ contains a submanifold homeomorphic to $D^2(\frac{p_{i_1}}{q_{i_1}}, \dots, \frac{p_{i_l}}{q_{i_l}})$.
\end{lemma}
\begin{proof}
Consider the projection of $Y$ onto its base orbifold $\widehat{F}$. Choose a disk in $\widehat{F}$ containing the cone points corresponding to the exceptional fibers given by the fractions $\frac{p_{i_1}}{q_{i_1}}, \dots, \frac{p_{i_l}}{q_{i_l}}$ in its interior. The pre-image of this disk in $Y$ is the desired submanifold.
\end{proof}
This allows us to characterize when a Seifert fibered space with $\varepsilon=0$ embeds in $S^4$ in terms of the existence of an embedding for a Seifert fibered space over $D^2$. This characterization shows that existence of an embedding is independent of the genus of the base surface. This is in contrast to the situation for spaces with $\varepsilon \neq 0$, where it is unknown how important the genus of the base surface is to the existence of an embedding into $S^4$.
\begin{prop}\label{prop:disk_to_eps=0}
The Seifert fibered space $Y=F(0;\frac{p_1}{q_1},-\frac{p_1}{q_1}, \dots, \frac{p_k}{q_k},-\frac{p_k}{q_k})$ over orientable base surface $F$ embeds smoothly in $S^4$ if and only if the Seifert fibered space $\wt{Y}=D^2(\frac{p_1}{q_1}, \dots, \frac{p_k}{q_k})$ smoothly embeds in $S^4$.
\end{prop}
\begin{proof}
By Lemma~\ref{lem:disk_subspace}, $Y$ contains $\wt{Y}$ as a submanifold, so an embedding of $Y$ gives an embedding of $\wt{Y}$. This proves the ``only if'' direction. In the opposite direction notice that the manifold $Y'=\wt{Y} \cup_\partial -\wt{Y}$ we obtain by doubling $\wt{Y}$ along its boundary is homeomorphic to $S^2(0;\frac{p_1}{q_1},-\frac{p_1}{q_1}, \dots, \frac{p_k}{q_k},-\frac{p_k}{q_k})$. If $\wt{Y}$ embeds in $S^4$ then it has a tubular neighbourhood $\wt{Y}\times [0,1] \subseteq S^4$. The boundary of this tubular neighbour is homeomorphic to $Y'\cong S^2(0;\frac{p_1}{q_1},-\frac{p_1}{q_1}, \dots, \frac{p_k}{q_k},-\frac{p_k}{q_k})$. By applying Proposition~\ref{prop:genus_bump} to raise the genus of the base surface if necessary, this shows that $Y$ embeds smoothly in $S^4$.  
\end{proof}

We also extend the result of Crisp-Hillman described above.
\begin{prop} Let $Y = S^2(0; \frac{p_1}{q_1}, -\frac{p_1}{q_1}, \ldots, \frac{p_k}{q_k}, -\frac{p_k}{q_k})$ where $p_i$ is even for at most one $i$. Then $Y$ smoothly embeds in $S^4$.
\end{prop}
\begin{proof} If precisely one of the $p_i$ is even, then let $Y' = S^2(0; \frac{p_1}{q_1}, \ldots, \frac{p_k}{q_k})$. If all the $p_i$ are odd, then define $Y'$ by
\[
Y'=
\begin{cases}
S^2(0; \frac{p_1}{q_1}, \ldots, \frac{p_k}{q_k}) &\text{if $q_1+\dots+ q_k\equiv 1 \bmod 2$}\\
S^2(1; \frac{p_1}{q_1}, \ldots, \frac{p_k}{q_k}) &\text{if $q_1+\dots+ q_k\equiv 0 \bmod 2$}
\end{cases}
\]
These are chosen to ensure that $|H_1(Y')|$ is odd. Therefore $Y'$ is the double branched cover of a Montesinos knot $K$, and Zeeman's twist-spinning theorem \cite{MR0195085} implies that $Y'\setminus \{pt\}$ smoothly embeds in $S^4$ as a fiber of the complement of the $2$-twist spin of $K$. However Lemma~\ref{lem:disk_subspace} shows that $Y'\setminus \{pt\}$ contains a submanifold homeomorphic to $D^2(\frac{p_1}{q_1}, \ldots, \frac{p_k}{q_k})$. Therefore $Y$ embeds in $S^4$ by Proposition~\ref{prop:disk_to_eps=0}.
\end{proof}
Further variations on these ideas are also possible.
\begin{eg} There is a smooth embedding of $S^2(0;4,-4,\frac{12}{5}, -\frac{12}{5})$ into $S^4$. In \cite[Example~2.14]{MR3271270}, Donald showed that $S^2(1;4,4,\frac{12}{5})$ embeds smoothly in $S^4$. This contains a $D^2(4, \frac{12}{5})$ submanifold, giving an embedding of $S^2(0;4,-4,\frac{12}{5}, -\frac{12}{5})$.
\end{eg}

\section{The Neumann-Siebenmann invariant}\label{sec:mubar}
In this section, we apply the $\mubar$ invariant to the question of when a Seifert fibered space can embed smoothly into $S^4$. The main result of this section is Proposition~\ref{prop:even_conditions}, which allows us to add further conditions to partitions arising from Theorem~\ref{thm:partitions} when there is an exceptional fiber of even multiplicity. This allows us to prove Theorem~\ref{thm:lower_bound} and Proposition~\ref{prop:k2_even_cond}. Throughout this section let $Y=S^2(e;\frac{p_1}{q_1}, \dots,\frac{p_k}{q_k})$ be a Seifert fibered space with $\varepsilon(Y)>0$ and $\frac{p_i}{q_i}>1$ for all $i$. Let $\Gamma$ be the canonical plumbing graph corresponding to $Y$ with vertex set $V$ and $X$ the positive definite manifold obtained by plumbing according to $\Gamma$.

We say that a subset $C\subseteq V$ is {\em characteristic} if $x=\sum_{v\in C} v$ is characteristic when considered as a vector in the intersection lattice $(\Z^{|\Gamma|}, Q_{\Gamma})$. Recall that a vector $x$ in an integer lattice is characteristic if
\[
x\cdot z \equiv z\cdot z \bmod{2}
\]
for all $z$ in the lattice. It is well known that there is a bijective correspondence between characteristic subsets of $\Gamma$ and $\spin(Y)$ \cite[Proposition 5.7.11]{MR1707327}\footnote{This correspondence is much more general than we are using here: it applies whenever we have a 3-manifold with a given surgery presentation. It is usually described in terms of characteristic sublinks of a surgery diagram.}.

The following definition of the $\mubar$ invariant is due to Neumann \cite{MR585657}. Siebenmann also gave an equivalent definition in \cite{MR585660}.

\begin{mydef}
Given a spin structure $\spincs$ on $Y$, the Neumann-Siebenmann invariant $\mubar(Y,\spincs)$ is defined as
\[
\mubar(Y,\spincs)=|\Gamma| - \norm{w},
\]
where $w=\sum_{v \in C} v$ and $C$ is the characteristic subset corresponding to $\spincs$ and $|\Gamma|=|V|$ is the number of vertices in $\Gamma$.
\end{mydef}
\begin{remark} Some comments on this definition are in order: 
\begin{enumerate}
\item We have chosen to define $\mubar$ in terms of the positive definite plumbing. There is a more general definition that allows $\mubar$ to be calculated from any plumbing cobounding $Y$.

\item It is not hard to see that any characteristic subset of $C \subset V$ must consist of isolated vertices,\footnote{The characteristic condition implies that any vertex in a characteristic set must have an even number of neighbours in the set. Since $\Gamma$ is a tree this forces the subset to be isolated.} that is, no pair of adjacent vertices are both in $C$.
   So we can equivalently define
\[
\mubar(Y,\spincs)=|\Gamma| - \sum_{v \in C} \norm{v}.
\]
\end{enumerate}
\end{remark}

It is known that for Seifert fibered spaces over $S^2$, $\mubar$ is a spin rational homology cobordism invariant \cite{MR2178795} and that $\mubar(Y,\spincs)=0$ whenever $(Y,\spincs)$ is the boundary of a spin rational homology ball.

 In order to apply $\mubar$ effectively we need to understand which characteristic subsets correspond to spin structures which extend over a given cobounding spin rational homology ball. We can do this by studying lattice embeddings.

\begin{prop}\label{prop:extending_condition}
Suppose that $Y$ bounds a smooth spin rational homology $4$-ball $W$ with $H^3(W;\Z) = 0$. The inclusion map $X \xhookrightarrow{} X \cup -W$ induces a map on second homology, which we identify with $\iota : (\Z^{|\Gamma|}, Q_\Gamma) \rightarrow (\Z^{|\Gamma|}, \mbox{Id})$. Let $e_1,\dots, e_{|\Gamma|}$ be an orthonormal basis for $(\Z^{|\Gamma|}, \mbox{Id})$. Let $\mathfrak{s}$ be a spin structure on $Y$ with corresponding characteristic subset $C \subset V$. Then $\mathfrak{s}$ extends over $W$ if and only if $\sum_{v\in C} \iota(v)$ is characteristic in $\Z^{|\Gamma|}$, that is
\[\sum_{v\in C} \iota(v)\cdot e_i\equiv 1 \bmod 2\]
for all basis elements $e_i$.
\end{prop}
\begin{proof} Let $Z = X \cup -W$. Since $H^3(W; \Z) = 0$ and $H_1(X; \Z) = 0$, the Mayer-Vietoris sequence and Poincar{\'e}-Lefschetz duality imply that $H_1(Z; \Z) = 0$, and thus $H_2(Z; \Z)$ is torsion free. Hence, $H_2(Z; \Z) \cong \Z^{|\Gamma|}$. Since $Z$ is positive definite, Donaldson's theorem implies that $(H_2(Z; \Z), Q_Z) \cong (\Z^{|\Gamma|}, \mbox{Id})$. 

  Let $F \subset X$ be a closed connected oriented surface, such that $[F] \in H_2(X; \Z)$ represents $\sum_{v\in C} v \in (\Z^{|\Gamma|}, Q_\Gamma) \cong H_2(X; \Z)$. Then $F$ is the obstruction to extending $\mathfrak{s}$ over $X$, that is, $\mathfrak{s}$ extends to a spin structure $\mathfrak{s}_X$ on $X\backslash F$ which does not extend across $F$.

  Suppose that $\mathfrak{s}$ extends to a spin structure $\mathfrak{s}_W$ on $W$. Then gluing the spin structures $\mathfrak{s}_W$ and $\mathfrak{s}_X$ along $Y$ gives a spin structure $\mathfrak{s}_Z$ on $Z\backslash F$ which does not extend across $F$. Thus, the mod $2$ reduction of $[F] \in H_2(Z; \Z)$ is Poincar\'{e} dual to the second Stiefel-Whitney class $w_2(Z) \in H^2(Z;\Z_2)$. However, the Wu formula states that $PD(w_2(Z)) \in H_2(Z; \Z_2)$ is the unique element satisfying $PD(w_2(Z)) \cdot x = x \cdot x$ for all $x \in H_2(Z; \Z_2)$. Thus, we see that $PD(w_2(Z))$ is the mod $2$ reduction of a characteristic element of $H_2(Z;\Z)$. This implies that $\sum_{v\in C} \iota(v)\cdot e_i \equiv 1 \bmod 2$, as required.

  Conversely, suppose that $\sum_{v\in C} \iota(v)\cdot e_i \equiv 1 \bmod 2$ for all $e_i$. This shows that $\sum_{v\in C} \iota(v)$ reduced mod 2 is Poincar\'{e} dual to $w_2(Z)$. Then $Z\backslash F$ admits a spin structure $\mathfrak{s}_Z$. The bijection between characteristic sublinks and spin structures on $Y$ then shows that $\mathfrak{s}_Z$ restricts to $\mathfrak{s}$ on $Y$. Restricting $\mathfrak{s}_Z$ to $W \subset Z$ then shows that $\mathfrak{s}$ extends to a spin structure on $W$.
\end{proof}


This allows us to obtain further restrictions on the image of the characteristic subsets corresponding to spin structures that extend over a homology ball.

\begin{prop}\label{prop:mubar_embed_conditions}
Suppose that $Y$ bounds a spin rational homology ball $W$ with $H^3(W;\Z) = 0$. Let $\iota : (\Z^{|\Gamma|}, Q_\Gamma) \rightarrow (\Z^{|\Gamma|}, \mbox{Id})$ be the lattice embedding induced by the inclusion $X \xhookrightarrow{} X \cup -W$. For any choice of orthonormal basis $\{e_i\}$, the following are true:
\begin{enumerate}
\item\label{it:unique_ei} Let $C$ be a characteristic subset corresponding to a spin structure which extends over $W$. Then for all $v \in C$, we have $|\iota(v)\cdot e_i|\leq 1$ for all $e_i$ and for each $e_i$ there is precisely one $v \in C$ with $|\iota(v)\cdot e_i|=1$.
\item\label{it:common_ei}
  For any $m \in \{1,\ldots,|\Gamma|\}$, there are at most two distinct vertices with the property that the image of each vertex under $\iota$ pairs non-trivially with $e_m$ and each vertex belongs to a characteristic subset corresponding to a spin structure that extends over $W$.
\end{enumerate} 
\end{prop}
\begin{proof}
We will abuse notation by identifying each vertex of $\Gamma$ with its image under $\iota$. If the spin structure corresponding to $C$ extends over $W$, then the corresponding $\mubar$ invariant vanishes. This implies that
\[
\sum_{v\in C}v = \sum_{v\in C} \sum_{i=1}^{|\Gamma|} (v\cdot e_i)^2 = |\Gamma|.
\]
By Proposition~\ref{prop:extending_condition}, we have $\sum_{v\in C}e_i\cdot v$ is odd for all $i$. Thus there is at least one vertex in $C$ satisfying $v\cdot e_i \neq 0$. However by the above equation we see that there is at most one such $v$ and it satisfies $|v\cdot e_i|=1$. This verifies \eqref{it:unique_ei}.

Now suppose that we have characteristic subsets $C_1, C_2$ and $C_3$ corresponding to spin structures that extend over $W$. Suppose that $v_1, v_2$ and $v_3$ are distinct vertices satisfying $v_i \cdot e_m \neq 0$ and $v_i\in C_i$ for $i\in\{1,2,3\}$. It follows from \eqref{it:unique_ei} that $v_i \in C_j$ if and only if $i=j$. Now define $C_4$ to be the set of vertices such that $v$ is in $C_4$ if and only it is contained in precisely one or three of $C_1, C_2$ or $C_3$. We have that $v_1,v_2$ and $v_3$ are all in $C_4$. It is easy to verify that not only is $C_4$ a characteristic subset, but that for any unit basis vector $e_i$, we have
\[\sum_{v\in C_4} v\cdot e_i \equiv \sum_{v\in C_1} v\cdot e_i +\sum_{v\in C_2} v\cdot e_i +\sum_{v\in C_3} v\cdot e_i \equiv 1\bmod 2.
\]
So by Proposition~\ref{prop:extending_condition} we see that $C_4$ also corresponds to a spin structure that extends over $W$. Thus by \eqref{it:unique_ei} we see that at most one of $v_1 \cdot e_m,v_2 \cdot e_m$ and $v_3 \cdot e_m$ can be non-zero, a contradiction. This proves \eqref{it:common_ei}.
\end{proof}

We now need to understand the characteristic subsets of $\Gamma$. When $p_i$ is even for at least one $i$, these are determined by choosing characteristic subsets on the linear chains corresponding to the fibers of $Y$. Thus we need to understand the characteristic subsets on linear chains first.
\begin{lemma}\label{lem:char_sets_linear_case}
Let $\Delta$ be the linear chain corresponding to $p/q=[a_1, \dots, a_l]^-$, where $a_j\geq 2$ for all $j$.
\begin{enumerate}
\item If $p$ is odd, then $\Delta$ has a unique characteristic subset.
\item If $p$ is even, then $\Delta$ has two characteristic subsets, where one contains the first vertex and the other does not.
\end{enumerate}
\end{lemma}
\begin{proof}
The characteristic subsets on $\Delta$ are in bijection with spin structures on the lens space $L(p,q)$. Thus there is precisely one if $p$ is odd and precisely two if $p$ is even. Now suppose that $p$ is even and we will justify the statement concerning the leading vertex. Consider the matrix
\[
M=\begin{pmatrix}
a_1 & -1 &  \\
-1  & \ddots & -1 \\
  & -1 & a_l
\end{pmatrix} \bmod{2}.
\]
We can think of a characteristic subset of $\Delta$ as a vector $w\in \Z_2^l$ such that 
\[Mw \equiv \begin{pmatrix}
a_1 \\
\vdots \\
a_l
\end{pmatrix}
\bmod{2}.
\]
Thus if $w$ and $w'$ are the vectors in $\Z_2^l$ corresponding to the two distinct characteristic subsets, then the vector $w-w'$ is a non-zero element of $\ker M$ mod two.
However, if 
$v=\begin{pmatrix}
v_1 \\
\vdots \\
v_l
\end{pmatrix}$ is a non-zero element of the kernel of $M$ mod two, then $v_1$ is non-zero. Otherwise, suppose that 
$v_1=\dots =v_{k-1}=0$ and $v_k\neq 0$ for some $k\leq l$, this would imply that the $(k-1)$-st row of $Mv$ is non-zero. Thus precisely one of the two characteristic subsets contains the first vertex.
\end{proof}
\begin{remark}\label{rem:odd_leading}
Although we will not need this fact, one can show that if $p$ is odd, then the unique characteristic subset on $\Delta$ contains the leading vertex if and only if $q$ is odd.
\end{remark}

This allows us to construct the characteristic subsets on $\Gamma$ when at least one $p_i$ is even.

\begin{lemma}\label{lem:char_set_construction}
Suppose that $p_i$ is even for at least one $i$. Then no characteristic subset of $\Gamma$ contains the central vertex and any characteristic subset on $\Gamma$ is uniquely determined by the set of the vertices adjacent to the central vertex it contains. In fact, it suffices to determine which of the leading vertices on arms corresponding to even $p_i$ it contains.
\end{lemma}
\begin{proof}
We prove this by constructing all characteristic subsets. Suppose that $N\geq 1$ of the $p_i$ are even. By Lemma~\ref{lem:order_mod2_cohom}, $Y$ admits $|H^1(Y;\Z_2)|=2^{N-1}$ spin structures. We may construct a characteristic subset $C$ as follows. For each arm of $\Gamma$ corresponding to $p_i/q_i$ with $p_i$ odd include the vertices corresponding to the unique characteristic subset on that linear chain. Suppose that $\alpha$ of these chains include the leading vertex. Now choose a subset $S$ of the arms corresponding to even $p_i$ such that $|S|\equiv \alpha +e \bmod 2$. For each arm in $S$ choose the characteristic subset containing its leading vertex. For all other arms choose the characteristic subset on the linear chain not containing the leading vertex. This defines a characteristic subset since it is characteristic on the arms by construction and does not contain the central vertex. Moreover, it is chosen so that it contains $|S| + \alpha \equiv e \bmod 2$ vertices adjacent to the central vertex. Notice however that of the set of $N$ arms corresponding to even $p_i$, there are $2^{N-1}$ even subsets and $2^{N-1}$ odd subsets. Thus we can construct all the characteristic subsets this way irrespective of the parity of $\alpha$.
\end{proof}

We can now add further conditions to the partitions in Theorem~\ref{thm:partitions}. The following proposition, although sufficient for our applications, is certainly not the most general statement that can be proven. For example, using Remark~\ref{rem:odd_leading}, one could also add further conditions relating to the parity of the $q_i$.
\begin{prop}\label{prop:even_conditions}
Let $Y = S^2(e; \frac{p_1}{q_1}, \ldots, \frac{p_k}{q_k})$ be a Seifert fibered space with $\varepsilon(Y) > 0$, $\frac{p_i}{q_i}>1$ for all $i$ and $p_j$ even for at least one $j$. Suppose that $Y$ smoothly embeds in $S^4$ and let $P$ be one of the partitions of $\{1,\dots, k\}$ given by Theorem~\ref{thm:partitions}. Then the following further conditions apply to $P$.
\begin{enumerate}
\item There is precisely one class containing an odd number of $i$ for which $p_i$ is even and there are one or three such $i$.
\item In all other classes there are zero or two values of $i$ such that $p_i$ is even.
\end{enumerate}
Moreover suppose that $C=\{1,\dots, l\}$ is a complementary class such that $p_{1}$ and $p_{2}$ are even and $p_i$ is odd for all $3 \le i \le l$, then
\begin{equation}\label{eq:even_fiber_bound}
\lceil p_1/q_1 \rceil \leq 1 + \sum_{i=2}^l (p_i-1).
\end{equation} 
\end{prop}
\begin{proof}
Recall that these partitions are constructed by taking the splitting $S^4=U_1 \cup_Y -U_2$ and $\iota_1 : (\Z^{|\Gamma|}, Q_\Gamma) \rightarrow (\Z^{|\Gamma|}, \mbox{Id})$ be the lattice embedding induced by the inclusion $X \xhookrightarrow{} X \cup -U_i$ for $i=1$ or $2$. Without loss of generality, we will work with $\iota=\iota_1$. We will abuse notation and identify each vertex of $\Gamma$ with its image under $\iota$. As shown in Lemma~\ref{lem:embedding_structure} we may assume that the central vertex is given by $\nu=e_1+\dots+e_e$ and for $i=1,\dots, e$ the class $C_i$ is taken to be the subset of $\{1,\dots, k\}$ such that the first vertex of the linear chain corresponding to $p_i/q_i$ pairs non-trivially with $e_i$.

Suppose that $Y$ has $N\geq 1$ exceptional fibers of even order, so that $\dim H^1(Y;\Z_2)=N-1$ by Lemma~\ref{lem:order_mod2_cohom}. Let $n_i$ be the number of fibers of even order in each class of the partition. Let $C$ be a characteristic set corresponding to a spin structure which extends over the ball $U_1$. By Proposition~\ref{prop:S4_splitting}\eqref{item:h3=0} we have $H^3(U_1;\Z) = 0$, so Proposition~\ref{prop:mubar_embed_conditions}\eqref{it:unique_ei} applies, implying that for each class there is precisely one arm from each class whose leading vertex is in $C$. Moreover Proposition~\ref{prop:mubar_embed_conditions}\eqref{it:common_ei} shows that for each class in the partition there are at most two choices for the arm whose leading vertex can appear in any such $C$. However since characteristic subsets all coincide on arms corresponding to odd $p_i$, we see that two choices for the leading vertex from arms in a class $C_i$ can only be realized if $n_i\geq 2$. Thus, if there are $m$ values of $n_i$ such that $n_i\geq 2$, then at most $2^{m}$ spin structures extend over $U_1$. However by Lemma~\ref{lem:spin_splitting}, we know that $2^{(N-1)/2}$ spin structures extend over $U_1$. This shows that
\[
2m \geq N-1 =n_1 + \dots + n_e -1.
\]
This shows that with exactly one exception $n_i\in \{0,2\}$ and for this exception we must have $n_i\in \{1,3\}$, which completes the count of even $p_i$ in each class.

Now we establish \eqref{eq:even_fiber_bound}. Suppose that we have the class $C_1=\{1,\dots, l\}$ is complementary with $p_1$ and $p_2$ even and all other $p_i$ is this class odd, that is $n_1=2$. The argument in the previous paragraph shows that the leading vertices of both the arms corresponding to $p_1/q_1$ and $p_2/q_2$ must appear in characteristic subsets corresponding to spin structures that extend over $U_1$. In particular if $v$ is the leading vertex of the arm corresponding to $p_1/q_1$, then $v$ satisfies $|v\cdot e_i|\leq 1$ for all $i$ by Proposition~\ref{prop:mubar_embed_conditions}\eqref{it:common_ei} and $\norm{v}=\lceil p_1/q_1 \rceil$ by definition. So to bound $\norm{v}$ above it suffices to bound above the number of basis elements $e_i$ for which  $|v\cdot e_i|\neq 0$. To do this notice that if $|v\cdot e_i|\neq 0$, then $w\cdot e_i \neq 0$ for some other vertex $w$ appearing in one of the other chains in the class $C_1$. Otherwise we could consider the vector $v'=v-(v\cdot e_i) e_i$ to obtain an embedding of linear chains with corresponding fractions $\lceil p_1/q_1 \rceil-1, p_2/q_2, \dots, p_l/q_l$. Since $\lceil p_1/q_1 \rceil-1 < p_1/q_1$, this would contradict Theorem~\ref{thm:emb_ineq_eq_case}. However, by inducting on the length of the continued fraction, one can see that an embedding of the linear chain corresponding to $r/s$ can use at most $r$ distinct orthonormal basis vectors. Thus we see that
\[
\lceil p_1/q_1 \rceil =\norm{v}\leq 1 + \sum_{i=2}^l (p_i-1),
\]
where $p_i-1$ terms come from observing that by definition all the linear chains in $C_1$ have at least one common basis element with which they pair non-trivially. This is the required upper bound. 
\end{proof}
We now have the tools to establish our lower bound on $e$.
\thmlowerbound*
\begin{proof}
First note that if $Y$ has no exceptional fibers of even order and $Y$ embeds in $S^4$, then $H^1(Y;\Z_2)=0$. So we may suppose that $Y$ has at least one exceptional fiber of even order. Proposition~\ref{prop:even_conditions} shows that there can be at most $2e+1=3+2(e-1)$ such fibers. Thus by Lemma~\ref{lem:order_mod2_cohom} we have $H^1(Y;\Z_2)\leq 2e$ in this case too.
\end{proof}

\begin{remark} Donald showed that $S^2(1; 4, 4, \frac{12}{5})$ smoothly embeds in $S^4$ \cite[Example~2.14]{MR3271270}. This Seifert fibered space and its expansions show that the bound in Theorem~\ref{thm:lower_bound} is sharp.
\end{remark}

We conclude with the following lemma which justifies Proposition~\ref{prop:k2_even_cond}. To see this, note that the Seifert fibered spaces in Theorem~\ref{thm:classification_e_ge_k2}\eqref{classif:case2} only arise when applying Theorem~\ref{thm:partitions} when there is a partition containing a complementary class of the form $\{\frac{p}{q}, \frac{r}{s}, rp\}$ (cf. Remark~\ref{rem:size_3_class}). The following lemma shows that $rp$ must be odd.
\begin{lemma}
If $Y = S^2(e; \frac{p_1}{q_1}, \ldots, \frac{p_k}{q_k})$ embeds smoothly into $S^4$, then neither of the partitions given by Theorem~\ref{thm:partitions} can contain a complementary class of the form $\{\frac{p}{q}, \frac{r}{s}, rp\}$ with $rp$ even.
\end{lemma}
\begin{proof}

Suppose that we had such a class. Since the class is complementary, we have $\frac{s}{r}+\frac{q}{p}+\frac{1}{rp}=1$. This implies that $p$ and $r$ are coprime so precisely one of $r$ or $p$ is even. Thus \eqref{eq:even_fiber_bound} from Proposition~\ref{prop:even_conditions} applies to show that $rp\leq r+p-1$.
This is easily seen to be impossible as $r,p>1$.
\end{proof}

\section{Doubly slice Montesinos links}\label{sec:dbly_slice}
In this section we turn our attention to doubly slice links. We prove that the Seifert fibered spaces over $S^2$ in Theorem~\ref{thm:classification_e_ge_(k+1)/2} and Theorem~\ref{thm:classification_e_ge_k2}\eqref{classif:case2} are double branched covers of Montesinos links. We also prove Theorem \ref{thm:oddpretzels} which provides a classification of the smoothly doubly slice odd pretzel knots up to mutation. Finally, we prove Proposition \ref{prop:qa_doubly_slice} showing that no non-trivial quasi-alternating Montesinos link is doubly slice.

\begin{prop}\label{prop:s4_embeddings} Let $Y$ be a Seifert fibered space over $S^2$, with $k > 2$ exceptional fibers, in either of the following two families:
  \begin{enumerate}[(a)]
  \item\label{enum:dbly_slice_1} $S^2\left(\frac{k+1}{2}; \frac{a}{a-1}, a, \ldots, \frac{a}{a-1} \right) = S^2(0; -a, a, \ldots, -a),$ where $a > 1$ is an integer, or
  \item\label{enum:dbly_slice_2} $S^2\left(\frac{k}{2}; \frac{p}{q}, \frac{p}{p-q}, \cdots, \frac{p}{q}, \frac{r}{s}, \frac{r}{r-s}, \cdots, \frac{r}{s}\right) = S^2\left(0; \frac{p}{q}, -\frac{p}{q}, \ldots, \frac{p}{q}, \frac{r}{s}, -\frac{r}{s}, \ldots, \frac{r}{s}\right)$ where $\frac{p}{q}, \frac{r}{s} > 1$ and $\frac{s}{r} + \frac{q}{p} = 1 - \frac{1}{pr}.$
  \end{enumerate}
  Then $Y$ is the double branched cover of a smoothly doubly slice Montesinos link. 
\end{prop}

As discussed in Remark~\ref{rmk:embs_in_lit}, some special cases of Proposition~\ref{prop:s4_embeddings} were previously known by work of Donald \cite{MR3271270}. We will use the following doubly slice criterion of his \cite[Corollary 2.5]{MR3271270} to prove Proposition~\ref{prop:s4_embeddings}.

\begin{thm}\label{thm:slice_criterion} Suppose $L$ is a link in $S^3$ and there are two sets of band moves $\{A_i\}_{1\le i \le k}$ and $\{B_j\}_{1 \le j \le l}$ such that performing the moves:
  \begin{itemize}
  \item $\{A_i\}_{1\le i \le k} \cup \{B_j\}_{1 \le j \le l}$ gives the unknot,
  \item $\{A_i\}_{1 \le i \le k} \cup \{B_j\}_{1 \le j \le l-n}$ gives an $(n+1)$-component unlink for all $n \in \{1,2,\ldots, l\}$,
  \item $\{A_i\}_{1 \le i \le k-n} \cup \{B_j\}_{1 \le j \le l}$ gives an $(n+1)$-components unlink for all $n \in \{1,2,\ldots, k\}$.
  \end{itemize}
    Then $L$ is smoothly doubly slice.
\end{thm}

The collection of band moves that we will use can be quite complicated when viewed in $(S^3, L)$. Instead, these band moves can be more naturally viewed as corresponding to certain $2$-handle attachments in the double branched cover of $(S^3, L)$. The following theorem of Montesinos will allow us to make this correspondence.

\begin{thm}[Theorem 3 of \cite{MR511423}]\label{thm:montesinos} Consider a handle representation $W^4 = H^0 \cup n H^2$ of a $4$-manifold with boundary given by attaching $n$ $2$-handles to the $4$-ball. If the $n$ $2$-handles are attached along a strongly invertible link in $S^3$, then $W$ is a $2$-fold cyclic covering space of $D^4$ branched over a $2$-manifold.
\end{thm}

Montesinos \cite{MR511423} describes how to obtain the branched surface in $D^4$ from the attaching link and involution. We now describe this construction in the case of interest to us. This is also described in \cite{MR2833583}, where Lecuona used similar ideas to show certain Montesinos knots are ribbon.

Suppose that the $2$-handles in Theorem~\ref{thm:montesinos} are attached along a framed link $L \subset S^3$, where the strong involution is a rotation by $\pi$ about an axis in $S^3$. Suppose furthermore that each component of $L$ is an unknot which is given by a trivial arc above and below the rotation axis, see left of Figure~\ref{fig:sfs_branched_surface}. The branch surface in Theorem~\ref{thm:montesinos} has a simple description as follows. Replace each arc below the rotation axis with a twisted band following the arc, with twisting such that the signed number of crossings in the band is equal to the framing of the link component containing the arc, see Figure~\ref{fig:sfs_branched_surface}. These bands are attached to a rectangular disc with an edge lying on the axis of rotation. The bands and rectangular disc form a surface in $S^3$. Pushing this surface into $D^4$ gives the branch surface in Theorem~\ref{thm:montesinos}.

Observe that if $L = L' \cup \{K\}$ as framed links then the branched surface $S$ for $L$ is obtained from the branched surface $S'$ for $L'$ by a band attachment. In particular, the link $\partial S$ is obtained from $\partial S'$ by a band or ribbon move. If $L$ is the integer surgery presentation of a Seifert fibered space $Y$ over $S^2$ coming from the plumbing graph, then the boundary of the branch surface $S$ is a Montesinos link.

\begin{eg}\label{eg:montesinos}
  \begin{figure}[h]
  \begin{overpic}[width=\textwidth]{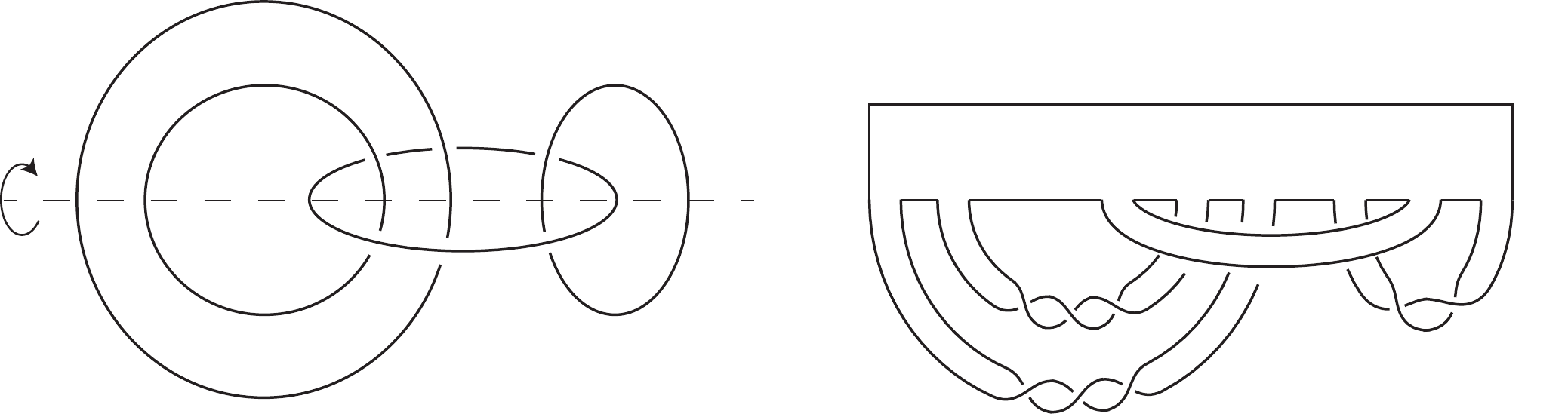}
    \put (4.5,3) {$-3$}
    \put (15,4) {$3$}
    \put (30,8) {$0$}
    \put (38,4) {$2$}
    \put (0.5,17) {$\pi$}
  \end{overpic}
  \caption{Left: Kirby diagram of a $4$-manifold with boundary $S^2(0; 3, -3, 2)$. Right: corresponding branch surface with boundary a Montesinos knot.}
  \label{fig:sfs_branched_surface}
\end{figure}
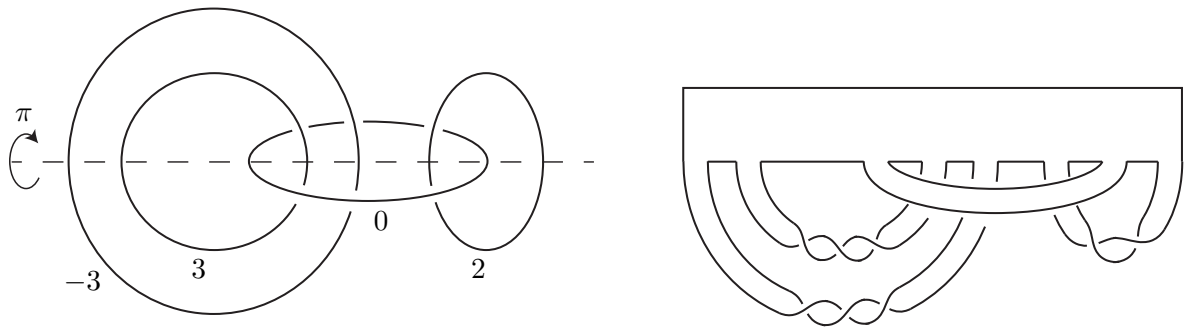
Consider the Seifert fibered space $Y = S^2(0; 3, -3, 2)$ with surgery presentation and strong involution as in Figure~\ref{fig:sfs_branched_surface}. Interpreting the surgery presentation as a Kirby diagram for the plumbing $4$-manifold $X$, we see that $X$ is the double branched cover of $(D^4, S)$, where $S$ is the surface in the right of Figure~\ref{fig:sfs_branched_surface} pushed into the $4$-ball. The knot $\partial S \subset S^3$ is the Montesinos knot with double branched cover $Y$.

Attaching an additional $2$-handle to $X$ which respects the strong involution, as shown in bold in the left of Figure~\ref{fig:sfs_branched_surface_attachment}, gives a $4$-manifold $X'$ which is the double branched cover of the surface $S'$ in the right of Figure~\ref{fig:sfs_branched_surface_attachment}. We see that $S'$ is obtained from $S$ by attaching a $2$-dimensional $1$-handle. Hence, the link $\partial S'$ is obtained from $\partial S'$ by a band, or ribbon move. One can check that $\partial X' = S^2 \times S^1$. Since the $2$-component unlink is the only link in $S^3$ with double branched cover $S^2 \times S^1$ \cite{MR584691}, we get that $\partial S'$ is the $2$-component unlink (one can also see this directly) and the Montesinos knot $\partial S$ is ribbon.

\begin{figure}[h]
  \begin{overpic}[width=\textwidth]{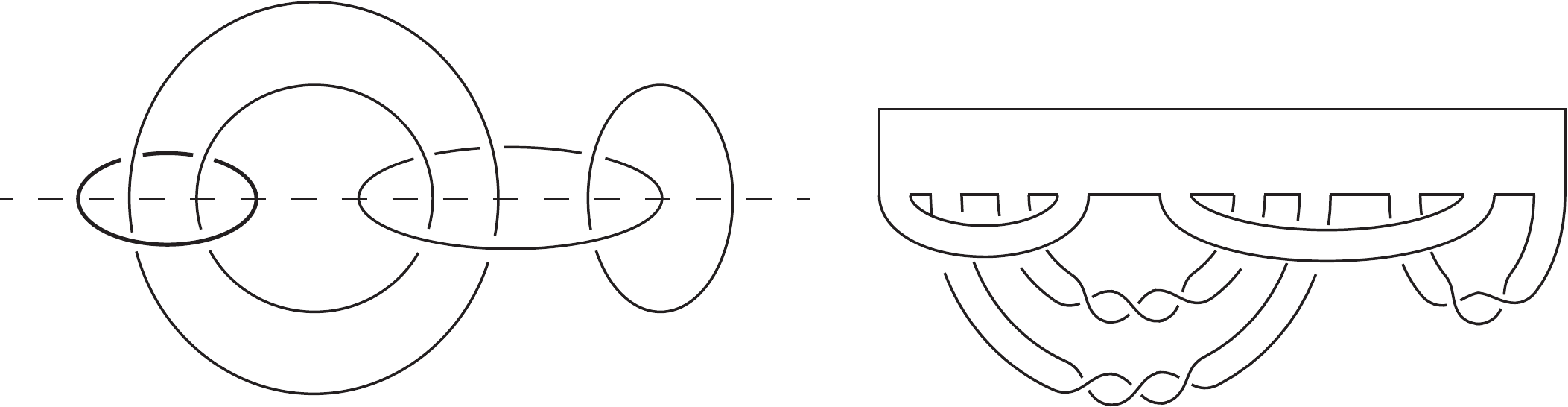}
    \put (5,9.5) {$0$}
    \put (7,3) {$-3$}
    \put (18,4) {$3$}
    \put (34,7.5) {$0$}
    \put (40,4) {$2$}
  \end{overpic}
  \caption{Left: The Kirby diagram with the extra $2$-handle. Right: The corresponding band in the link.}
  \label{fig:sfs_branched_surface_attachment}
\end{figure}
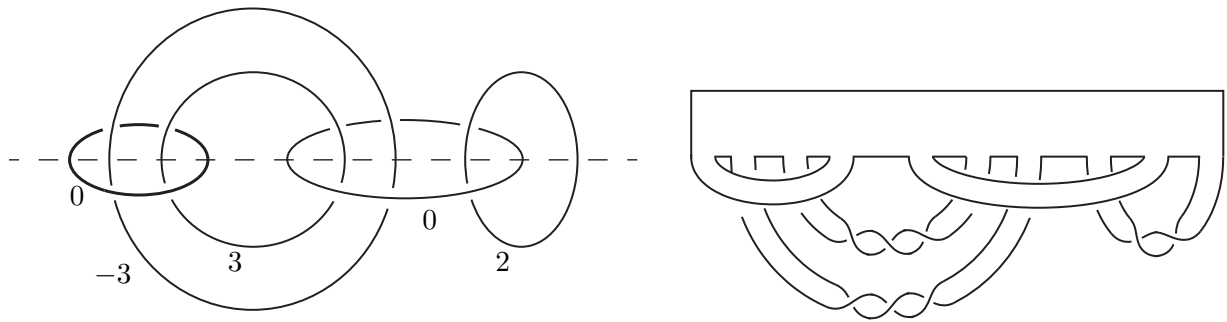
\end{eg}

We are now ready to prove Proposition~\ref{prop:s4_embeddings}.

\begin{proof}[Proof of Proposition~\ref{prop:s4_embeddings}\eqref{enum:dbly_slice_1}]
  Let $Y = S^2(0; -a, a, \ldots, -a)$ with $k$ fibers, where $k \ge 1$ is odd and $a > 2$ is an integer. If $k = 1$ then $Y$ is $S^3$ which is the double branched cover of the unknot which is trivially doubly slice. Assume that $k > 1$. Then $Y$ is the boundary of the $4$-manifold $X$ given by attaching $2$-handles to the $4$-ball, as shown in Figure~\ref{fig:doubly_slice_family1_involution} for $k=5$ (ignoring for now the $2$-handles with labels $A_1, A_2, B_1$ and $B_2$). The $2$-handles are attached along a strongly invertible link in Figure~\ref{fig:doubly_slice_family1_involution}, where the involution is given by a $\pi$ rotation about the dotted axis. Thus, Theorem~\ref{thm:montesinos} implies that $X$ is the double branched cover of $D^4$ over a properly embedded surface $S$ where $L = \partial S \subset S^3$ is the Montesinos link with double branched cover $\Sigma(L) = Y$. 

In Figure~\ref{fig:doubly_slice_family1_involution}, there are $2m := k-1$ ($k=5$ shown) additional $0$-framed $4$-dimensional $2$-handles, shown in bold, which are attached equivariantly with respect to the strong involution. By the discussion above Example \ref{eg:montesinos}, there are $2m := k-1$ disjoint bands $A_1, A_2, \ldots, A_m, B_1, B_2, \ldots, B_m$ defining band moves on $L$ such that doing any subset $S$ of these band moves changes $L \mapsto L'$ in such a way that $\Sigma(L') = \partial X_S$, where $X_S$ is the $4$-manifold given by attaching the correspondingly labeled subset of $0$-framed $2$-handles to $X$, as in Figure~\ref{fig:doubly_slice_family1_involution}, or by an isotopy, as in Figure~\ref{fig:doubly_slice_family1_surgery}.

\begin{figure}[h]
  \begin{overpic}[height=150pt]{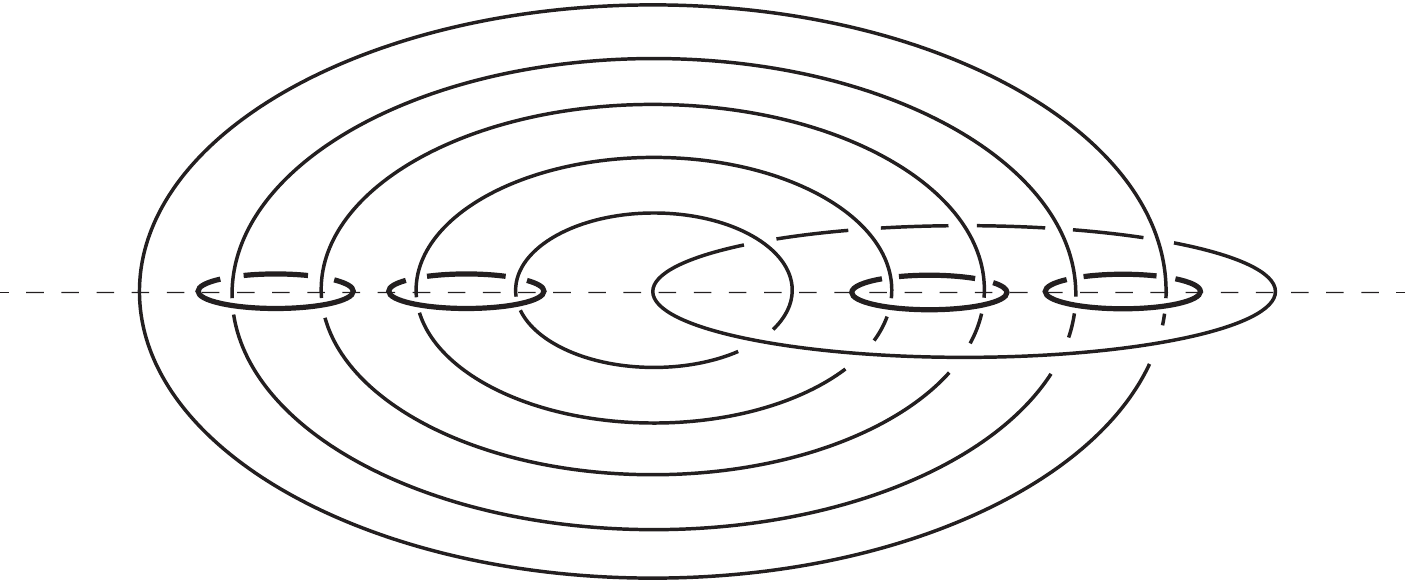}
    \put (10,31) {$-a$}
    \put (20,30) {$a$}
    \put (24, 29) {$-a$}
    \put (33,27.6) {$a$}
    \put (39, 26) {$-a$}
    \put (18.5, 16) {$A_2$}
    \put (33, 16) {$A_1$}
    \put (58, 17.2) {$B_1$}
    \put (72, 17.5) {\smaller $B_2$}
    \put (86, 15) {$0$}
  \end{overpic}
  \caption{Ignoring the curves in bold, $Y = S^2(0; -a, a, \ldots, -a)$ is the doubly branched cover of the link $L \subset S^3$ given by quotienting out by the involution given by rotating about the dotted axis. The case where $Y$ has $5$ exceptional fibers is shown.}
  \label{fig:doubly_slice_family1_involution}
\end{figure}
  
  \begin{figure}[h]
  \begin{overpic}[height=100pt]{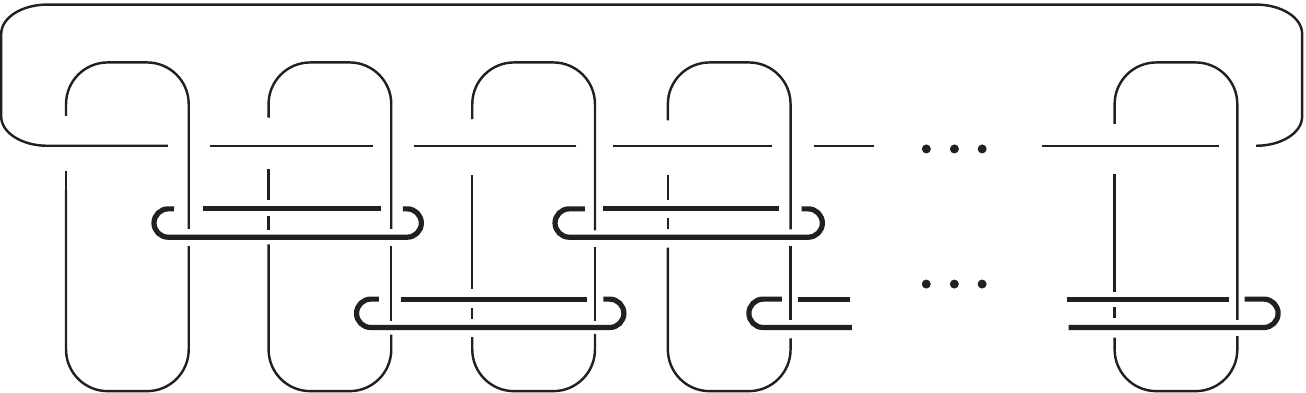}
    \put (1, 25) {$0$}
    \put (8, 23) {$-a$}
    \put (24, 23) {$a$}
    \put (38, 23) {$-a$}
    \put (55, 23) {$a$}
    \put (88, 23) {$-a$}
    \put (8, 12) {$A_1$}
    \put (38, 12) {$A_2$}
    \put (23, 4) {$B_m$}
    \put (52, 3) {$B_{m-1}$}
    \put (80, 2) {$B_1$}
  \end{overpic}
  \caption{Ignoring $2$-handles in bold, this is a Kirby diagram of $4$-manifold with boundary $S^2(0; -a, a, -a, \ldots, -a)$ containing $k = 2m+1$ fibers.}
  \label{fig:doubly_slice_family1_surgery}
\end{figure}

We now show that the two sets of bands $\{A_i\}_{1\le i \le m}$ and $\{B_i\}_{1\le j \le m}$ satisfy the doubly slice hypotheses of Theorem~\ref{thm:slice_criterion}, thereby showing that $L$ is doubly slice. First, let $S_n = \{A_i\}_{1 \le i \le m} \cup \{B_j\}_{1 \le j \le m-n}$, where $n \in \{0,1,2,\ldots, m\}$. We can realise $X_{S_n}$ as a union of linear plumbings, by handlesliding the central $0$-framed $2$-handle over each of the handles labeled $A_1, \ldots, A_m$ as shown in Figure~\ref{fig:doubly_slice_family1_slide}.

  We claim that $\partial X_{S_n}$ is a connected sum of $n$ copies of $S^1\times S^2$. Assuming the claim, by \cite{MR584691}, the $(n+1)$-component unlink is the unique link in $S^3$ with double branched cover $\#_{n} (S^1 \times S^2)$. This implies that performing band moves $S_n$ results in the $(n+1)$-component unlink, for all $n \in \{1,2,\ldots,m \}$. To show that $\partial X_{S_n} = \#_n (S^1 \times S^2)$, note that $\partial X_{S_n}$ consists of $n+1$ disjoint linear chains of unknots, where $n$ of these chains have length $3$ with components having framings (in linear order) $-a, 0, a$ giving an $S^1 \times S^2$ summand. Similarly, the remaining chain has framings $0, -a, 0, a, \ldots, -a, 0, a$ which represents $S^3$. Thus, $\partial X_{S_n} = \#_n (S^1 \times S^2)$.

  By symmetry we may interchange the roles of the $\{A_i\}$ and $\{B_i\}$ bands in the argument given above, which shows that the remaining hypothesis of Theorem~\ref{thm:slice_criterion} is satisfied, where band moves are performed on $S'_n = \{A_i\}_{1 \le i \le m-n} \cup \{B_j\}_{1 \le j \le m}$, for $n \in \{1,2,\ldots, m\}$.

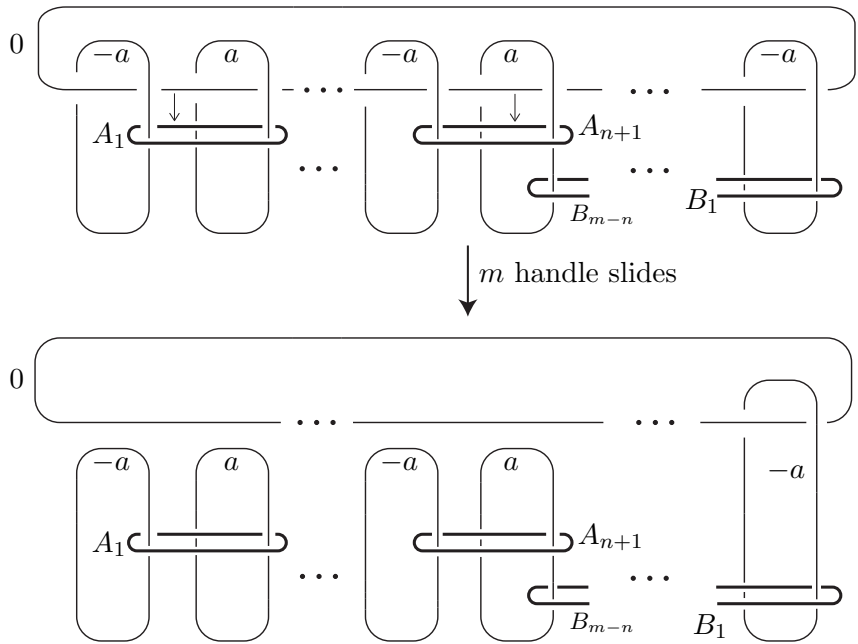
\begin{figure}[h]
  \begin{overpic}[height=240pt]{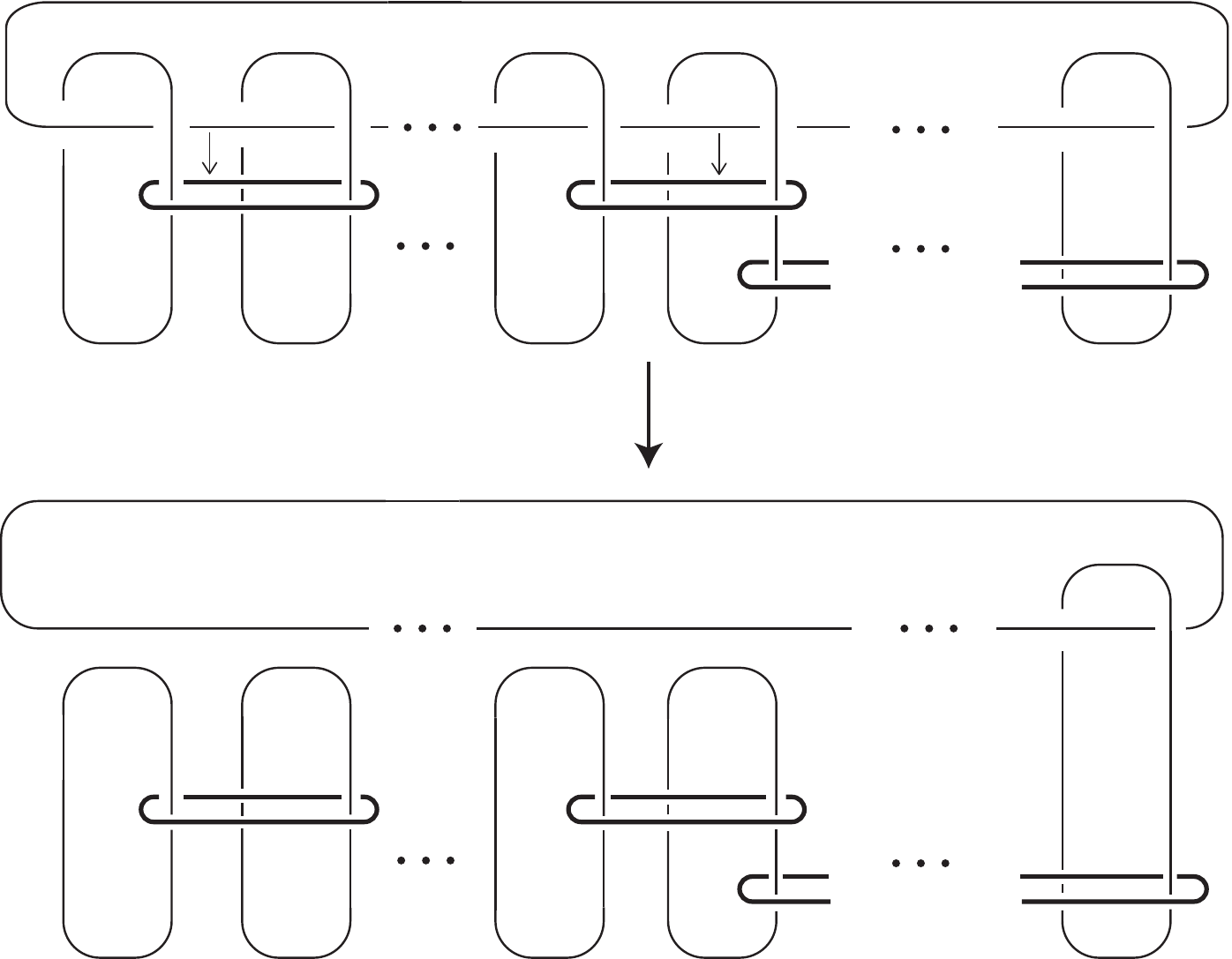}
    \put (-3, 72) {$0$}
    \put (7, 71) {$-a$}
    \put (23, 71) {$a$}
    \put (42, 71) {$-a$}
    \put (57, 71) {$a$}
    \put (88, 71) {$-a$}
    \put (7, 61) {$A_1$}
    \put (66, 62) {$A_{n+1}$}
    \put (65, 51.5) {\smaller $B_{m-n}$}
    \put (79, 53) {$B_1$}
    
    \put (-3, 31) {$0$}
    \put (7, 21) {$-a$}
    \put (23, 21) {$a$}
    \put (42, 21) {$-a$}
    \put (57, 21) {$a$}
    \put (89, 20) {$-a$}
    \put (7, 11) {$A_1$}
    \put (66, 12) {$A_{n+1}$}
    \put (65, 1.5) {\smaller $B_{m-n}$}
    \put (80, 1) {$B_1$}
    \put (54, 44) {$m$ handle slides}
  \end{overpic}
  \caption{The $4$-manifold $X_{S_n}$. Handleslide the $0$-framed central $2$-handle over each of the handles labelled $A_1,\ldots, A_m$.}
  \label{fig:doubly_slice_family1_slide}
\end{figure}
\end{proof}

\begin{proof}[Proof of Proposition~\ref{prop:s4_embeddings}\eqref{enum:dbly_slice_2}]
  Let $Y = S^2(0; \frac{p}{q}, -\frac{p}{q}, \ldots, \frac{p}{q}, \frac{r}{s}, -\frac{r}{s}, \ldots, \frac{r}{s})$ with $k$ fibers, where $k$ is even and $\frac{s}{r} + \frac{q}{p} = 1 - \frac{1}{pr}$. When $k = 2$, we have that $Y$ is a lens space with trivial first homology, so $Y =S^3$ and $Y$ is the doubly branched cover of the unknot, which is doubly slice. Assume that $k > 2$ and let $\ell$ be the number of fibers of the form $\pm \frac{p}{q}$ and $b = n-\ell$ be the number of fibers of the form $\pm \frac{r}{s}$. Observe that $\ell$ and $b$ are both odd. Let $[a_1, a_2, \ldots, a_g]^-$ (resp. $[b_1, b_2, \ldots, b_h]^-$)  be the continued fraction expansion for $\frac{p}{q}$ (resp. $\frac{r}{s}$).

We follow the same strategy as in the proof of Proposition~\ref{prop:s4_embeddings}\eqref{enum:dbly_slice_1} above to show that $Y$ is the double branched cover of a doubly slice link. The Seifert fibered space $Y$ is the boundary of a star-shaped plumbing $4$-manifold $X$ as shown in Figure~\ref{fig:doubly_slice_family2_surgery} (ignoring the $2$-handles in bold). By Theorem~\ref{thm:montesinos}, $X$ is the double branched cover of $(D^4, S)$ where $S$ is a surface. Then $Y = \partial X$ is the double branched cover of $S^3$ branched over the Montesinos link $L = \partial S$. There are bands $A_1,\ldots, A_m, B_1, \ldots, B_m$ which may be attached to $L$, where $m = \frac{k}{2} - 1$, such that performing a subset $S$ of these band moves changes $L \mapsto L'$ such that $\Sigma(L') = \partial X_S$, where $X_S$ is the $4$-manifold obtained by attaching the $0$-framed $2$-handles with labels in $S$ to $X$ in Figure~\ref{fig:doubly_slice_family2_surgery}. Figure~\ref{fig:doubly_slice_family2_involution}, obtained by an isotopy of the link in Figure~\ref{fig:doubly_slice_family2_surgery}, shows that the $2$-handles may be attached equivariantly with respect to the involution.

We check the hypotheses of Theorem~\ref{thm:slice_criterion}. First let $S_n = \{A_i\}_{1 \le i \le m} \cup \{B_j\}_{1 \le j \le m-n}$, where $n \in \{0,1,2,\ldots, m\}$. We can realise $X_{S_n}$ as a plumbing of a union of linear chains, by handle sliding the central $0$-framed handle over each of the handles labeled $A_1, \ldots, A_m$ in Figure~\ref{fig:doubly_slice_family2_surgery}. This union of linear chains consists of:

\begin{enumerate}
\item\label{enum:linear_chain1}  $n$ linear chains of one of two forms, either with framings $-a_1, -a_2, \ldots, -a_g, 0, a_g, \ldots, a_1$ or with framings $b_1, b_2, \ldots, b_h, 0, -b_h, \ldots, b_1$, and
\item\label{enum:linear_chain2} a linear chain with framings
  \begin{eqnarray*}
    a_g, \ldots, a_1, 0, -a_1, \ldots, -a_g, \ldots, 0, a_g, \ldots, a_1, 0, b_1, \ldots, b_h, 0,\\
    \ldots, -b_h, \ldots, -b_1, 0, b_1 \ldots, b_h.
  \end{eqnarray*}
\end{enumerate}

Each linear chain in \eqref{enum:linear_chain1} contributes an $S^1 \times S^2$ summand to $\partial X_{S_n}$, and the linear chain in \eqref{enum:linear_chain2} contributes an $S^3$ summand to $\partial X_{S_n}$. In order to see this, we repeatedly use the fact that a subchain with framings $r, 0, -r$ where $r \in \Z$, can be replaced by a single $0$ framed component. This fact follows by handlesliding the $r$ framed component over the $-r$ framed component, then cancelling the $-r$ framed component and its $0$ framed meridian. Repeatedly applying this fact, in case \eqref{enum:linear_chain2}, we will be left with a linear chain $a_g, \ldots, a_1, 0, b_1, \ldots, b_h$ representing the Seifert fibered space $S^2(0; \frac{p}{q}, \frac{r}{s})$ which is homeomorphic $S^3$ since the condition $\frac{s}{r} + \frac{q}{p} = 1 - \frac{1}{pr}$ implies that it is a lens space with trivial first homology. This verifies that $\partial X_{S_n} = \#_n (S^1 \times S^2)$.

\begin{figure}[h]
  \begin{overpic}[width=\textwidth]{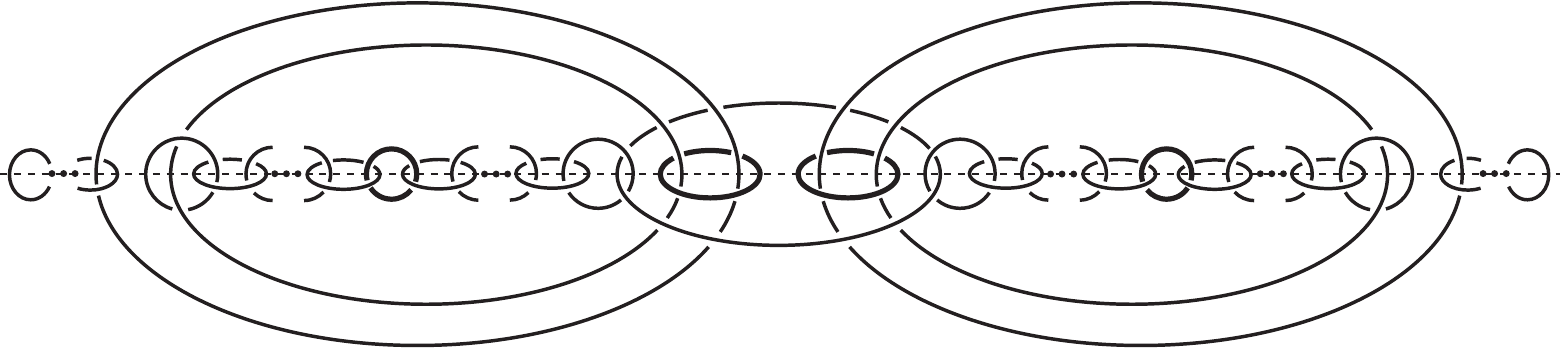}
    \put (37, 7.5) {\smaller $a_1$}
    \put (27, 8) {\smaller $a_g$}
    \put (23.8, 13.5) {\smaller $B_1$}
    \put (19.5, 8) {\smaller $-a_g$}
    \put (25, 4) {\smaller $-a_1$}
    \put (42.8, 18.5) {\smaller $a_1$}
    \put (49.3, 4.5) {\smaller $0$}
    \put (43.5, 8) {\tiny $A_1$}
    \put (53.5, 8) {\tiny $A_2$}
    \put (53, 18.5) {\smaller $-b_1$}
    \put (72, 4) {\smaller $b_1$}
    \put (59, 7) {\smaller $-b_1$}
    \put (69, 8) {\smaller $-b_h$}
    \put (73, 13.5) {\smaller $B_2$}
    \put (77, 8) {\smaller $b_h$}
    \put (96, 7.3) {\smaller $-b_h$}
    \put (1, 7.3) {\smaller $a_g$}
  \end{overpic}
  \caption{Kirby diagram for $X_{S'_n}$. Ignoring the components in bold gives a Kirby diagram for $X$ with boundary $Y$. For simplicity only the case with $k = 6$ and $\ell = 3$ is shown. The strong involution is rotation by $\pi$ about the dotted axis.}
  \label{fig:doubly_slice_family2_involution}
\end{figure}

\begin{figure}[h]
  \begin{overpic}[width=0.95\textwidth]{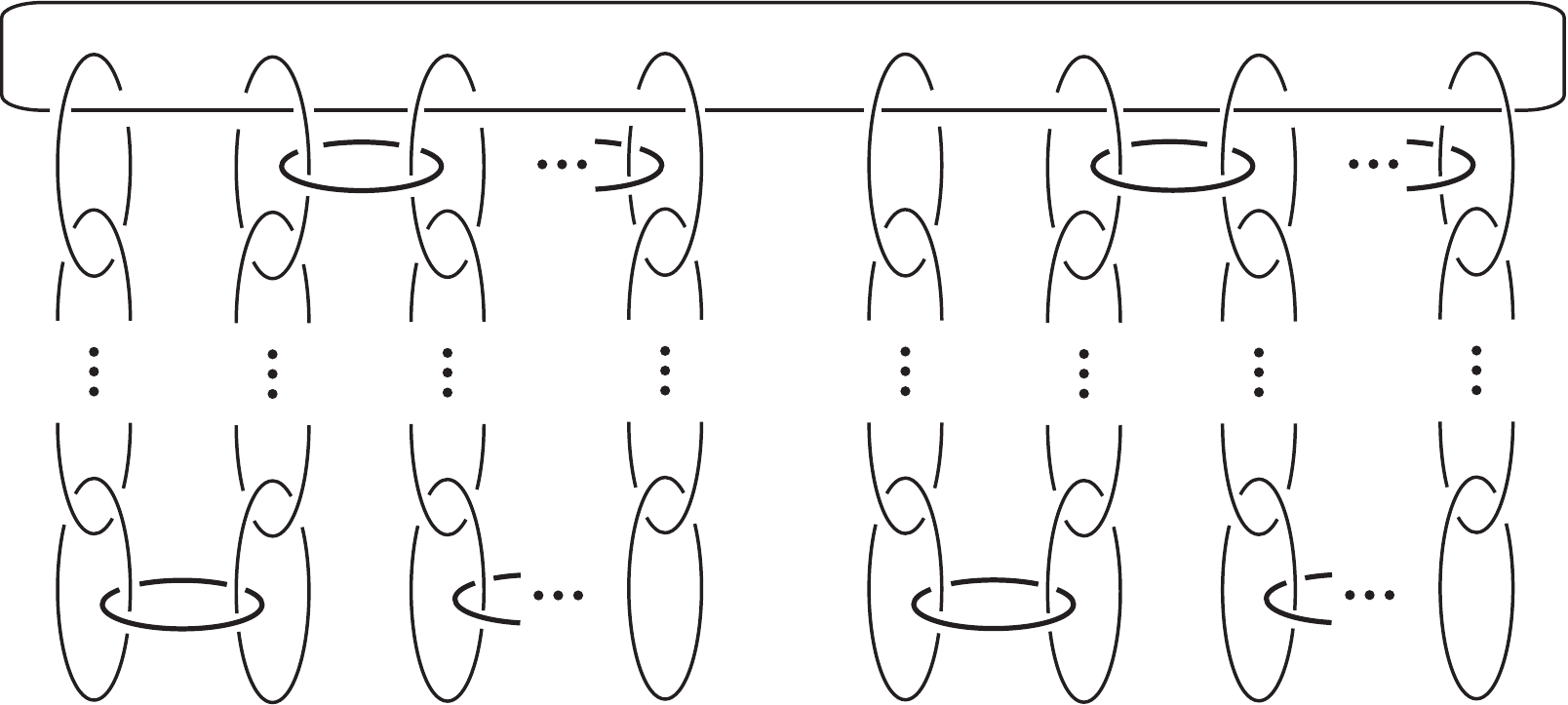}
    \put (-2, 41) {$0$}
    \put (0, 34) {$a_1$}
    \put (0, 25) {$a_2$}
    \put (-2, 16) {$a_{g-1}$}
    \put (0, 7) {$a_{g}$}

    \put (10, 34) {$-a_1$}
    \put (10, 25) {$-a_2$}
    \put (10, 18.3) {$-a_{g-1}$}
    \put (10, 10) {$-a_{g}$}

    \put (10,2) {$B_1$}
    \put (21.5, 30) {$A_1$}
    \put (32, 2) {$B_2$}
    \put (35, 31) {$A_{\frac{\ell-1}{2}}$}
    \put (61, 2) {$B_{\frac{\ell+1}{2}}$}
    \put (72, 30) {$A_{\frac{\ell+1}{2}}$}
    \put (84, 2) {$B_{\frac{\ell+3}{2}}$}
    \put (88, 30) {$A_{m}$}

    \put (31.4, 35.5) {$a_1$}
    \put (31.4, 25) {$a_2$}
    \put (31.4, 16) {$a_{g-1}$}
    \put (31.4, 10) {$a_{g}$}

    \put (45.5, 34) {$a_1$}
    \put (45.5, 25) {$a_2$}
    \put (45.5, 18) {$a_{g-1}$}
    \put (45.5, 7) {$a_{g}$}

    \put (53, 34) {$b_1$}
    \put (53, 25) {$b_2$}
    \put (51, 14) {$b_{h-1}$}
    \put (53, 7) {$b_{h}$}

    \put (62, 34) {$-b_1$}
    \put (62, 25) {$-b_2$}
    \put (62, 18) {$-b_{h-1}$}
    \put (62, 10) {$-b_{h}$}

    \put (83, 32) {$b_1$}
    \put (83, 25) {$b_2$}
    \put (83, 16) {$b_{h-1}$}
    \put (83, 10) {$b_{h}$}

    \put (97, 34) {$-b_1$}
    \put (97, 25) {$-b_2$}
    \put (97, 16) {$-b_{h-1}$}
    \put (97, 7) {$-b_{h}$}
  \end{overpic}
  \caption{Kirby diagram for $X_{S'n}$. Ignoring the components in bold gives a Kirby diagram for $X$ with boundary $Y$.}
  \label{fig:doubly_slice_family2_surgery}
\end{figure}
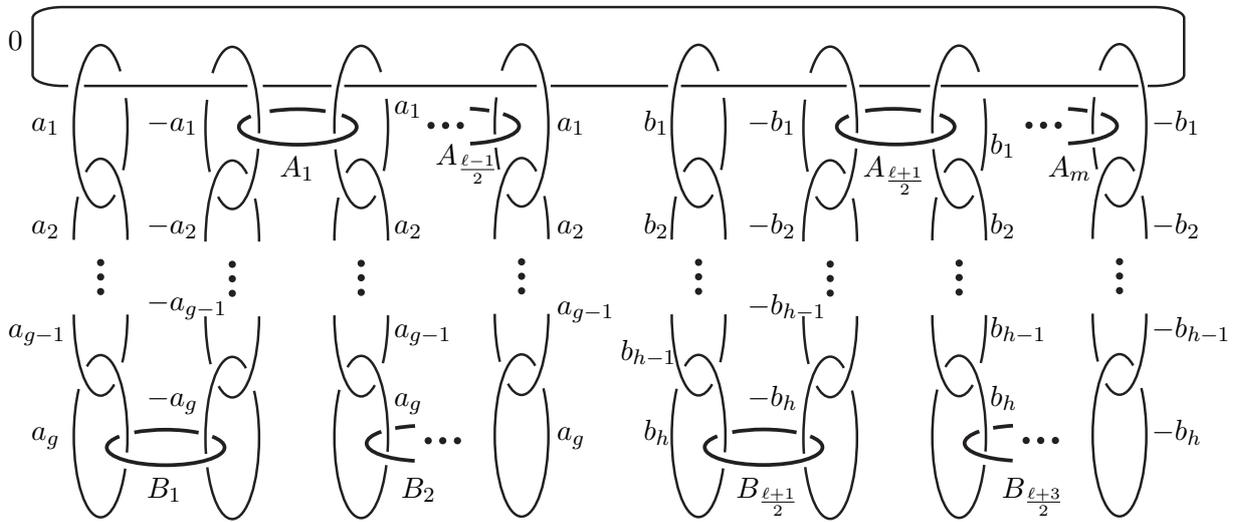

Now let $S'_n = \{A_i\}_{1 \le i \le m-n} \cup \{B_j\}_{1 \le j \le m}$, for $n \in \{1,2,\ldots, m\}$. Each $2$-handle attached to $X$ corresponding to a band of the form $B_j$ links two unknotted components with framings $-a_g$ and $a_g$. We use the same fact as above, that is, handlesliding the $a_g$ framed component over the $-a_g$ framed component leads to the $B_j$ labelled $2$-handle linking the $a_g$ framed component as a meridian, and hence we can cancel these two components without changing $\partial X_{S'_n}$. We see a $0$-framed unknot linking components with framings $-a_{g-1}$ and $-a_{g-1}$ and we can again handleslide the $a_{g-1}$ component over the $-a_{g-1}$ and remove the $-a_{g-1}$ framed components and its $0$-framed meridian. Repeating this procedure leads to the surgery presentation for $\partial X_{S'_n}$ shown in Figure~\ref{fig:doubly_slice_family2_step2}.

Next, we handleslide the $0$-framed central curve in Figure~\ref{fig:doubly_slice_family2_step2} over the $m$ $0$-framed components as indicated by the arrows (note that the handleslides here are thought of merely as a move on surgery presentations for $\partial X_{S'_n}$). This gives a presentation for $\partial X_{S'_n}$ from which, by an analogous computation to the previous case, one can check that $\partial X_{S'_n} = \#_n (S^1 \times S^2)$.

\begin{figure}[h]
  \begin{overpic}[width=0.95\textwidth]{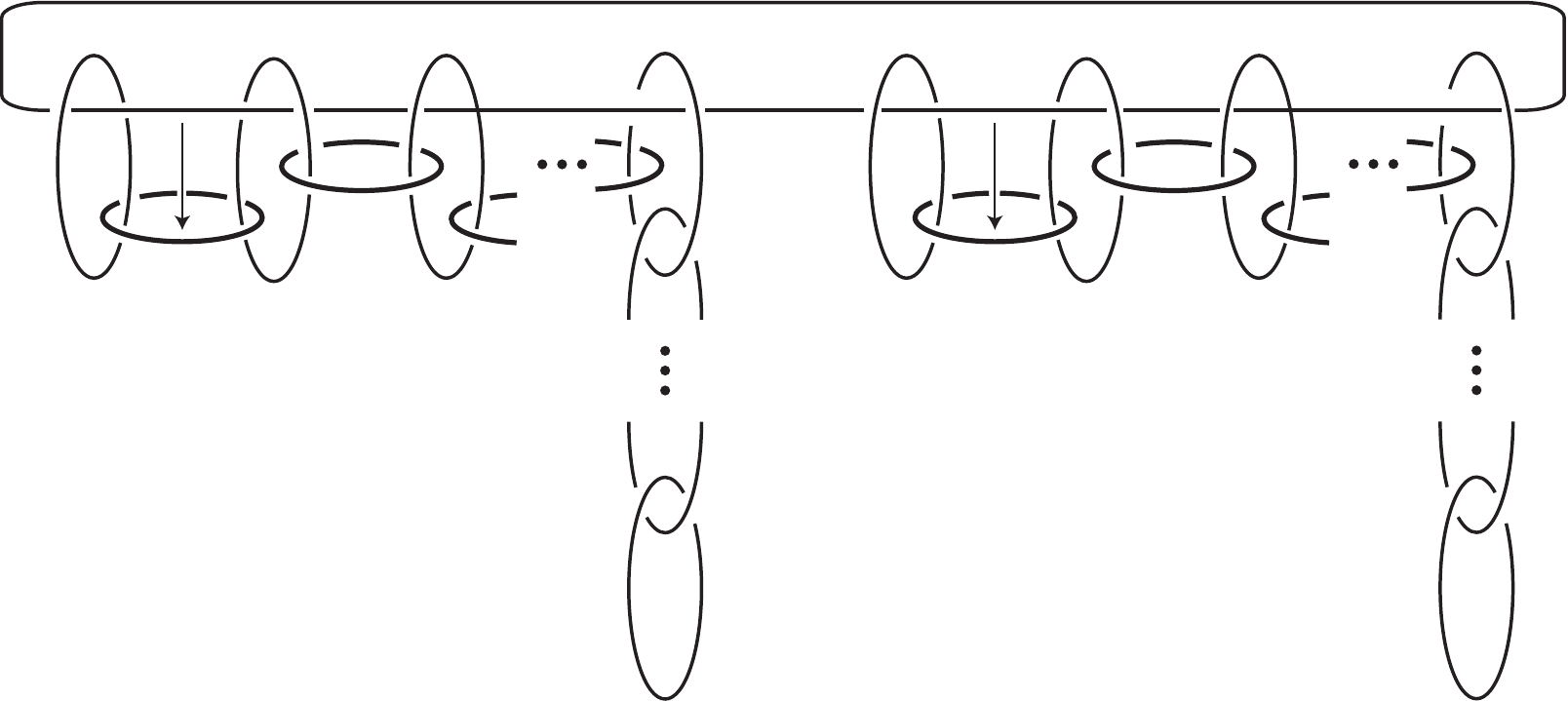}
    \put (-2, 41) {$0$}
    \put (5, 25) {$a_1$}

    \put (15, 25) {$-a_1$}

    \put (10,27) {$0$}
    \put (21.5, 30) {$A_1$}
    \put (32, 27) {$0$}
    \put (35, 31) {$A_{\frac{\ell-1}{2}}$}
    \put (61, 27) {$0$}
    \put (72, 30) {$A_{\frac{\ell+1}{2}}$}
    \put (84, 27) {$0$}
    \put (88, 30) {$A_{m}$}

    \put (27, 25) {$a_1$}

    \put (45, 34) {$a_1$}
    \put (45, 25) {$a_2$}
    \put (45, 16) {$a_{g-1}$}
    \put (45, 6) {$a_{g}$}

    \put (57, 24) {$b_1$}

    \put (67, 24) {$-b_1$}

    \put (79, 24) {$b_1$}

    \put (97, 34) {$-b_1$}
    \put (97, 25) {$-b_2$}
    \put (97, 16) {$-b_{h-1}$}
    \put (97, 7) {$-b_{h}$}
  \end{overpic}
  \caption{Surgery presentation for $X_{S'_m}$. In the next step, we handleslide the central $0$-framed component over each of the $m$ $0$-framed components indicated by the arrows. The general case $X_{S'_n}$, $1 \le n \le m$, is analogous.}
  \label{fig:doubly_slice_family2_step2}
\end{figure}
\end{proof}

This construction along with the obstructions from earlier in the paper allows us to prove the following theorem which classifies the smoothly doubly slice odd pretzel knots up to mutation. For $3$ or $4$-strand odd pretzel knots this was proved by Donald \cite[Theorem 1.5]{MR3271270}.
\newline

\begin{figure}[h]
  \begin{overpic}[width=0.9\textwidth]{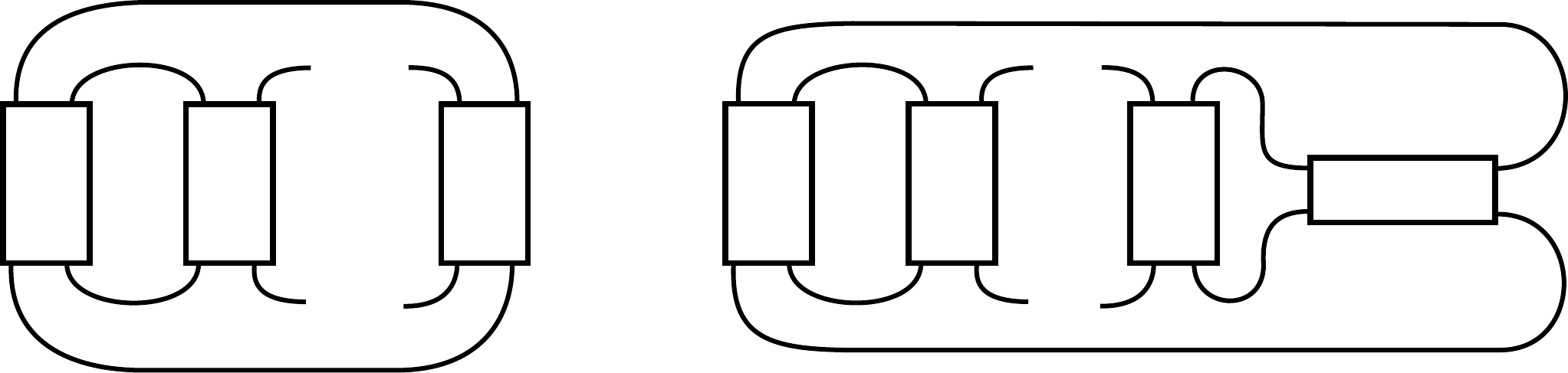}
  	\put (-5,20) {(a)}
    \put (2,12) {$c_1$}
    \put (13,12) {$c_2$}
    \put (30,12) {$c_m$}
    \put (22,12) {$\dots$}
    \put (40,20) {(b)}
    \put (48,12) {$c_1'$}
    \put (60,12) {$c_2'$}
    \put (66,12) {$\dots$}
    \put (73,12) {$c_{m'}'$}
    \put (88,11) {$e$}
  \end{overpic}
  \caption{Two diagrams for pretzel knots, where the labelled boxes are used to denote twist regions with the corresponding number of crossings. In the right hand side, we may assume $|c_i'|>1$ for all $i$.}
  \label{fig:pretzel}
\end{figure}

\thmoddpretzels*
\begin{proof}
  The implication \eqref{it:pretzel_list}$\Rightarrow$\eqref{it:pretzel_mutant} follows from the proof of Proposition~\ref{prop:s4_embeddings}. In order to see this, following Example~\ref{eg:montesinos}, one can check that the doubly slice pretzel knot corresponding to quotienting out Figure~\ref{fig:doubly_slice_family1_involution} by the strong involution indicated is precisely $P(-a,a,\ldots,-a)$. The implication \eqref{it:pretzel_mutant}$\Rightarrow$\eqref{it:pretzel_embed} is well-known. The content of this proof is in the implication \eqref{it:pretzel_embed}$\Rightarrow$\eqref{it:pretzel_list}, which we prove now.

Consider a pretzel knot $K=P(c_1, \dots, c_k)$ as depicted in Figure~\ref{fig:pretzel}(a), where the $c_i$ are all odd. Notice that if $|c_i|=1$, for some $i$, then the corresponding twist region is just a single crossing. By performing flypes and Reidemeister~II moves if necessary we can assume that these crossings are in a single twist region as in Figure~\ref{fig:pretzel}(b). That is, we can assume $K$ takes the form
\[
K=P(c_1',\dots, c_{m'}',\underbrace{\varepsilon, \dots, \varepsilon}_{|e|}),
\]
where $|c_i'|>1$ for all $i$ and $e=\varepsilon |e|$ is an integer. For such a $K$ double branched cover $\Sigma(K)$ takes the form
\[
\Sigma(K)= S^2(e;a_1,\dots, a_n,-b_1,\dots, -b_m).
\]
Assume, by reflecting $K$ if necessary, that $\varepsilon(\Sigma(K))>0$. So writing $\Sigma(K)$ in standard form we obtain,
\[
\Sigma(K)= S^2(m+e;a_1,\dots, a_n,\frac{b_1}{b_1-1},\dots, \frac{b_m}{b_m-1}).
\]
Now assume that $\Sigma(K)$ embeds smoothly in $S^4$. First consider a partition as given by Theorem~\ref{thm:partitions}. Note that since $\frac{b_i-1}{b_i}>\frac{1}{2}$ for all $i$, each class in the partition contains at most one of the fibers corresponding to $\frac{b_i}{b_i-1}$. This shows that there are at least $m$ such classes, implying that $e\geq 0$.

Now consider the condition that $\mubar(\Sigma(K))=0$. Consider the standard positive definite plumbing for $\Sigma(K)$. Since
$\frac{b_i}{b_i-1}$ has continued fraction
\[
\frac{b_i}{b_i-1}=[\underbrace{2, \dots, 2}_{b_i-1}]^-,
\]
 each of the arms corresponding to $\frac{b_i}{b_i-1}$ has $b_i-1$ vertices. Thus the plumbing has $1-m+n+\sum_{i=1}^m b_i$ vertices. Now it is easily checked that the (unique) characteristic subset on this plumbing is obtained by taking the central vertex along with $\frac{b_i-1}{2}$ vertices of norm two from each of the arms corresponding to $\frac{b_i}{b_i-1}$. Thus the sum of norms in the characteristic subset is $e+m+\sum_{i=1}^m (b_i-1)=e+\sum_{i=1}^m b_i$.
Thus we have
\[
\mubar(\Sigma(K))=n-m+1-e=0.
\]
Thus $e=n-m+1\geq 0$. However notice that $\Sigma(K)$ has $n+m$ exceptional fibers. Thus by Theorem~\ref{thm:classification_e_ge_(k+1)/2} we have $m+e\leq \frac{n+m+1}{2}$. Altogether this shows
\[
0\leq e \leq \frac{n-m+1}{2}=\frac{e}{2},
\]
which implies that $e=0$. Thus $\Sigma(K)$ has $n+m=2m-1$ exceptional fibers. Thus Theorem~\ref{thm:classification_e_ge_(k+1)/2} implies that $b_i=a_j>1$ for all $i$ and $j$. Thus $K$ is of the desired form.
\end{proof}
Finally we prove our results on doubly slice quasi-alternating Montesinos links.

\qaltMontesinos*
\begin{proof} Let $K$ be a quasi-alternating Montesinos link. The double branched covers of quasi-alternating Montesinos links have been classified \cite{issa2017quasialternating}. After possibly reflecting $K$, we can assume that
\[
\Sigma(K)=S^2(e;\frac{p_1}{q_1}, \dots, \frac{p_k}{q_k}),
\]
where $\varepsilon(\Sigma(K))>0$ and $\frac{p_i}{q_i}>1$ and either
\begin{enumerate}
\item $e\geq k$ or
\item $e=k-1$ and $\frac{q_{k-1}}{p_{k-1}} + \frac{q_{k}}{p_{k}}<1$
\end{enumerate}
holds.
However notice that in the first case we have a partition
\[\mathcal{P}=\{\{1\}, \dots, \{k\}\}\]
violating Lemma~\ref{lemma:homology_is_double2}, and in the second case we have a partition 
\[\mathcal{P}=\{\{1\}, \dots, \{k-2\},\{k-1,k\}\}\]
violating Lemma~\ref{lemma:homology_is_double2}. Thus in neither case can $H_1(\Sigma(K))$ split as a direct double. This shows that $\Sigma(K)$ cannot embed topologically locally flatly in $S^4$ and hence that $K$ is not topologically doubly slice.
\end{proof}
\phantomsection
\bibliography{references}{}
\bibliographystyle{alpha}

\end{document}